\documentclass[a4paper]{amsart}
\usepackage{amssymb}
\usepackage{amsmath}
\usepackage{todonotes}
\usepackage{esint}

\newtheorem{thm}{Theorem}[section]
\newtheorem{corollary}[thm]{Corollary}
\newtheorem{lemma}[thm]{Lemma}
\newtheorem{proposition}[thm]{Proposition}
\theoremstyle{definition}
\newtheorem{definition}[thm]{Definition}
\newtheorem{example}[thm]{Example}

\theoremstyle{remark}
\newtheorem{remark}[thm]{Remark}

\newcommand{\cL}{{\mathcal L}}

\newcommand{\cD}{{\mathcal D}}

\newcommand{\hook}{\lrcorner \,}
\newcommand{\re}{\mathrm{Re}}
\newcommand{\im}{\mathrm{Im}}
\newcommand{\sgn}{\mathrm{sgn}}

\newcommand{\G}{\mathrm{G}}

\newcommand{\bR}{{\mathbb R}}

\newcommand{\bC}{{\mathbb C}}

\newcommand{\so}{{{\mathfrak{so}}} }
\newcommand{\SU}{{\mathrm{SU}}}
\newcommand{\Spin}{{\mathrm{Spin}}}
\newcommand{\diag}{{\mathrm{diag}}}

\newcommand{\mfg}{\mathfrak{g}}
\newcommand{\mfh}{\mathfrak{h}}
\newcommand{\mfu}{\mathfrak{u}}

\newcommand{\mfr}{\mathfrak{r}}

\newcommand{\ad}{\operatorname{ad}}
\newcommand{\tr}{\operatorname{tr}}
\newcommand{\Ann}[1]{#1^0}

\newcommand{\spa}[1]{\mathrm{span}(#1)}

\newcommand{\id}{\operatorname{id}}

\newcommand{\SO}{\operatorname{SO}}
\newcommand{\GL}{\operatorname{GL}}
\newcommand{\SL}{\operatorname{SL}}

\newcommand{\artanh}{\operatorname{artanh}}
\newcommand{\arcoth}{\operatorname{arcoth}}
\numberwithin{equation}{section}

\begin{document}

\title{$\SU(4)$-holonomy via the left-invariant hypo and Hitchin flow}

\author{Marco Freibert}
\address{Mathematisches Seminar\\
Christian-Albrechts-Universit\"at zu Kiel\\
Ludewig-Meyn-Strasse 4\\
D-24098 Kiel\\
Germany}

\email{freibert@math.uni-kiel.de}
\thanks{Marco Freibert partially supported by Danish
Council for Independent Research \textbar\ Natural Sciences project
DFF - 4002-00125.}

\date{}


\begin{abstract}
The Hitchin flow constructs eight-dimensional Riemannian manifolds $(M,g)$ with holonomy in $\Spin(7)$ starting with a cocalibrated $\G_2$-structure on a seven-dimensional manifold.
As $\mathrm{Sp}(2)\subseteq \SU(4)\subseteq \Spin(7)$, one may also obtain Calabi-Yau fourfolds or hyperK\"ahler manifolds via the Hitchin flow.

In this paper, we show that the Hitchin flow on almost Abelian Lie algebras and on Lie algebras with one-dimensional commutator always yields
Riemannian metrics with $Hol(g)\subseteq \SU(4)$ but $Hol(g)\neq \mathrm{Sp}(2)$.
We investigate when we actually get $Hol(g)=\SU(4)$ and obtain so many new explicit examples of Calabi-Yau fourfolds.
The results rely on the connections between cocalibrated $\G_2$-structures and hypo $\SU(3)$-structures
and between the Hitchin and the hypo flow and on a systematic study of hypo $\SU(3)$-structures and the hypo flow on Lie algebras.
This study gives us many other interesting results: We obtain full classifications of hypo $\SU(3)$-structures with particular intrinsic torsion on Lie algebras.
Moreover, we can exclude reducible or $\mathrm{Sp}(2)$-holonomy or do get $Hol(g)=\SU(4)$ for the Riemannian manifolds obtained by the
hypo flow with initial values in some other intrinsic torsion classes.
\end{abstract}

\maketitle
\section{Introduction}
In Riemannian geometry, dimension eight is of particular interest for various reasons.
First of all, an eight-dimensional Riemannian manifold $(M,g)$ can have exceptional holonomy $\Spin(7)$.
However, there are even two other irreducible special holonomy groups from Berger's list \cite{Be} in this dimension,
namely $\SU(4)$, i.e. $(M,g)$ can be Calabi-Yau, or $\mathrm{Sp}(2)$, i.e. $(M,g)$ can be hyperK\"ahler.
In all three cases, $(M,g)$ is Ricci-flat and admits one or more parallel spinor fields \cite{W}. Both properties make these manifolds also very attractive for physicists and they occur as internal spaces in $M$- or $F$-theory compactifications, cf., e.g., \cite{ALRY}, \cite{V}.

A method to construct eight-dimensional Riemannian manifolds with holonomy in $\Spin(7)$ or $\SU(4)$ is by the so-called \emph{Hitchin} or \emph{hypo flow}, respectively. These flows are systems of partial differential equations for one-parameter families of cocalibrated $\G_2$- or hypo $\SU(3)$-structures on a seven-dimensional manifold $M$, whose solution on an interval $I$ then defines a Riemannian metric $g=g_t+dt^2$ on $M\times I$ with holonomy in $\Spin(7)$ or $\SU(4)$, respectively. Conversely, given a eight-dimensional Riemannian manifold $N$ with holonomy in $\Spin(7)$ or $\SU(4)$, any oriented hypersurface $M$ in $N$ has an induced cocalibrated $\G_2$- or hypo $\SU(3)$-structure, respectively, and the induced one-parameter families of these geometric structures on equidistant hypersurfaces in $N$ fulfill the Hitchin or the hypo flow, respectively.

The Hitchin flow has been introduced in \cite{Hi} and local existence and uniqueness of solutions has been shown for real-analytic initial data in \cite{CLSS}. Analogous results have been obtained in \cite{CS} for the hypo flow in dimension five and in \cite{C} and \cite{CF} for the hypo flow in arbitrary odd dimension greater than five. Note that for smooth initial data, solutions may not exist, cf. \cite{Br} and \cite{AMM}.

A particular interesting and more manageable case occurs when the initial data is homogeneous. Then the solution of both flows stays homogeneous for all times $t\in I$, the flow equations get non-linear systems of ordinary differential equations and the solutions define a cohomogeneity one metric on $M\times I$ with holonomy contained in $\Spin(7)$ or $\SU(4)$, respectively. In the homogeneous setting, the Hitchin flow has been solved explicitly for certain initial values, cf., e.g, \cite{R1}, \cite{R2}, \cite{DFISUV}. In some of these cases, the holonomy of the outcoming Riemannian manifold has been determined and it was found that there are examples with full holonomy $\Spin(7)$ but that in many cases the holonomy reduces further to $\SU(4)\subseteq \Spin(7)$.

Generally, the Hitchin flow cannot be solved explicitly even in the homogeneous setting and it is then of interest to find conditions on the initial value
which ensure such a holonomy reduction or full holonomy. In this paper, we provide such conditions for the left-invariant Hitchin flow on certain kinds of Lie groups.
More exactly, we show that any cocalibrated $\G_2$-structure $\varphi$ on an almost Abelian Lie algebra or a Lie algebra with one-dimensional commutator
is induced by a hypo $\SU(3)$-structure $(\alpha,\omega,\psi)$ and then use the general fact that then the Hitchin flow with initial value $\varphi$ is
induced by the hypo flow with initial value $(\alpha,\omega,\psi)$, so necessarily the Hitchin flow with initial value $\varphi$ yields a
Riemannian manifold with holonomy contained in $\SU(4)$. Moreover, we exclude the third possible irreducible holonomy group $\mathrm{Sp}(2)\subseteq \SU(4)$
in dimension eight in the two mentioned cases and obtain so that the Hitchin flow on the seven-dimensional Heisenberg algebra $\mfh_7$ always yields
$\SU(4)$-holonomy metrics. Also, we examine when the Hitchin flow on almost Abelian Lie algebras $\mfg$ leads to Riemannian manifolds with holonomy equal to $\SU(4)$
if there is a basis of $\mfg$ which stays orthogonal during the Hitchin flow.

The mentioned results rely on a proper investigation of hypo $\SU(3)$-structures and of the hypo flow on seven-dimensional Lie algebras in Sections \ref{sec:hypoSU3} and \ref{sec:leftinvhypoflow}, leading also to many other interesting results for hypo $\SU(3)$-structures and the hypo flow.

More exactly, after an introduction to the $G$-structures occuring in this paper and to the two mentioned flows in Section \ref{sec:Hitchinhypoflow},
we compute the intrinsic torsion of (hypo) $\SU(3)$-structures in Subsection \ref{subsec:inttors}.
In Subsection \ref{subsec:torsionclasses}, we consider different torsion classes of hypo $\SU(3)$-structures and
obtain some general results for arbitrary manifolds $M$. On Lie algebras, we classify hypo $\SU(3)$-structures for which $\ker(\omega)$
is an ideal in terms of six-dimensional K\"ahler Lie algebras with additional data, classify all hypo $\SU(3)$-structures with $d\alpha=0$,
show that hypo $\SU(3)$-structures with invariant intrinsic torsion $(\lambda_1,\lambda_2)\in \bR^2$ with $\lambda_1\lambda_2<0$ cannot exist
and give explicit examples with $\lambda_1\lambda_2>0$ in Example \ref{ex:invinttors}.
Note that in \cite{CF}, hypo $\SU(3)$-structures on Lie algebras of the form $\bR^4\rtimes \mathfrak{h}$ with a four-dimensional solvable Lie algebra
fulfilling $d\alpha=-2\omega$ are classified, including some examples with invariant intrinsic torsion with $\lambda_1\lambda_2>0$.
However, our examples with invariant intrinsic torsion are new and not contained in \cite{CF}. 

In Section \ref{sec:leftinvhypoflow}, we look at the hypo flow on Lie algebras. We first show that the hypo flow preserves various intrinsic torsion classes and reduces to simpler flow equations in these intrinsic torsion classes. Our first main result is Theorem \ref{th:irreduciblehol}, which states that for initial values with $d\psi=i\lambda_2 \psi$ and $(d\alpha)^3\neq 0$, the hypo flow always yields Riemannian manifolds with irreducible holonomy $\mathrm{Sp}(2)$ or $\SU(4)$. The proof uses the mentioned simpler form of the flow equations to show that there is a basis which stays orthogonal during the flow, which then allows to prove that parts of the Riemannian curvature tensor do not vanish, implying good enough lower bounds on the dimension of the holonomy group by Ambrose-Singer to conclude the result. Afterwards, we investigate the possibility of holonomy equal to $\mathrm{Sp}(2)$, using representation theoretic arguments to show that then the initial hypo $\SU(3)$-structure is induced by a hypo $\mathrm{Sp}(1)$-structure and that the hypo flow comes, in fact, from a flow of hypo $\mathrm{Sp}(1)$-structures. But the existence of such an $\mathrm{Sp}(1)$-structure inducing a hypo $\SU(3)$-structure of particular intrinsic torsion implies various constraints on the Lie algebra.
These constraints together with the flow equations for the hypo $\mathrm{Sp}(1)$-structures allow us exclude holonomy equal to $\mathrm{Sp}(2)$
for the Riemannian manifold obtained by the hypo flow for initial values with $d\alpha=0$
and allow us to show that for initial values with invariant intrinsic torsion and $\lambda_1\neq 0$ we always get holonomy equal to $\SU(4)$.
Moreover, we give new explicit examples of such holonomy $\SU(4)$-metrics in Example \ref{ex:intrtorsholSU4}.

Finally, in Section \ref{sec:leftinvHitchinflow}, we put our results together in Subsection \ref{subsec:holredHitchinflow} and prove the mentioned reduction results for the holonomy of the Riemannian manifolds obtained by the Hitchin flow on almost Abelian Lie algebras, on Lie algebras with one-dimensional commutator and on $\mfh_7$. We also give an explicit example of a holonomy $\SU(4)$-metric obtained by the Hitchin flow on $\mfh_7$ in Example \ref{ex:holSU4h7}. In the final Subsection \ref{subsec:diagonalHitchin}, we determine all cocalibrated $\G_2$-structures on almost Abelian Lie algebras for which there exists a particular type of basis which stays orthogonal during the Hitchin flow and for which the outcoming Riemannian manifold has holonomy equal to $\SU(4)$. This leads then to many new explicit examples of Riemannian manifolds with holonomy equal to $\SU(4)$.

\section{The Hitchin and the hypo flow}\label{sec:Hitchinhypoflow}

In this section, we define the different kinds of $G$-structures appearing in this article and introduce the Hitchin and the hypo flow. All these $G$-structures will be defined by a collection of differential forms whose common \emph{model forms} on $\bR^n$ have stabilizer $G$. As $\SU(3)$ will appear both as a subgroup of $\SO(6)$ as well as one of $\SO(7)$, we stick to the notation given in \cite{MC} and call $\SU(n)$-structures in even dimensions \emph{special almost Hermitian structures}. More information on the discussed $G$-structures and flows as well as proofs of the mentioned properties can be found, e.g., in \cite{CF}, \cite{CLSS}, \cite{Hi}, \cite{MC} and \cite{SH}.

We start by recalling the concept of \emph{model forms}:
\begin{definition}\label{def:modelforms}
Let $M$ be an $n$-dimensional manifold and $\rho_i\in \Omega^{k_i} M$ be a $k_i$-form on $M$ for $i=1,\ldots, m$. The $m$-tuple $(\rho_1,\ldots,\rho_m)$ is said to \emph{have the model forms} $(\rho^0_1,\ldots,\rho^0_k)\in \Lambda^{k_1} \left(\bR^n\right)^* \times \ldots\times \Lambda^{k_m}\left(\bR^n\right)^*$ if for all $x\in M$ there exists an isomorphism $u:\bR^n\rightarrow T_x M$ such that $\bigl( u^*(\rho_1)_x,\ldots, u^*(\rho_m)_x\bigr)=(\rho^0_1,\ldots,\rho^0_m)$. In this case, we call $(u(e_1),\ldots,u(e_n))$ an \emph{adapted basis (for $(\rho_1,\ldots,\rho_m)$)}. If the common $\GL(n,\bR)$-stabilizer of $(\rho^0_1,\ldots,\rho^0_k)\in \Lambda^{k_1} \left(\bR^n\right)^* \times \ldots\times \Lambda^{k_m}\left(\bR^n\right)^*$ is in $\SO(n)$, then one has a well-defined induced Riemannian metric $g_{(\rho_1,\ldots,\rho_m)}$ and orientation by requiring that adapted bases are orthonormal and oriented. So we also have an induced Hodge star operator $\star_{(\rho_1,\ldots,\rho_m)}$.
\end{definition}
We proceed by defining the different geometric structures that we need in this article and start with \emph{special almost Hermitian structures}:
\begin{definition}\label{def:saHs}
Let $M$ be a $2n$-dimensional manifold with $n\geq 3$. A \emph{special almost Hermitian structure on M} is a pair $(\Omega,\Psi)\in \Omega^2 M\times \Omega^n (M,\bC)$ with model tensors
\begin{equation*}
\left(e^{12}+\ldots+e^{2n-1\, 2n},e^1_{\bC}\wedge \ldots \wedge e^n_{\bC}\right)\in \Lambda^2 \left(\bR^{2n}\right)^*\times\Lambda^n\left(\bR^{2n}\right)^*\otimes \bC,
\end{equation*}
where $e^j_{\bC}:=e^{2j-1}-ie^{2j}$ for $j=1,\ldots,n$. We usually set $\Psi_+:=\re{\Psi}$ and $\Psi_-:=\im(\Psi)$ so that $\Psi=\Psi_+ + i\Psi_-$. Note that any $\Psi\in \Omega^n(M,\bC)$ with model tensor $e^1_{\bC}\wedge \ldots \wedge e^n_{\bC}$ and any $\Omega\in \Omega^2 M$ with model tensor $e^{12}+\ldots+e^{2n-1\, 2n}$
naturally define volume forms $\phi(\Psi)$ and $\phi(\Omega)$ on $M$ by
\begin{equation*}
\phi(\Psi):=\begin{cases}
            \tfrac{1}{4} \Psi\wedge \overline{\Psi}=\tfrac{1}{2} \Psi_+^2, & \text{if $n$ is even},\\
						 \tfrac{1}{4i} \Psi\wedge \overline{\Psi}=\tfrac{1}{2} \Psi_-\wedge \Psi_+, & \text{if $n$ is odd},
            \end{cases},\qquad \phi(\Omega)=\tfrac{\Omega^n}{n!}.
\end{equation*}
For a special almost Hermitian structure, these two volume forms fulfill the normalization condition
\begin{equation}\label{eq:normalization}
\phi(\Psi)=2^{n-2} \phi(\Omega).
\end{equation}
As the common $\GL(2n,\bR)$-stabilizer of the model tensors is $\SU(n)\subseteq \SO(2n)$, we have an induced Riemannian metric $g_{(\Omega,\Psi)}$. Moreover, we get an induced almost complex structure $J_{\Psi}$ if we require that for any $x\in M$ and for any adapted basis $(f_1,\ldots,f_{2n})$ at $x$ we have
\begin{equation*}
\left(J_{\Psi}\right)_x(f_{2i-1})=-f_{2i},\quad \left(J_{\Psi}\right)_x(f_{2i})=f_{2i-1}
\end{equation*}
for all $i=1,\ldots,n$. The stabilizer of the model tensor of $\Psi$ is $\SL(n,\bC)\subseteq \GL(n,\bC)\subseteq \GL(2n,\bR)$ and so $J_{\Psi}$ depends, in fact, only on $\Psi$ as the notation indicates. Even more, if $n=2l-1$ is odd and $M$ is oriented, then already a real $n$-form $\Psi_+$ with model tensor $\re\left(e^1_{\bC}\wedge \ldots \wedge e^n_{\bC}\right)$ induces an almost complex structure given in an adapted basis as above and $\Psi:=\Psi_+ +i (-1)^l J_{\Psi_+}^*\Psi_+$ has model tensor $e^1_{\bC}\wedge \ldots \wedge e^n_{\bC}$, cf., e.g., \cite[Chapter 1, Proposition 1.5]{SH}. In this situation, we set $\phi(\Psi_+):=\phi(\Psi)$.

Coming back to arbitrary special almost Hermitian structures, one easily sees that the pair $(g_{(\Omega,\Psi)},J_{\Psi})$ constitutes an almost Hermitian structure with fundamental two-form $\Omega$, i.e. $J_{\Psi}$ is orthogonal with respect to $g_{(\Omega,\Psi)}$ and $\Omega=g_{(\Omega,\Psi)}(\cdot,J_{\Psi}\cdot)$. Now an almost Hermitian structure is K\"ahler if and only if the fundamental two-form is closed and the almost complex structure is integrable. So the induced almost Hermitian structure is K\"ahler if and only if
\begin{equation*}
\begin{split}
d\Omega=&\,0,\\
d\Psi=&\,(\gamma+i J_{\Psi}^*\gamma)\wedge \Psi=\gamma\wedge \Psi_+-J_{\Psi}^*\gamma\wedge \Psi_-+i\left(\gamma\wedge \Psi_-+J_{\Psi}^*\gamma\wedge \Psi_+\right)
\end{split}
\end{equation*}
for some $\gamma\in \Omega^1 M$, cf., e.g., \cite[Chapter 3, Proposition 1.3]{SH}.
However, $J_{\Psi}^*\gamma\wedge \Psi_-=-\gamma\wedge \Psi_+$ and $J_{\Psi}^*\gamma\wedge \Psi_+=\gamma\wedge \Psi_-$, and so the second equation
is equivalent to $d\Psi=\beta\wedge \Psi$ for some $\beta\in \Omega^1 M$. We also call $(\Omega,\Psi)$ then \emph{K\"ahler} and if additionally $\beta=0$, i.e. if
\begin{equation*}
d\Omega=0,\quad d\Psi=0,
\end{equation*}
$(\Omega,\Psi)$ is called a \emph{Calabi-Yau structure}. In this case, $(\Omega,\Psi)$ is parallel, $Hol(g_{(\Omega,\Psi)})\subseteq \SU(n)$ and $g_{(\Omega,\Psi)}$ is Ricci-flat. Note that if $n\geq 4$, then $(\Omega,\Psi)$ is already Calabi-Yau if $d\Omega=0$ and $d\Psi_+=0$ as then automatically $d\Psi_-=0$ \cite{MC}.
\end{definition}
\begin{remark}\label{re:stable}
The case $n=3$ is rather special in the sense that the set of real $3$-forms on a six-dimensional vector space $V$ with model tensor $\re\left(e^1_{\bC}\wedge e^2_{\bC} \wedge e^3_{\bC}\right)$ is an open subset of $\Lambda^3 V^*$. Note further that the set of two-forms on $V$ with model tensor $e^{12}+e^{34}+e^{56}$ is open in $\Lambda^2 V^*$ as well. Moreover, a pair $(\Omega,\psi)\in \Lambda^2 V^*\times \Lambda^3 V^*$ is a special almost Hermitian if and only if both have individually the model tensors $e^{12}+e^{34}+e^{56}$ and $\re\left(e^1_{\bC}\wedge e^2_{\bC} \wedge e^3_{\bC}\right)$, respectively, $\omega\wedge \psi=0$, the normalization condition \eqref{eq:normalization} holds for $n=3$ and $\omega(J_{\psi}\cdot,\cdot)$ is positive-definite, the latter condition being an open condition if the other conditions hold.
\end{remark}
Next, we define \emph{$\SU(3)$-structures} on \emph{seven} dimensional manifolds.
\begin{definition}
Let $M$ be a seven-dimensional manifold. An \emph{$SU(3)$-structure on $M$} is a triple $(\alpha,\omega,\psi)\in \Omega^1 M\times \Omega^2 M \times \left(\Omega^3 (M,\bC)\right)$ with model tensor
\begin{equation*}
(\alpha^0,\omega^0,\psi^0):=\left(e^7,e^{12}+e^{34}+e^{56},e^1_{\bC}\wedge e^2_{\bC}\wedge e^3_{\bC}\right)\in \left(\bR^7\right)^*\times \Lambda^2 \left(\bR^7\right)^*\times \Lambda^3 \left(\bR^7\right)^*.
\end{equation*}

We always set $\rho:=\re(\psi)$ and $\hat{\rho}:=\im(\psi)$ so that $\psi=\rho+i\hat{\rho}$. There is a natural six-dimensional distribution $\cD_{\alpha}$ defined as the kernel of $\alpha$ and a complementary one-dimensional distribution $\cD_{\omega}$ defined as the kernel of $\omega$. Moreover, we denote in this situation by $X$ the vector field tangential to $\cD_{\omega}\subseteq TM$ with $\alpha(X)=1$ and call it the \emph{Reeb} vector field of $(\alpha,\omega,\psi)$.

Now the restriction of $(\omega,\psi)$ to $\cD_{\alpha}$ is a special almost Hermitian structure on $\cD_{\alpha}$ in the sense that for any $x\in M$ it is a special almost Hermitian structure on the vector space $(\cD_{\alpha})_x$. By the above, we have an almost Hermitian structure $(g_{(\omega,\psi)},J_{\psi})$ on $\cD_{\alpha}$ and we define a Riemannian metric $g_{(\alpha,\omega,\psi)}$ on $M$ by
\begin{equation*}
g_{(\alpha,\omega,\psi)}=g_{(\omega,\psi)}+d\alpha^2.
\end{equation*}
This Riemannian metric coincides with the one induced by the $\SU(3)$-structure as in Definition \ref{def:modelforms}. Similarly, we extend $J_{\psi}$ to a vector bundle morphism $J_{(\alpha,\psi)}$ of $TM$ by $J_{(\alpha,\psi)}(X)=0$. Note that then $\left(J_{(\alpha,\psi)},X,\alpha,g_{(\alpha,\omega,\psi)}\right)$ is an almost contact metric structure on $M$ with associated fundamental two-form $\omega$.

We define $(1,0)$- and $(0,1)$-forms $\beta\in \Omega^1 (M,\bC)$ by requiring that $\beta\circ J_{(\alpha,\psi)}=i\beta$ or $\beta\circ J_{(\alpha,\psi)}=-i\beta$, respectively. This allows then to define also (complex) $(p,q)$-forms and real forms of type $(p,q)$ and $(q,p)$. Note that all these forms have $\cD_{\alpha}$ in their kernel. For these forms, we may define a Lefschetz operator and so also \emph{primitive} differential forms. We use the usual notations for all these spaces like $\Omega^{p,q} M$ for the space of all complex $(p,q)$-forms or $[[\Omega_0^{p,q} M]]$ for the space of all primitive real forms of type $(p,q)$ and $(q,p)$ noting that now $\Lambda^k TM \otimes \bC\neq \sum_{i=0}^ k \Lambda^{i,k-i}$. Finally, we call an $\SU(3)$-structure $(\alpha,\omega,\psi)$ \emph{hypo} if
\begin{equation*}
d\omega=0,\quad d(\alpha\wedge \psi)=0.
\end{equation*}
\end{definition}
Next, we consider another $\G$-structure in seven dimensions.
\begin{definition}
\item
Let $M$ be a seven-dimensional manifold. A \emph{$\G_2$-structure on $M$} is a three-form $\varphi\in \Omega^3 M$ with model tensor
\begin{equation*}
\varphi^0:=\omega^0\wedge \alpha^0+\rho^0.
\end{equation*}
As the $\GL(7,\bR)$-stabilizer of $\varphi^0$ is $\G_2\subseteq \SO(7)$, we get an induced Riemannian metric $g_{\varphi}$ and an induced orientation and so also a Hodge star operator $\star_{\varphi}$. One then has
\begin{equation*}
\left(\star_{\varphi}\varphi\right)_x=f^{1234}+f^{1256}+f^{3456}+f^{1367}+f^{1457}+f^{2367}-f^{2467}
\end{equation*}
for all $x\in M$ and for any adapted basis $\left(f^1,\ldots,f^n\right)$ at $x\in M$. $\varphi$ is called \emph{cocalibrated} if
\begin{equation*}
d\star_{\varphi}\varphi=0.
\end{equation*}
One knows that $\varphi$ is parallel, and so $Hol(g_{\varphi})\subseteq \G_2$, if and only $d\varphi=0$ and $d\star_{\varphi}\varphi=0$ \cite{FG}.
\end{definition}
Finally, we will also need the following $\G$-structure in eight dimensions.
\begin{definition}\label{def:Spin7}
Let $M$ be an eight-dimensional manifold. A four-form $\Phi\in \Omega^4 M$ is called \emph{$\Spin(7)$-structure on $M$} if it has model tensor
\begin{equation*}
\Phi^0=\varphi^0\wedge e^8+\star_{\varphi^0}\varphi^0\in \Lambda^4 \left(\bR^8\right)^*
\end{equation*}
As the stabilizer of $\Phi^0$ is $\Spin(7)\subseteq \SO(8)$, we have an induced Riemannian metric $g_{\Phi}$. Moreover, $\Phi$ is parallel with respect to $\nabla^{g_{\Phi}}$ if and only if $d\Phi=0$ and then $Hol(g_{\Phi})\subseteq\Spin(7)$.
\end{definition}
Next, we recall the Hitchin flow:
\begin{proposition}[Hitchin]\label{pro:Hitchinflow}
Let $M$ be a seven-dimensional manifold, $I$ be an open interval and $t$ be the standard coordinate on $I$. Moreover, let $\Phi\in \Omega^4 (M\times I)$ be a parallel $\Spin(7)$-structure on $M\times I$ such that the induced Riemannian metric is of the form $g(t)+dt^2$. Then the induced smooth one-parameter family $I\ni t\mapsto \varphi(t)\in \Omega^3 M$ given by $\varphi(t):=-\left.\frac{\partial}{\partial t}\hook \Phi\right|_{M\times \{t\}}$ consists of cocalibrated $\G_2$-structures which fulfill \emph{Hitchin's flow equations}
\begin{equation}\label{eq:Hitchinflow}
\frac{d}{dt}\star_{\varphi(t)}\varphi(t)=-d\varphi(t).
\end{equation}
Conversely, any smooth one-parameter family $I\ni t\mapsto \varphi(t)$ of $\G_2$-structures on $M$ which is cocalibrated for some $t_0\in I$ and fulfills Hitchin's flow equations \eqref{eq:Hitchinflow} on $I$ defines a parallel $\Spin(7)$-structure $\Phi$ on $M\times I$ given by
\begin{equation}\label{eq:parallelSpin7}
\Phi:=\varphi(t)\wedge dt+\star_{\varphi(t)}\varphi(t).
\end{equation}
The Riemannian metric $g_{\Phi}$ on $M\times I$ is given by $g_{\Phi}=g_{\varphi(t)}+dt^2$ and has holonomy in $\Spin(7)$.
\end{proposition}
The Riemannian manifold $(M\times I, g)$ obtained by the Hitchin flow has, in general, not holonomy equal to $\Spin(7)$. We are, in fact, interested in the cases when the holonomy is less than $\Spin(7)$ but still irreducible. Then the holonomy is either equal to $\SU(4)$ or to $\mathrm{Sp}(2)$ and we have a Calabi-Yau structure or a hyperK\"ahler structure, respectively. We will mainly talk about the first case and refer for a discussion of the second case to Subsection \ref{subsec:hyperKahler}. Similarly to above for $\Spin(7)$-structures on $M\times I$, we may obtain Calabi-Yau structures of certain kind on $M\times I$ by the flow of one-parameter families of structures induced on the hypersurfaces $M\times \{t\}$. Here, the induced structures are hypo $\SU(3)$-structures. More exactly, we have:
\begin{proposition}[Conti, Fino]\label{pro:hypoflow}
Let $M$ be a seven-dimensional manifold, $I$ be an open interval and $t$ be the standard coordinate on $I$.
Let $(\Omega,\Psi)\in \Omega^2 (M\times I)\times \Omega^4 (M\times I)\otimes \bC$ be a Calabi-Yau structure on $M\times I$ such that the induced Riemannian metric is of the form $g(t)+dt^2$. Then the induced smooth one-parameter family $I\ni t\mapsto (\alpha(t),\omega(t),\psi(t))\in \Omega^1 M\times \Omega^2 M\times \Omega^3 (M,\bC)$ given by $\alpha(t):=-\left.\frac{\partial}{\partial t}\hook \Omega\right|_{M\times \{t\}}$, $\omega(t):=\left.\Omega\right|_{M\times \{t\}}$, $\psi(t):=-i\left.\frac{\partial}{\partial t}\hook\Psi\right|_{M\times \{t\}}$ consists of hypo $\SU(3)$-structures which fulfill the \emph{hypo flow equations (for $\SU(3)$-structures)}
\begin{equation}\label{eq:hypoflowSU3}
\frac{d}{dt}\omega(t)=-d\alpha(t),\quad \frac{d}{dt}\left(\alpha(t)\wedge \psi(t)\right)=-id\psi(t)
 \end{equation}
Conversely, any smooth one-parameter family $I\ni t\mapsto (\alpha(t),\omega(t),\psi(t))$ of $\SU(3)$-structures on $M$ which is hypo for some $t_0\in I$ and fulfills the hypo flow equations \eqref{eq:hypoflowSU3} on $I$ defines a Calabi-Yau structure $(\Omega,\Psi)$ on $M\times I$ given by
\begin{equation}\label{eq:parallelSU4}
\Omega:=\omega(t)+\alpha(t)\wedge dt,\quad \Psi:=\psi(t)\wedge (\alpha(t)-idt)
\end{equation}
Moreover, the induced Riemannian metric $g_{(\Omega,\Psi)}$ on $M\times I$ is given by $g_{(\Omega,\Psi)}=\linebreak g_{(\alpha(t),\omega(t),\psi(t))}+dt^2$ and has holonomy in $\SU(4)$.
\end{proposition}
\begin{remark}
\begin{itemize}
\item
Note that there are some different sign conventions in Definition \ref{def:Spin7}, Proposition \ref{pro:Hitchinflow} and Proposition \ref{pro:hypoflow} as, e.g., in \cite{CLSS} and \cite{CF}.
\item
Note further that the flows studied in Proposition \ref{pro:Hitchinflow} and Proposition \ref{pro:hypoflow} are known to admit a unique local solution on an open neighborhood $U$ of $M\times \{0\}$ in $M\times \bR$ for a given initial real-analytic cocalibrated $\G_2$-structure or real-analytic hypo $\SU(3)$-structure on a real-analytic seven-dimensional manifold $M$ \cite{Hi}, \cite{C}, respectively. For the hypo case, note that \cite{C} states only the existence of a solution. However, the proof is based on the Cartan-K\"ahler theorem which gives also uniqueness in the considered case as $M$ has codimension one in $M\times \bR$. In the homogeneous or compact case, we may choose $U=M\times I$ for some open interval $I$ containing $0$.  Note further that in the smooth category, these flows do, in general, not have a local solution, cf. \cite{Br}, \cite{AMM}.
\end{itemize}
\end{remark}
As $\SU(3)\subseteq \G_2$, any $\SU(3)$-structure induces a $\G_2$-structure, which turns out to be cocalibrated if the $\SU(3)$-structure is hypo. Moreover, the solutions of the hypo and the Hitchin flow for these initial values are then related in the same way:
\begin{lemma}\label{le:hypotococalibrated}
Let $M$ be a seven-dimensional manifold. Let $(\alpha,\omega,\psi)\in \Omega^1 M\times \Omega^2 M\times \Omega^3 M$ be a hypo $\SU(3)$-structure on $M$. Then $\varphi:=\omega\wedge \alpha-\hat{\rho}\in \Omega^3 M$ is a cocalibrated $\G_2$-structure with Hodge dual $\star_{\varphi}\varphi=\frac{\omega^2}{2}+\alpha\wedge \rho$ inducing the same metric as $(\alpha,\omega,\psi)$. Moreover, if $I\ni t\mapsto (\alpha(t),\omega(t),\psi(t))\in \Omega^1 M\times \Omega^2 M\times \Omega^3 M$ is a solution of the hypo flow for $\SU(3)$-structures on $M$ with initial value $(\alpha,\omega,\psi)$, then $t\mapsto \varphi(t):=\omega(t)\wedge \alpha(t)-\hat{\rho}(t)\in \Omega^3 M$ is a solution of the Hitchin flow on $M$ with initial value $\varphi$ and the induced Riemannian metric on $M\times I$ coincides with the one induced by $t\mapsto (\alpha(t),\omega(t),\psi(t))$.
\end{lemma}
\begin{proof}
The fact that $\varphi=\omega\wedge \alpha-\hat{\rho}$ is a $\G_2$-structure with Hodge dual $\frac{\omega^2}{2}+\alpha\wedge \rho$ can be checked using
at each point $p\in M$ an adapted basis $(e_1,\ldots,e_7)$ for $(\alpha,\omega,\psi)$ and
noting that then $(-e_2,e_1,-e_4,e_3,-e_6,e_5,e_7)$ is an adapted basis for $\varphi$. But then the closure of the Hodge dual is clear as $(\alpha,\omega,\psi)$ was hypo. Moreover, by the hypo flow equations,
\begin{equation*}
\left(\star_{\varphi(t)}\varphi(t)\right)'=\omega'(t)\wedge \omega(t)+(\alpha(t)\wedge \rho(t))'=-d\alpha(t)\wedge \omega(t)+d\hat{\rho}(t)=-d\varphi(t),
\end{equation*}
i.e. $\varphi(t)$ solves the Hitchin flow as claimed.
\end{proof}
The last lemma has the following important easy consequence.
\begin{corollary}\label{co:holonomyinSU4}
Let $M$ be a real-analytic seven-dimensional manifold and $\varphi\in \Omega^3 M$ be a real-analytic cocalibrated $\G_2$-structure on $M$ which is induced by a real-analytic hypo $\SU(3)$-structure $(\alpha,\omega,\psi)\in \Omega^1 M\times \Omega^2 M\times \Omega^3 M$ in the sense that $\varphi=\omega\wedge \alpha-\hat{\rho}$. Then the Riemannian manifold obtained by the Hitchin flow with initial value $\varphi$ has holonomy contained in $\SU(4)$.
\end{corollary}
In the following sections, we will often concentrate on left-invariant structures on Lie groups $\G$, which we will identify with the corresponding structures on the associated Lie algebra $\mfg$. So we are able to speak about \emph{cocalibrated $\G_2$-structures}, \emph{hypo $\SU(3)$-structures}, etc., \emph{on a Lie algebra $\mfg$}. Moreover, we will consider the Hitchin/hypo flow on a Lie algebra $\mfg$ by which we mean the corresponding flow equation on the associated \emph{simply-connected} Lie group $\tilde{G}$ with left-invariant initial value. Note that then the solution of the hypo/Hitchin flow stays left-invariant and we can, in fact, consider the flow as a flow on $\mfg$.
\section{Hypo $\SU(3)$-structures}\label{sec:hypoSU3}
\subsection{Intrinsic torsion of $\SU(3)$-structures}\label{subsec:inttors}
In this section, we compute the intrinsic torsion $\tau$ of $\SU(3)$-structures $(\alpha,\omega,\rho)\in \Omega^1 M\times \Omega^2 M\times \Omega^3(M,\bC)$.

Naturally, the intrinsic torsion $\tau$ is a section of the vector bundle $T^*M\otimes \mathfrak{su}(3)^{\perp}(P)$ of rank $91$, where $P$ is the $\SU(3)$-reduction of the frame bundle associated to the $\SU(3)$-structure $(\alpha,\omega,\rho)$ and $\mathfrak{su}(3)^{\perp}(P)$ is the vector bundle associated to the natural $\SU(3)$-representation on the orthogonal complement $\mathfrak{su}(3)^{\perp}$ of $\mathfrak{su}(3)$ in $\mathfrak{so}(7)$ with respect to the Killing form of $\mathfrak{so}(7)$. Conti \cite{C} showed that the intrinsic torsion is fully determined by the differentials $(d\alpha,d\omega,d\psi)$. As $\SU(3)$-modules, we have $\bR^{7}=\cD_{\alpha}\oplus \cD_{\omega}=\bR^6\oplus \bR$. Now the well-known representation theory of $\SU(3)$ gives us
\begin{equation*}
\begin{split}
\Lambda^2 (\bR^7)^*=&\alpha^0\wedge (\bR^6)^*\oplus \bR\cdot \omega^0\oplus [\Lambda^{1,1}_0 (\bR^6)^*]\oplus [[\Lambda^{2,0} (\bR^6)^*]]\\
\Lambda^3 (\bR^7)^*=& \bR\cdot \alpha^0\wedge \omega^0\oplus \alpha^0\wedge  [\Lambda^{1,1}_0 (\bR^6)^*]\oplus \alpha^0\wedge [[\Lambda^{2,0} (\bR^6)^*]]\oplus \bR \cdot \rho^0\oplus \bR  \cdot \hat{\rho}^0\\
&\oplus \left[\left[\Lambda^{2,1}_0 \left(\bR^6\right)^* \right]\right]\oplus \omega^0\wedge \left(\bR^6\right)^*\\
\Lambda^4 (\bR^7)^*=&\bR \cdot \alpha^0 \wedge \rho^0\oplus \bR  \cdot \alpha^0\wedge \hat{\rho}^0\oplus \alpha^0\wedge\left[\left[\Lambda^{2,1}_0 \left(\bR^6\right)^* \right]\right]\oplus \alpha^0\wedge \omega^0\wedge \left(\bR^6\right)^*\\
&\oplus \bR  \cdot \left(\omega^0\right)^2\oplus \left[\Lambda^{1,1}_0 \left(\bR^6\right)^* \right]\wedge \omega^0 \oplus \left(\bR^6\right)^*\wedge \rho^0.
\end{split}
\end{equation*}
as $\SU(3)$-modules. These decompositions induce decomposition of the associated vector bundles and we can decompose $d\alpha,\, d\omega,\,d\rho$ and $d\hat{\rho}$ accordingly. 
We have the equation $\omega\wedge \rho=0$ which gives us $d\omega\wedge \rho=-\omega\wedge d\rho$ and so relations between the different components. 
Similarly, we have $d\omega\wedge \hat\rho=-\omega\wedge d\hat\rho$. 
Moreover, $d\psi$ has to lie in the subbundle associated to the $\SU(3)$-module $\alpha\wedge \left(\Lambda^{3,0}\oplus \Lambda^{2,1}\right)\oplus \Lambda^{3,1}\oplus \Lambda^{2,2}$.
Furthermore, $\hat{\rho}\wedge \rho=\frac{2}{3}\omega^3$ gives us $d\hat{\rho}\wedge \rho-\hat{\rho}\wedge d\rho=2d\omega\wedge \omega^2$.
Together, these equations, the fact that $\bR^6\ni X\mapsto X\hook \rho\in \Lambda^{2,0} \left(\bR^6\right)^*$ is an $\SU(3)$-module isomorphism, the identities
$(\beta^{\sharp}\hook \rho^0)\wedge \rho^0=(\omega^0)^2\wedge \beta=-(J(\beta^{\sharp})\hook\rho^0) \wedge \hat{\rho}^0$ for all $\beta\in \left(\bR^6\right)^*$, which only have to be checked for one particular
$\beta\in \left(\bR^6\right)^*$ using Schur's Lemma, and straightforward computations yield the following proposition. Note that parts of the computations
are also done in \cite[Chapter 3, Proposition 3.4]{SH}.
\begin{proposition}\label{pro:inttorsall}
Let $(\alpha,\omega,\psi)\in \Omega^1 M\times \Omega^2 M\times \Omega^3 (M,\bC)$ be an $\SU(3)$-structure on a seven-dimensional manifold $M$. Then
\begin{equation*}
\begin{split}
d\alpha=&\,\alpha\wedge \beta_1+\mu_1 \omega+\tilde{\omega}_1+\beta_2^{\sharp}\hook \rho,\\
d\omega=&\, \frac{3}{2} w_1^- \rho-\frac{3}{2}w_1^+ \hat{\rho}+w_3+w_4\wedge \omega+\frac{2}{3}\mu_2 \alpha\wedge \omega-\alpha\wedge (\beta_3^{\sharp}\hook \rho)+\alpha\wedge \tilde{\omega}_2,\\
d\rho=&\, w_1^+ \omega^2+w_2^+\wedge \omega+w_5\wedge \rho+\mu_2 \alpha\wedge \rho-\mu_3 \alpha\wedge \hat{\rho}+\alpha\wedge \gamma+\alpha\wedge \beta_3\wedge \omega,\\
d\hat{\rho}=&\, w_1^- \omega^2+w_2^-\wedge \omega+w_5\wedge \hat{\rho}+\mu_3 \alpha\wedge \rho+\mu_2 \alpha\wedge \hat{\rho}-\alpha\wedge J^*\gamma-\alpha\wedge J^*\beta_3\wedge \omega
\end{split}
\end{equation*}
with $\mu_1,\,\mu_2,\,\mu_3,\,w_1^+,\,w_1^-\in \Omega^0 M$, $\beta_1,\,\beta_2,\,\beta_3,\,w_4,\,w_5\in [[\Omega^{1,0} M]]$ (i.e. real one-forms which annihilate $\cD_{\omega}$), $\tilde{\omega}_1,\tilde{\omega}_2,w_2^+,w_2^-\in [\Omega^{1,1}_0 M]$ and $\gamma,w_3\in [[\Omega^{2,1}_0 M]]$ and where $J=J_{(\alpha,\psi)}^*$ is the almost complex structure induced on $\cD_{\alpha}$. These forms encode the intrinsic torsion of $(\alpha,\omega,\psi)$ which lies in the subbundle associated to the $\SU(3)$-module $\left(\bR^7\right)^*\otimes \so(7)/\mathfrak{su}(3)=5\bR\oplus 5\bR^6\oplus 4[\Lambda^{1,1}_0]\oplus 2 [[\Lambda^{2,1}_0]]$.
\end{proposition}
Next, we compute the intrinsic torsion of a \emph{hypo} $\SU(3)$-structure. Thereto, we note that $d\omega=0$ if and only if $w_1^-=0$, $w_1^+=0$, $w_3=0$, $w_4=0$, $\mu_2=0$, $\beta_3=0$ and $\tilde{\omega}_2=0$. If this is the case, we have
\begin{equation*}
d(\alpha\wedge \rho)=d\alpha\wedge \rho-\alpha\wedge d\rho=\alpha\wedge \beta_1\wedge \rho+\left(\beta_2^{\sharp}\hook \rho\right)\wedge \rho-\alpha\wedge w_2^+\wedge \omega-\alpha\wedge w_5\wedge \rho
\end{equation*}
and so $d(\alpha\wedge \rho)=0$ if and only if $w_5=\beta_1$, $\beta_2=0$ and $w_2^+=0$. Consequently, if $d\omega=0$ and $d(\alpha\wedge \rho)=0$, then $d(\alpha\wedge \hat{\rho})=-\alpha\wedge w_2^-\wedge \omega$ and so also $d(\alpha\wedge \hat{\rho})=0$ if and only if $w_2^-=0$. Altogether, we obtain:
\begin{proposition}\label{pro:inttorshypo}
Let $(\alpha,\omega,\psi)\in \Omega^1 M\times \Omega^2 M\times \Omega^3 (M,\bC)$ be a hypo $\SU(3)$-structure on a seven-dimensional manifold $M$. Then
\begin{equation}\label{eq:extderhypo}
\begin{split}
d\alpha=&\,\alpha\wedge \beta+\lambda_1 \omega+\tilde{\omega},\\
d\rho=&\, \beta\wedge \rho-\lambda_2 \alpha\wedge \hat{\rho}+\alpha\wedge \gamma,\\
d\hat{\rho}=&\, \beta \wedge \hat{\rho}+\lambda_2 \alpha\wedge \rho-\alpha\wedge J_{(\alpha,\psi)}^*\gamma
\end{split}
\end{equation}
with $\lambda_1,\,\lambda_2\in \Omega^0 M$, $\beta\in [[\Omega^{1,0} M]]$, $\tilde{\omega}\in [\Omega^{1,1}_0 M]$ and $\gamma\in [[\Omega^{2,1}_0 M]]$. So hypo $\SU(3)$-structure are those $\SU(3)$-structures whose intrinsic torsion lies in the $\SU(3)$-submodule
$2V_1\oplus V_6\oplus V_8\oplus V_{12}:=2\bR\oplus \bR^6\oplus [\Lambda^{1,1}_0  \left(\bR^6\right)^*]\oplus [[\Lambda^{2,1}_0 \left(\bR^6\right)^*]]\subseteq \left(\bR^7\right)^*\otimes \so(7)/\mathfrak{su}(3)$
\end{proposition}
We use the decomposition of the intrinsic torsion of a hypo $\SU(3)$-structure to distinguish them into different classes and say that a hypo $\SU(3)$-structure is \emph{of class $2V_1\oplus V_6$} etc. if the intrinsic torsion lies pointwise in $2V_1\oplus V_6$ etc.
If we need to distinguish the two $V_1$-classes we denote them by $V_1(\lambda_1)$ and $V_1(\lambda_2)$.
\subsection{Different torsion classes}\label{subsec:torsionclasses}
In this subsection, we consider different torsion classes and assume throughout this section that $(\alpha,\omega,\psi)$ is a hypo $\SU(3)$-structure on a seven-dimensional manifold $M$.

\subsubsection{Torsion class $V_1(\lambda_1)\oplus V_6\oplus V_8$}
Consider the distribution $\cD_{\omega}$ and note that it is always integrable as it has rank one.
Hence, we may consider the leaf space $W:=M/\cD_{\omega}$, which we assume in the following to be smooth.
Then $\omega$ is a basic form and so can be pushed down to a form on $W$. However, $\psi$ is, in general, only semi-basic.
From Proposition \ref{pro:inttorshypo}, we see that it is basic if and only if it is of class $V_1(\lambda_1)\oplus V_6\oplus V_8$.
Moreover, the pushed-down forms then constitute a special almost Hermitian structure on $W$.
By Proposition \ref{pro:inttorshypo}, we have $d\psi=\beta\wedge \psi$ and so $0=d^2\psi=d\beta\wedge \psi$, which implies $\cD_{\omega}\hook d\beta=0$.
Hence, $\beta$ is also basic and by what we noted in Definition \ref{def:saHs}, the induced special almost Hermitian structure on $W$ is K\"ahler.
\begin{proposition}\label{pro:classlambda1V6V8sol}
Let $(\alpha,\omega,\psi)$ be a hypo $\SU(3)$-structure of class $V_1(\lambda_1)\oplus V_6\oplus V_8$ on a seven-dimensional manifold $M$
and assume that the space of leaves $W:=M/\cD_{\omega}$ is smooth. Then there exists a K\"ahler special almost Hermitian structure $(\Omega,\Psi)$ on $W$
with $\omega=\pi^* \Omega$ and $\psi=\pi^*\Psi$ for $\pi:M\rightarrow W$ being the canonical projection.
\end{proposition}
Next, we like to invert this construction, i.e. we want start with data on $W$ and construct $M$ and the hypo $\SU(3)$-structure of class
$V_1(\lambda_1)\oplus V_6\oplus V_8$ on $M$ from the data on $W$. Thereto, note that by Proposition
\ref{pro:inttorshypo} $d\alpha=\alpha\wedge \beta+\tau$ for some semi-basic $\tau\in \Omega^2 M$. As $\tau$ should come from data on $W$, it should be basic.
Hence, $0=d^2\alpha=\tau\wedge \beta-\alpha\wedge d\beta+d\tau$ and $\cD_{\omega}\hook d\beta=0$ imply $d\tau=-\tau\wedge \beta$ and $d\beta=0$.

So assume now that we have a six-dimensional manifold endowed with a special almost Hermitian structure $(\Omega,\Psi)$ with $d\Psi=\beta\wedge \Psi$ such that $d\beta=0$ and a real $(1,1)$-form $\tau\in \Omega^2 M$ with $d\tau=-\tau\wedge \beta$. We additionally assume that $\beta$ is not only closed but exact and let $f\in C^{\infty}(W)$ be such that $df=\beta$. Set $\tilde\tau:=e^f\tau$ and note that $d\tilde\tau=0$. If $\tilde\tau\in \Omega^2 W$ has integral periods, we may build a principal $S^1$-bundle $\pi:M\rightarrow W$ with principal $S^1$-connection $\theta\in\Omega^1 M$ such that $d\theta=\pi^*\tilde\tau$. Setting $\alpha:=e^{-\pi^*f} \theta$, we get $d\alpha=\alpha\wedge \pi^*\beta+\pi^*\tau$. So we have obtained
\begin{proposition}\label{pro:classlambda1V6V8sol2}
Let $(W,\Omega,\Psi)$ be six-dimensional K\"ahler special almost Hermitian manifold endowed with a real $(1,1)$-form $\tau$ such that $d\tau=-\tau\wedge \beta$ 
for the unique one-form $\beta$ with $d\Psi=\beta\wedge \Psi$. Assume further that $\beta=df$ for some $f\in C^{\infty}(W)$ and that $e^f \tau$ has integral periods. Then there exists a principal $S^1$-bundle $\pi^*:M\rightarrow W$ and a
semi-basic one-form $\alpha\in \Omega^1 M$ such that $(\alpha,\pi^*\Omega,\pi^*\Psi)$ is a hypo $\SU(3)$-structure of class $V_1(\lambda_1)\oplus V_6\oplus V_8$ on $M$.
\end{proposition}
\begin{remark}
A special case of Proposition \ref{pro:classlambda1V6V8sol2} is when $(W,\Omega,\Psi)$ is Calabi-Yau and $\Omega$ has integral periods. Taking then $\tau=\Omega$
gives a hypo $\SU(3)$-structure of class $V_1(\lambda_1)$ on the Boothby-Wang fibration over $W$.
\end{remark}
We are ultimately interested in left-invariant hypo $\SU(3)$-structures $(\alpha,\omega,\psi)$ on Lie groups $G$. As the space of leaves $G/\cD_{\omega}$ should be again a Lie group,
we need that $\cD_{\omega}$ is an ideal in the associated Lie algebra $\mfg$. Note that then $(\alpha,\omega,\psi)$ is automatically of class $V_1(\lambda_1)\oplus V_6\oplus V_8$.
However, not all hypo $\SU(3)$-structures of class $V_1(\lambda_1)\oplus V_6\oplus V_8$ have $\cD_{\omega}$ as an ideal.

Take, e.g., the seven-dimensional Lie algebra $\mfg$
defined by the differentials of a dual basis $\left(e^1,\ldots,e^7\right)$ by $(e^{27},-e^{17},-e^{47},e^{37},0,0,0)$ and the $\SU(3)$-structure
$(\alpha,\omega,\psi):=\left(e^7,e^{12}+e^{34}+e^{56},e^1_{\bC}\wedge e^2_{\bC}\wedge e^3_{\bC}\right)$ on $\mfg$ and note that $(\alpha,\omega,\psi)$
is a hypo $\SU(3)$-structure of class $V_1(\lambda_1)\oplus V_6\oplus V_8$ (even of class $\{0\}$, i.e. parallel) but $\cD_{\omega}=\spa{e_7}$ is not an ideal.

So we are considering now a stricter class in the left-invariant situation then in the general situation but are so able to obtain a one-to-one correspondence between six- and seven-dimensional data.
\begin{thm}\label{th:Domegaideal}
There exists a one-to-one correspondence between the following data:
\begin{itemize}
\item[(i)]
Isomorphism classes of seven-dimensional Lie algebras $\mfg$ endowed with a hypo $\SU(3)$-structure $(\alpha,\omega,\psi)$ such that $\cD_{\omega}$ is an ideal.
\item[(ii)]
Isomorphism classes of six-dimensional K\"ahler Lie algebras $(\mfh,\Omega,J)$ endowed with a $(3,0)$-form $\Psi$ and a real $(1,1)$-form $\tau$ such that $\Psi\wedge \overline{\Psi}=\frac{4i}{3}\Omega^3$ and such that $d\beta=0$ and $d\tau=-\tau\wedge \beta$ for the unique real one-form $\beta\in \mfh^*$ fulfilling $d\Psi=\beta\wedge \Psi$
\end{itemize}
\end{thm}
\begin{proof}
Take first a seven-dimensional Lie algebra $\mfg$ endowed with a hypo $\SU(3)$-structure $(\alpha,\omega,\psi)$ for which $\cD_{\omega}$ is an ideal.
Then Proposition \ref{pro:classlambda1V6V8sol} yields that $\mfh:=\mfg/\cD_{\omega}$ is a six-dimensional Lie algebra which possesses a K\"ahler special almost
structure $(\Omega,\Psi)$ given by the push-downs of $(\omega,\psi)$. In particular, $\Psi\wedge \overline{\Psi}=\frac{4i}{3}\Omega^3$.
Moreover,
\begin{equation*}
0=d^2 \alpha=d(\alpha\wedge \beta+\tau)=\tau\wedge \beta-\alpha\wedge d\beta+d\tau
\end{equation*}
for some real two-form $\tau$ of type $(1,1)$ with $\cD_{\omega}\hook \tau=0$. As $\cD_{\omega}$ is an ideal, we must have $\cD_{\omega}\hook d\tau=0$ and
$\cD_{\omega}\hook d\beta=0$ and so get $d\tau=-\tau\wedge \beta$ and $d\beta=0$. Pushing down $\tau$ and $\beta$ to $\mfh$ gives the claimed result.

Let now $(\mfh,\Omega,J)$ be a K\"ahler Lie algebra endowed with a $(3,0)$-form $\Psi$ and a real $(1,1)$-form $\tau$ such that $\Psi\wedge \overline{\Psi}=\frac{4i}{3}\Omega^3$ and such that $d\beta=0$ and $d\tau=-\tau\wedge \beta$ for the unique real one-form $\beta\in \mfh^*$ fulfilling $d\Psi=\beta\wedge \Psi$. We build a seven-dimensional Lie algebra $\mfg$ as the vector space direct
sum $\mfh\oplus \bR$ together with the Lie brackets
\begin{equation*}
[1,X]_{\mfg}=-\beta(X)1,\quad [X,Y]_{\mfg}=[X,Y]_{\mfh}-\tau(X,Y)1
\end{equation*}
for all $X,\, Y\in \mfh$. We have to check the Jacobi identity. To do this, take $X,\,Y,\,Z\in \mfh$. Then
\begin{equation*}
\begin{split}
[1,[X,Y]_{\mfg}]_{\mfg}+[X,[Y,1]_{\mfg}]_{\mfg}+[Y,[1,X]_{\mfg}]_{\mfg}=&[1,[X,Y]_{\mfh}]_{\mfg}+[X,\beta(Y)1]_{\mfg}-[Y,\beta(X)1]_{\mfg}\\
=&-\beta([X,Y]_{\mfh})1=d_{\mfh}\beta(X,Y)1=0
\end{split}
\end{equation*}
and
\begin{equation*}
\begin{split}
\sum_{cyc}[X,[Y,Z]_{\mfg}]_{\mfg}=&\sum_{cyc} [X,[Y,Z]_{\mfh}]_{\mfg}-\sum_{cyc}[X,\tau(Y,Z)1]_{\mfg}\\
=&\sum_{cyc} [X,[Y,Z]_{\mfh}]_{\mfh}-\sum_{cyc} \tau(X,[Y,Z]_{\mfh})1-\sum_{cyc}\beta(X)\tau(Y,Z)\\
=&-(d_{\mfh}\tau+\beta\wedge \tau)1=0.
\end{split}
\end{equation*}
So $\mfg$ is, in fact, a Lie algebra with $\cD_{\omega}$ being an ideal. Let $\pi:\mfg\rightarrow \mfh$ be the projection onto $\mfh$ and note that this map is a Lie algebra homomorphism.
Let $\alpha$ be that element in the annihilator of $\mfh$ in $\mfg$ with $\alpha(1)=1$. Then $d\alpha=\alpha\wedge \pi^*\beta+\pi^*\tau$
and so $(\alpha,\omega,\psi):=(\alpha,\pi^*\Omega,\pi^*\Psi)$ is an $\SU(3)$-structure. Moreover, $(\alpha,\omega,\psi)$ is hypo as $d\omega=\pi^*d\Omega=0$ and
\begin{equation*}
\begin{split}
d(\alpha\wedge \psi)=&\, d\alpha\wedge \pi^*\Psi-\alpha\wedge \pi^*d\Psi=\alpha\wedge \pi^*(\beta\wedge \Psi)+\pi^*(\tau\wedge \Psi)-\alpha\wedge \pi^*(\beta\wedge \Psi)\\
=&\, 0,
\end{split}
\end{equation*}
using that $\tau$ is of type $(1,1)$ and $\psi$ of type $(3,0)$.
\end{proof}
\begin{remark}\label{re:V1lambda1V6V8}
\begin{itemize}
\item
If we start with a six-dimensional K\"ahler Lie algebra with $d\beta=0$ we may always take $\tau=0$ to construct a seven-dimensional Lie algebra with hypo $\SU(3)$-structure with $\cD_{\omega}$ being an ideal. Note that then the hypo $\SU(3)$-structure is of class $V_6$ and $\cD_{\alpha}$ is a subalgebra.
\item
There are six-dimensional K\"ahler Lie algebras for which $\beta$ is closed and ones for which $\beta$ is not closed. An example with closed $\beta$ is given in Example \ref{ex:closedbeta} below whereas an example with $d\beta\neq 0$ is the following one:\\
Take the six-dimensional Lie algebra $\mfh$ with basis $(e_1,\ldots,e_6)$ such that the only non-zero Lie bracket (up to anti-symmetry) are $[e_4,e_1]=e_1$, $[e_4,e_2]=-e_3$, $[e_4,e_3]=e_2$. Set $\Omega:=e^{14}+e^{23}+e^{56}$ and define $J$ uniquely by $J(e_1)=-e_4$, $J(e_2)=-e_3$, $J(e_5)=-e_6$. Then $d\Omega=0$ and $\Psi:=\left(e^1-ie^4\right)\wedge \left(e^2-ie^3\right)\wedge  \left(e^5-ie^6\right)$ is a non-zero $(3,0)$-form with $d\Psi=\left(-e^1-e^4\right)\wedge \Psi$. Hence, $(\mfh,\Omega,J)$ is K\"ahler and
and $\beta=-e^1-e^4$ but $d\beta=-de^1=-e^{14}\neq 0$.
\end{itemize}
\end{remark}
\begin{example}\label{ex:closedbeta}
We consider $\mfh=\mathfrak{r}_2\oplus\mathfrak{r}_2\oplus \mathfrak{r}_2$ with basis $(e_1,\ldots,e_6)$ whose only non-zero Lie bracket (up to anti-symmetry)
are $[e_{2i-1},e_{2i}]=e_{2i-1}$ for $i=1,2,3$. Set $\Omega:=e^{12}+e^{34}+e^{56}$ and $J(e_{2i-1})=-e_{2i}, J(e_{2i})=e_{2i-1}$ for all $i=1,\,2,\,3$.
Then $d\Omega=0$ and $\Psi:=\left(e^1-ie^2\right)\wedge \left(e^3-ie^4\right)\wedge \left(e^5-ie^6\right)$ is a non-zero $(3,0)$-form with
$\Psi\wedge \overline{\Psi}=\frac{4i}{3}\Omega^3$ and $d\Psi=(e^2+e^4+e^6)\wedge \Psi$. So $(\mfh,\omega,J)$ is K\"ahler with $\beta=e^2+e^4+e^6$ and $d\beta=0$. Now
\begin{equation*}
\tau:=e^{14}+e^{16}+e^{36}-e^{23}-e^{25}-e^{45}+2\Omega
\end{equation*}
is a $(1,1)$-form with $d\tau=-\tau\wedge \beta$. Hence, by Theorem \ref{th:Domegaideal} the seven-dimensional Lie algebra $\mfg$ dually defined by
\begin{equation*}
(-e^{12},0,-e^{34},0,-e^{56},0,-e^{27}-e^{47}-e^{67}+\tau)
\end{equation*}
has $\cD_{\omega}=\spa{e_7}$ as ideal and admits the hypo $\SU(3)$-structure $(e^7,\Omega,\Psi)$ of class $V_1(\lambda_1)\oplus V_6\oplus V_8$.
\end{example}
\subsubsection{Torsion class $V_1(\lambda_2)\oplus V_6\oplus V_{12}$}
In this section, we are considering hypo $\SU(3)$-structures $(\alpha,\omega,\psi)$ on seven-dimensional manifolds $M$ of class $V_1(\lambda_2)\oplus V_6\oplus V_{12}$. Note that Proposition \ref{pro:inttorshypo} shows that $(\alpha,\omega,\psi)$ is of class $V_1(\lambda_2)\oplus V_6\oplus V_{12}$ if and only if $\cD_{\alpha}$ is integrable. If this is the case, we may consider an integral manifold $\iota:N\to M$ of $\cD_{\alpha}$ and pullback the forms $\omega$ and $\psi$. Then
\begin{equation*}
d\iota^*\omega=0,\quad d\iota^*\psi=\iota^*d\psi=\iota^*\left(\beta\wedge \psi+i\lambda_2 \alpha\wedge \psi+\alpha\wedge (\gamma-iJ_{(\alpha,\psi)}^*\gamma))\right)=\iota^*\beta\wedge \iota^*\psi
\end{equation*}
and so the almost Hermitian structure underlying the special almost Hermitian structure $(\iota^*\omega,\iota^*\psi)$ is K\"ahler. Moreover, 
\begin{equation*}
0=d^2\alpha=d(\alpha\wedge \beta)=-\alpha\wedge d\beta
\end{equation*}
and so $d\beta\in \alpha\wedge \Omega^1 M$. This implies $d\iota^*\beta=\iota^*d\beta=0$. Moreover, $(\alpha,\omega,\psi)$ is of class $V_1(\lambda_2)\oplus V_{12}$ if and only if $d\alpha=0$ and so if and only if $(\iota^*\omega,\iota^*\psi)$ is Calabi-Yau. Summarizing, we have obtained:
\begin{proposition}\label{pro:V1lambda2V6V12gen}
Let $M$ be a seven-dimensional manifold and $(\alpha,\omega,\psi)$ be a hypo $\SU(3)$-structure on $M$. Then $(\alpha,\omega,\psi)$ is of class $V_1(\lambda_2)\oplus V_6\oplus V_{12}$ if and only if $\cD_{\alpha}$ is integrable. If this is the case and $\iota:N\to M$ is an integral manifold of $\cD_{\alpha}$, then $(\iota^*\omega,\iota^*\psi)$ is a special almost Hermitian structure whose underlying almost Hermitian structure is K\"ahler and for which the unique real one-form $\beta\in \Omega^1 N$ with $d\iota^*\psi=\beta\wedge \iota^*\psi$ fulfills $d\beta=0$. $(\iota^*\omega,\iota^*\psi)$ is Calabi-Yau if and only if $(\alpha,\omega,\psi)$ is of class $V_1(\lambda_2)\oplus V_{12}$.
\end{proposition}

Conversely, suppose we are in a left-invariant setting and have given a six-dimensional K\"ahler Lie algebra $(\mfh,\omega,J)$ with $d\beta=0$ for the unique real one-form $\beta\in \mfh^*$ with $d\psi=\beta\wedge \psi$ for some $\psi\in\Lambda^{3,0} \mfh^*$ with $\psi\wedge \overline{\psi}=\tfrac{4i}{3}\omega^3$. One way to obtain a hypo $\SU(3)$-structure $(\alpha,\omega,\psi)$ of type $V_1(\lambda_2)\oplus V_6\oplus V_{12}$ on a seven-dimensional Lie algebra with $\mfh=\cD_{\alpha}$ is using Theorem \ref{th:Domegaideal} with $\tau=0$, cf. Remark \ref{re:V1lambda1V6V8}. But then $(\alpha,\omega,\psi)$ is of type $V_6$ and $\cD_{\omega}$ is an ideal. To get examples for which $\cD_{\omega}=\spa{X}$ is not necessarily an ideal, we consider the vector space sum $\mfg:=\mfh\oplus \spa{X}$ and first have to set the Lie brackets $[X,Z]_{\mfg}$ for $Z\in \mfh$ such that $\mfg$ gets a Lie algebra. As we want to get $d\alpha=\alpha\wedge \beta$, we must have $[X,Z]_{\mfg}=f(Z)-\beta(Z)X$ for some $f\in \mathrm{End}(\mfh)$ and the Jacobi identity implies that $f$ and $\beta$ have to fulfill the equations $f([Z_1,Z_2]_{\mfh})=[f(Z_1),Z_2]_{\mfh}+[Z_1,f(Z_2)]_{\mfh}$ and
$f([Y,Z_1]_{\mfh})=[f(Y),Z_1]_{\mfh}+[Y,f(Z_1)]_{\mfh}-||\beta||^2 f(Z_1)$ for all $Z_1,Z_2\in \ker(\beta)$ and $Y:=\beta^{\sharp}\in \mfh$. A solution of these equations is provided by setting $f|_{\ker(\beta)}:=\ad(Y)|_{\ker(\beta)}$ and $f(Y):=||\beta||^2 Y$ noting that $\ker(\beta)$ is an ideal in $\mfh$. Next, we have to determine when the $\SU(3)$-structure $(\alpha,\omega,\psi)$ on $\mfg$ is hypo. First of all, $d_{\mfg}(\alpha\wedge \psi)= \alpha\wedge \beta\wedge \psi-\alpha\wedge d_{\mfh}\psi=0$
and $d_{\mfg}\omega=d_{\mfh}\omega+\alpha\wedge f.\omega=\alpha\wedge f.\omega$. Thus, $(\alpha,\omega,\psi)$ is hypo if and only if $f.\omega=0$. Moreover, we have $(\cL_{Y}\omega)(Z_1,Z_2)= f.\omega(Z_1,Z_2)$ and
\begin{equation*}
-f.\omega(Y,Z_1)=\omega( ||\beta||^2 Y, Z_1)+\omega(Y,[Y,Z_1]_{\mfh})=\left(\beta\wedge Y\hook \omega\right)(Y,Z_1)-(\cL_{Y} \omega)(Y,Z_1)
\end{equation*}
for all $Z_1,\,Z_2\in \ker(\beta)$. Hence, $f.\omega=0$ if and only if $\cL_{Y}\omega= \beta\wedge Y\hook \omega$ and we got
\begin{proposition}\label{pro:KahlertoV1V6V12}
Let $(\mfh,\omega,J)$ be a six-dimensional K\"ahler Lie algebra and $\psi\in \Lambda^{3,0}\mfh^*$ be such that $\psi\wedge \overline{\psi}=\tfrac{4i}{3}\omega^3$ holds. Assume that the unique real one-form $\beta\in \mfh^*$ with $d\psi=\beta\wedge \psi$ fulfills $d\beta=0$ and that $Y:=\beta^{\sharp}$ fulfills $\cL_Y \omega=\beta\wedge (Y\hook \omega)$. Then the seven-dimensional Lie algebra $\mfg:=\mfh\oplus \bR\cdot X$ endowed with the anti-symmetric bilinear map $[\cdot,\cdot]_{\mfg}:\mfg\times \mfg\rightarrow \mfg$ defined by $[\cdot,\cdot]_{\mfg}|_{\Lambda^2\mfh}=[\cdot,\cdot]_{\mfh}$, $[X,Z]_{\mfg}:= [Y,Z]_{\mfh}$ for all $Z\in \ker(\beta)$ and $[X,Y]_{\mfg}:=\left\|\beta\right\|^2 (Y-X)$ is a seven-dimensional Lie algebra which admits the hypo $\SU(3)$-structure $(\alpha,\omega,\psi)$ of type $V_1(\lambda_2)\oplus V_6\oplus V_{12}$, where $\alpha\in \Ann{\mfh}$ is uniquely defined by $\alpha(X)=1$.
\end{proposition}
\begin{example}
We take $\mfh=\mathfrak{r}_2\oplus\bR^4$ with basis $(e_1,\ldots,e_6)$ such that the only non-zero Lie bracket (up to anti-symmetry) is given by $[e_1,e_2]=e_1$. Moreover, we take $\omega$ and $\psi$ of the same form as in Example \ref{ex:closedbeta}. Then $\beta=e^2$, $Y=e_2$ and $\cL_Y \omega=d(Y\hook \omega)=-de^1=e^{12}=e^2\wedge (-e^1)=e^2\wedge (Y\hook \omega)$. Thus, we may apply Proposition \ref{pro:KahlertoV1V6V12} to get a seven-dimensional Lie algebra $\mfg$ with basis $(e_1,\ldots,e_7)$ whose non-zero Lie brackets (up to anti-symmetry) $[e_1,e_2]=e_1$, $[e_7,e_1]=-e_1$, $[e_7,e_2]=-e_7+e_2$ with hypo $\SU(3)$-structure $(e^7,\omega,\psi)$ of type $V_1(\lambda_2)\oplus V_6\oplus V_{12}$.
\end{example}
The next proposition classifies hypo $\SU(3)$-structures of type $V_1(\lambda_2)\oplus V_{12}$ on seven-dimensional Lie algebras. Thereto, note that by Proposition \ref{pro:inttorshypo} a hypo $\SU(3)$-structure $(\alpha,\omega,\psi)$ is of type $V_1(\lambda_2)\oplus V_{12}$ if and only if $\cD_{\alpha}$ is an ideal in $\mfg$ and that then by Proposition \ref{pro:V1lambda2V6V12gen} the six-dimensional Lie algebra $\cD_{\alpha}$ possesses the Calabi-Yau structure $(\omega,\psi)$. Using these properties, we obtain
\begin{thm}\label{th:ClassV1lambda2V12}
Let $(\alpha,\omega,\psi)\in \mfg^*\times \Lambda^2 \mfg^*\times \left(\Lambda^3 \mfg^*\otimes \bC\right)$ be a hypo $\SU(3)$-structure on a seven-dimensional Lie algebra $\mfg$. Then $(\alpha,\omega,\psi)$ is of type $V_1(\lambda_2)\oplus V_{12}$ if and only if $(\mfg,\alpha,\omega,\psi)$ is isomorphic to a hypo $\SU(3)$ Lie algebra $(\tilde{\mfg},\tilde{\alpha},\tilde{\omega},\tilde{\psi})$ in the following list:
\begin{itemize}
\item[(i)]
$\tilde{\mfg}$ is an almost Abelian Lie algebra with codimension one Abelian ideal $\mfu$ for which there exists a non-degenerate two-form $\tilde{\omega}\in \Lambda^2 \mfu^*$ such that some $X \in \tilde{\mathfrak{g}}\backslash \mfu$ acts symplectically on $(\mfu,\tilde{\omega})$, $\tilde{\psi}\in \Lambda^3 \Ann{X}\otimes \bC\cong \Lambda^3 \mfu^*\otimes \bC$ is such that $(\tilde{\omega},\tilde{\psi})$ defines a special almost Hermitian structure on $\mfu$ and $\alpha\in \Ann{\mfu}$ fulfills $\alpha(X)=1$.
\item[(ii)]
$\tilde{\mfg}$ has a basis $(e_1,\ldots,e_7)$ with non-zero Lie brackets (up to anti-symmetry) given by
\begin{equation*}
\begin{split}
[e_1,e_6]&=a e_2,\, [e_2,e_6]=-a e_1,\, [e_3,e_6]=-ae_4,\, [e_4,e_6]=ae_3,\\
[e_1,e_7]&=-a_1 e_2+a_2 e_3+a_3 e_4,\, [e_2,e_7]=a_1 e_1+a_3 e_3-a_2 e_4,\\
[e_3,e_7]&= a_2 e_1+a_3 e_2-a_4 e_4,\, [e_4,e_7]=a_3 e_1-a_2 e_2+a_4 e_3,\, [e_6,e_7]=a_5 e_5
\end{split}
\end{equation*}
for certain $a\in \bR^*$, $(a_1,\ldots,a_5)\in \bR^5$ and
\begin{equation*}
(\tilde{\alpha},\tilde{\omega},\tilde{\psi}):=\left(e^7,e^{12}+e^{34}+e^{56},\left(e^1-ie^2\right)\wedge\left(e^3-ie^4\right)\wedge\left(e^5-ie^6\right)\right). 
\end{equation*}
\end{itemize}
\end{thm}
\begin{proof}
First of all, we classify six-dimensional Calabi-Yau Lie algebras $(\mfh,\omega,\psi)$. Of course, $\bR^6$ is a Calabi-Yau Lie algebra for any special almost Hermitian structure and we assume now that $\mfh$ is not Abelian. However, still the induced Riemannian metric is flat as it is Ricci-flat, cf. \cite{AK}. Now Milnor \cite{Mi} classified flat Lie algebras and from this classification we get, cf. also \cite[Proposition 2.1]{BDF}, that
 $\mfh=\bR^{2k}\rtimes \bR^{6-2k}$ for some $k\in \{1,2\}$ with orthogonal factors, $\bR^{6-2k}$ acting by skew-symmetric endomorphisms on $\bR^k$ with $\ad(Z)=\nabla_Z$ for all $Z\in \bR^{6-k}$, $\nabla_W=0$ for all $W\in \bR^k$ and $[\mfh,\mfh]= \bR^{2k}$. Moreover, $0=(\nabla_Z J)(W)=\nabla_Z J(W)-J(\nabla_Z W)=[Z,J(W)]-J([Z,W])$ for all $Z\in \bR^{6-k}$ and all $W\in \bR^k$, which gives us that the commutator $[\mfh,\mfh]$ is $J$-invariant. Furthermore, we must have $k=2$: Otherwise, we may choose a $\bC$-basis $W_1,\,W_2,\,W_3$ of $\mfh$ with $W_1\in [\mfh,\mfh]=\bR^2\subseteq \bR^2\rtimes \bR^4$ and $W_2,\,W_3\in \bR^4\subseteq \bR^2\rtimes \bR^4$ such that $[W_2,W_1]=a JW_1$ for some $0\neq a\in \bR$ and $[JW_2,W_1]=0=[W_3,W_1]=[JW_3,W_1]$. But then
\begin{equation*}
0\neq a\psi(JW_1,JW_2,W_3)=\psi([W_2,W_1],JW_2,W_3)=-d\psi(W_2,W_1,JW_2,W_3)=0,
\end{equation*} 
	a contradiction. Hence, $k=2$ and $\mfh=\bR^4\rtimes \bR^2$. Now $\SU(3)$ acts transitively on the space of all $J$-invariant four-dimensional subspaces and so we may assume that there exists an adapted basis $e_1,\ldots,e_6$ for $(\omega,\psi)$ with $e_1,\ldots,e_4$ being a basis for $\bR^4$ and $e_5,\,e_6$ being a basis of $\bR^2$. Then $\omega=\omega_1+e^{56}$, $\psi_+=\omega_2\wedge e^5-\omega_3\wedge e^6$ for $\omega_1:=e^{12}+e^{34}$, $\omega_2:=e^{13}-e^{24}$ and $\omega_3:=e^{14}+e^{23}$. Moreover, setting $f:=\ad(e_5)|_{[\mfh,\mfh]}$ and $g:=\ad(e_6)|_{[\mfh,\mfh]}$, we have $d\nu=e^5\wedge f.\nu+e^6\wedge g.\nu$ for all $\nu\in \Lambda^k \spa{e^1,\ldots,e^4}$, $de^5=0$ and $de^6=0$. Hence $0=d\omega=e^5\wedge f.\omega_1+e^6\wedge g.\omega_2$, i.e. $f.\omega_1=g.\omega_1=0$. So $f,g\in \mathfrak{u}(2)$ and $f$ and $g$ act on $\spa{\omega_2,\omega_3}$ as rotations, i.e. $f.\omega_2=a\omega_3$, $f.\omega_3=-a\omega_2$, $g.\omega_2=b\omega_3$, $g.\omega_3=-b\omega_2$ for $a,b\in \bR$. Thus,
\begin{equation*}
0=d\psi_+=e^6\wedge g.\omega_2\wedge e^5-e^5\wedge f.\omega_3\wedge e^6=-\left(b\omega_3-a\omega_2\right)\wedge e^{56}.
\end{equation*}
This implies that $f$, $g$ preserve $\omega_2$, $\omega_3$, i.e. that $f,\,g\in \mathfrak{sp}(1)=\mathfrak{su}(2)$. As they must commute, they have to be linearly dependent. Now any element in $\mathfrak{su}(2)$ can be conjugated to a matrix of the form $\diag(aD,-aD)$ with $D:=\left(\begin{smallmatrix} 0 & -1 \\ 1 & 0 \end{smallmatrix}\right)$ for some $a\in \bR^*$ by an element of $\mathrm{U}(2)$ and so we can find some $A\in SU(3)$ preserving the subspaces $\bR^4$ and $\bR^2$ of $\mfh$ such that $Ae_5$ is central and $g=\ad(Ae_6)|_{\bR^4}=\diag(aD,-aD)$ with respect to the basis $Ae_1,Ae_2,Ae_3,Ae_4$. Denoting the adapted basis $Ae_1,\ldots Ae_6$ by $e_1,\ldots,e_6$, we get $\omega=e^{12}+e^{34}+e^{56}$, $\psi=\left(e^1-ie^2\right)\wedge \left(e^3-ie^4\right)\wedge \left(e^5-ie^6\right)$ and that the Lie bracket of $\mfh$ is determined by $[e_6,e_1]=ae_2$, $[e_6,e_2]=-ae_1$, $[e_6,e_3]=-ae_4$ and $[e_6,e_4]=ae_3$. So any six-dimensional Calabi-Yau Lie algebra $(\mfh,\omega,\psi)$ is either Abelian or of the just determined form.

Now we have to check when an $\SU(3)$-structure $(e^7,\omega,\psi)$ on a seven-dimensional Lie algebra $\mfg$ of the form $\mfg=\mfh\rtimes \bR e_7$ with $(\omega,\psi)$ being a Calabi-Yau structure on $\mfh$ is hypo. Generally, setting $h:=\ad(e_7)|_{\mfh}$, we have $d_{\mfg}e^7=0$ and $d_{\mfg}\nu=d_{\mfh}\nu+e^7\wedge h.\nu$ for all $\nu\in \Lambda^k \mfh^*$ and so $d_{\mfh}\omega=\alpha\wedge h.\omega$ and $d_{\mfg}\left(e^7\wedge \psi\right)=0$. Hence, $(e^7,\omega,\psi)$ is hypo if and only if $h\in \mathfrak{sp}(\mfh,\omega)$. If $\mfh$ is Abelian, this gives (i) in Theorem \ref{th:ClassV1lambda2V12}. In the other case, we have more restrictions as $h\in \mathfrak{der}(\mfh)\neq \mathfrak{gl}(\mfh)$. As a derivation, $h$ has to preserve both the center $\spa{e_5}$ as well as the commutator $\spa{e_1,e_2,e_3,e_4}$ of $\mfh$ and so it has the form
\begin{equation*}
h=\begin{pmatrix} A & 0 & v \\ 0 & b & c \\ 0 & 0  & d \end{pmatrix}
\end{equation*}
with respect to the basis $e_1,\ldots,e_6$, where $A\in \bR^{4\times 4}$, $v\in \bR^2$ and $b,c,d\in \bR$. Since $\mathfrak{sp}(\mfh,\omega)=\mathfrak{sp}(6,\bR)$ with respect to the basis $e_1,\ldots,e_6$, we must have $A\in\mathfrak{sp}(4,\bR)$, $v=0$ and $d=-b$. Moreover, $h$ is a derivation if and only if
\begin{equation*}
\begin{split}
Ag(e_i)&=A\, [e_6,e_i]=h([e_6,e_i])=[h(e_6),e_i]+[e_6,h(e_i)]=-b[e_6,e_i]+[e_6,Ae_i]\\
&=-b\cdot g(e_i)+gA(e_i)
\end{split}
\end{equation*}
for all $i=1,\ldots,4$, i.e. if and only if $[A,g]=-b\cdot g$. Using that $A\in\mathfrak{sp}(4,\bR)$, we may solve this equation explicitly and get $b=0$ as well as
\begin{equation*}
A=\begin{pmatrix} 
 0 & a_1 & a_2 & a_3 \\
-a_1 & 0 &  a_3 & -a_2 \\
a_2 & a_3 & 0 & a_4 \\
a_3 & -a_2 & -a_4 & 0
\end{pmatrix}
\end{equation*}
for arbitrary $a_1,\ldots, a_4\in \bR$. This shows the statement.
\end{proof}
\subsubsection{Torsion class $2V_1\oplus V_{12}$}
Next, we look at the torsion class $2V_1\oplus V_{12}$. Then $d\alpha=\lambda_1\omega$, $d\omega=0$ and so $0=d^2\alpha=d\lambda_1\wedge\omega$, which gives $d\lambda_1=0$.
\begin{lemma}\label{le:normalizelambda1}
Let $(\alpha,\omega,\psi)\in  \Omega^1 M\times \Omega^2 M\times \Omega^3 M$ be a hypo $\SU(3)$-structure of type $2V_1\oplus V_{12}$. Then the function $\lambda_1\in C^{\infty} (M)$ uniquely defined by $d\alpha=\lambda_1 \omega$ is constant. In particular, any manifold which admits a hypo $\SU(3)$-structure of type $2V_1\oplus V_{12}$ with $\lambda_1\neq 0$ also admits one with $\lambda_1=1$.
\end{lemma}
In the particular case of torsion $2V_1$, we say that $(\alpha,\omega,\psi)$ has \emph{invariant intrinsic torsion}. This torsion class as well as the more general class $2V_1\oplus V_{12}$ includes some well-known structures. However, even in the invariant intrinsic torsion case, $\lambda_2$, in contrast to $\lambda_1$ need not to be constant.
\begin{remark}\label{re:EinsteinSasaki}
\begin{itemize}
\item
Even stronger, if $(\alpha,\omega,\psi)$ is of type $V_1(\lambda_2)$, it does not follow that $\lambda_2$ is constant. For a counterexample, take $M=N\times \bR$ with $(N,\omega,\psi)$ being Calabi-Yau and $f:\bR\rightarrow \bR$ a smooth function with non-constant derivative $f'$. Then $(dt,\omega,e^{if(t)}\psi)$ with $t$ being the standard coordinate on $\bR$ is a hypo $\SU(3)$-structure on $M$ of type $V_1(\lambda_2)$ with $d(e^{if(t)}\psi)=if'(t)dt\wedge \psi$, i.e. with $\lambda_2=f'(t)$ being non-constant.
\item
A Sasaki-Einstein structure on a seven-dimensional manifold $M$ can also be seen as a hypo $\SU(3)$-structure $(\alpha,\omega,\psi)\in \Omega^1 M\times \Omega^2 M\times \Omega^3 M$ fulfilling $d\alpha=-2\omega$, $d\psi=4i\alpha\wedge \psi$, cf. \cite{CF}, i.e. as a hypo $\SU(3)$-structure of type $2V_1$ with $\lambda_1=-2$, $\lambda_2=4$. Note that the the solution of the hypo flow with this initial value is given by $(\alpha(t),\omega(t),\psi(t))=((t+1)\alpha,(t+1)^2 \omega,(t+1)^3 \psi)$ and the induced metric on $M\times (-1,\infty)\cong M\times \bR_+$ is the cone metric. Recall that Sasaki-Einstein structures have positive scalar curvature and so we cannot have a left-invariant Sasaki-Einstein structure on a non-compact Lie group. As each Sasaki-Einstein structure in seven dimensions induces a nearly parallel $\G_2$-structure by Lemma \ref{le:hypotococalibrated}, i.e. a cocalibrated $\G_2$-structure with $d\varphi=\lambda \star_{\varphi}\varphi$ for some $\lambda\in \bR\setminus \{0\}$, \cite{FMKS} shows that there is also no left-invariant Sasaki-Einstein structure on a compact seven-dimensional Lie group. Thus, there do not exist Sasaki-Einstein structures on seven-dimensional Lie algebras at all.
\item
In \cite{CF}, a hypo $\SU(3)$-structure $(\alpha,\omega,\psi)\in \Omega^1 M\times \Omega^2 M\times \Omega^3 M$ is called \emph{contact hypo} if $d\alpha=-2\omega$. So contact hypo $\SU(3)$-structures are hypo $\SU(3)$-structures of type $2V_1\oplus V_{12}$ with $\lambda_1=-2$.
\end{itemize}
\end{remark}
In Lemma \ref{le:normalizelambda1} we saw that for hypo $\SU(3)$-structures of type $2V_1\oplus V_{12}$ we can normalize $\lambda_1$ to $1$ as it is constant. More generally, we can do this for arbitrary hypo $\SU(3)$-structures if the invariant intrinsic torsion part is constant.
\begin{lemma}\label{le:exludeintrinsictorsion}
Let $(\alpha,\omega,\psi)$ be a hypo $\SU(3)$-structure such that the invariant part $(\lambda_1,\lambda_2)$ of the intrinsic torsion is constant with
$\lambda_1\lambda_2\neq 0$. Then, for any $(a_1,a_2)\in \bR^2$ with $\sgn(\lambda_1\lambda_2)=\sgn(a_1a_2)$, the triple
$(\tfrac{\lambda_2}{a_2} \alpha,\tfrac{\lambda_1 \lambda_2}{a_1a_2}\omega,(\tfrac{\lambda_1 \lambda_2}{a_1a_2})^{\tfrac{3}{2}}\psi)$ is a hypo $\SU(3)$-structure with invariant intrinsic torsion equal to $(a_1,a_2)$.
In particular, there is no hypo $\SU(3)$-structure with invariant intrinsic torsion with $\lambda_1\lambda_2<0$ on a seven-dimensional Lie algebra $\mfg$.
\end{lemma}
\begin{proof}
The first statement is a direct computation. But then the existence of a hypo $\SU(3)$-structure on $\mfg$ with invariant intrinsic torsion $(\lambda_1,\lambda_2)$
with $\lambda_1\lambda_2<0$ implies that $\mfg$ admits a Sasaki-Einstein structure in contradiction to Remark \ref{re:EinsteinSasaki}.
\end{proof}
So any Lie algebra $\mfg$ admitting a hypo $\SU(3)$-structure of type $2V_1\oplus V_{12}$ with $\lambda_1\lambda_2\neq 0$ also admits a contact hypo $\SU(3)$-structure. Conti and Fino \cite{CF} classified semidirect sums of the form $\bR^4\rtimes H$ for $H$ being a three-dimensional solvable Lie algebra which admit contact hypo $\SU(3)$-structure, including some with invariant intrinsic torsion. We like to give here two others examples of contact-hypo $\SU(3)$-structures with invariant intrinsic torsion not included in \cite{CF}.
\begin{example}\label{ex:invinttors}
In both cases, we define $\mfg$ dually by the differentials $de^1,\ldots,de^7$ of a basis $e^1,\ldots,e^7$ of $\mfg^*$.
The contact hypo $\SU(3)$-structure with invariant intrinsic torsion $(\lambda_1,\lambda_2)=(-2,-4)$ is then given by
$(e^7,e^{12}+e^{34}+e^{56},e_{\bC}^1\wedge e_{\bC}^2\wedge e_{\bC}^3)$.
\begin{itemize}
\item
The first Lie algebra $\mfg$ is defined by
\begin{equation*}
\begin{split}
de^1&=2\sqrt{2}e^{12},\quad de^2=0,\quad de^3=\tfrac{2}{3}\sqrt{6}\left(e^{34}+e^{56}\right),\quad de^4=-\tfrac{2}{3}\sqrt{6}\left(e^{34}+e^{56}\right) \\
de^5&=-2 \sqrt{2} e^{16}-\tfrac{\sqrt{6}}{3}\left(e^{35}+3 e^{36}+e^{45}-3 e^{46}\right)+4 e^{67} \\
de^6&=2\sqrt{2} e^{15}+\tfrac{\sqrt{6}}{3}\left(3e^{35}- e^{36}-3e^{45}-e^{46}\right)-4 e^{57}\\
de^7&=-2\left(e^{12}+e^{34}+e^{56}\right).
\end{split}
\end{equation*}
One has $\mfg=\spa{\tfrac{\sqrt{6}}{6}\left(e_3-e_4\right)-\tfrac{1}{2}  e_7,\tfrac{1}{2}  e_6,-\tfrac{1}{2}  e_5,\tfrac{\sqrt{6}}{4}\left(e_3+e_4\right),\tfrac{1}{4}  e_7} \oplus  \linebreak\spa{\tfrac{1}{2}e_1-\tfrac{\sqrt{2}}{4}e_7,\tfrac{1}{2\sqrt{2}}e_2}\cong A_{5,37}\oplus \mfr_2$, where $A_{5,37}$ is the five-dimensional solvable Lie algebra with the same name in \cite{PSWZ} and $\mfr_2$ is the non-Abelian two-dimensional Lie algebra. So $\mfg$ is solvable but cannot be of the form $\bR^4\rtimes \mfh$ for some solvable three-dimensional Lie algebra $\mfg$ as $\mfg$ does not contain an Abelian ideal of dimension four. Hence, it is not contained in the examples of \cite{CF}.
\item
The second Lie algebra $\mfg$ is defined by
\begin{equation*}
\begin{split}
de^1& =2\sqrt{2}e^{12},\quad de^2=0,\quad de^3=2\sqrt{2} e^{34},\quad de^4=0,\\
de^5&=-2\sqrt{2}\left(e^{16}+e^{36}\right)+4 e^{67},\quad de^6=2\sqrt{2}\left(e^{15}+e^{35}\right)-4 e^{57}, \\
de^7& = -2\left(e^{12}+e^{34}+e^{56}\right).
\end{split}
\end{equation*}
$\mfg$ is not solvable but still decomposable with
$\mfg=\spa{e_2,e_1-\tfrac{\sqrt{2}}{2}e_7}\oplus \spa{e_4,e_3-\tfrac{\sqrt{2}}{2}e_7}\oplus \spa{e_5,e_6,e_7}\cong \mfr_2\oplus \mfr_2\oplus \mathfrak{sl}(2,\bR)$.
\end{itemize}
\end{example}
\section{The left-invariant hypo flow}\label{sec:leftinvhypoflow}
\subsection{Torsion classes invariant under the flow}
In this subsection, we consider several torsion classes on Lie algebras which stay invariant under the hypo flow and simplify the hypo flow equations in these cases. This simplification will turn out to be very useful for the determination of the holonomy of the Riemannian manifolds obtained by the hypo flow. We start with the broad class $2V_1\oplus V_8\oplus V_{12}$, i.e. only the $V_6$-part is zero. To show invariance of this class under the hypo flow and to simplify the flow equations in this case, we will need the following
\begin{lemma}\label{le:flowofhat}
Let $\mfg$ be a seven-dimensional Lie algebra with vector space splitting $\mfg=\spa{X}\oplus V$, $\dim(V)=6$ and $V$ being oriented,
such that $[X,V]\subseteq V$. Let $I$ be an open interval, $f:I\rightarrow \bR$ be smooth and $(\tau_t+i\hat{\tau}_t)_{t\in I}$ be
a smooth one-parameter family of complex $3$-forms on $V$ with model tensor $e^1_{\bC}\wedge e^2_{\bC}\wedge e^3_{\bC}$ fulfilling
\begin{equation*}
\tau'_t=f(t) X\hook d\hat{\tau}_t
\end{equation*}
for all $t\in I$. If, furthermore, $d(\Lambda^6 V^*)=0$ and $d\hat{\tau}_t\wedge \tau_t=0$ for all $t\in I$, then:
\begin{enumerate}
\item[(a)]
The induced family $(J_t)_{t\in I}$ of almost complex structure on $V$ fulfills
\begin{equation*}
J_t'v=-f(t)\left(J_t[X,J_t v]+[X,v]\right)
\end{equation*}
for all $v\in V$ and all $t\in I$.
\item[(b)]
For all $t\in I$ we have
\begin{equation*}
\hat{\tau}_t'=-f(t) X\hook d\tau_t.
\end{equation*}
\end{enumerate}
\end{lemma}
\begin{proof}
\begin{enumerate}
\item[(a)]
Recall from Definition \ref{def:saHs} that $\tau_t$ defines for any $t\in I$ a volume form $\phi(\tau_t)=\tfrac{1}{2}\hat{\tau}_t\wedge \tau_t$
and a complex structure $J_t$. Then we have, cf., e.g., \cite[Chapter 1]{SH}, the equalities $d\phi_{\tau_t}(\alpha)=\hat{\tau}_t\wedge \alpha$, 
$J_t v\hook \hat{\tau}_t=v\hook \tau_t$, $J_t v\hook \tau_t=-v\hook\hat{\tau}_t$, $v\hook \hat{\tau}_t\wedge \tau_t=-v\hook \tau_t\wedge \hat{\tau}_t$ and
$J_t v\hook \phi(\tau_t)=v\hook \tau_t\wedge \tau_t=v\hook \hat{\tau}_t\wedge \hat{\tau}_t$ for all $v\in V$ and all $\alpha\in \Lambda^3 V^*$.
Furthermore, by our assumptions $d\phi(\tau_t)=0$ and so
\begin{equation*}
\begin{split}
J_t'v\hook \phi(\tau_t)&=-J_t v\hook d\phi_{\tau_t}(\tau'_t)+v\hook \tau'_t\wedge \tau_t+v\hook \tau_t\wedge \tau_t'\\
&=-f(t)\Bigl(J_t v\hook \left(\hat{\tau}_t\wedge X\hook d\hat{\tau}_t\right)-(v\hook X\hook d\hat{\tau}_t)\wedge \tau_t-v\hook \tau_t\wedge X\hook d\hat{\tau}_t\Bigr)\\
&=-f(t)X\hook \Bigl(J_t v\hook  d\hat{\tau}_t \wedge \hat{\tau}_t+v\hook d\hat{\tau}_t\wedge \tau_t\Bigr)\\
&=-f(t)X\hook\Bigl(\cL_{J_t v} \hat{\tau}_t\wedge \hat{\tau}_t-d(v\hook \tau_t)\wedge \hat{\tau}_t-d\hat{\tau}_t\wedge v\hook \tau_t \Bigr)\\
&=-f(t)X\hook\bigl(\cL_{J_t v} \hat{\tau}_t\wedge \hat{\tau}_t+d(v\hook \phi(\tau_t))\bigr)\\
&=f(t)\bigl([\cL_{J_t v},X\hook]\hat{\tau}_t\wedge \hat{\tau}_t+[\cL_v,X\hook]\phi(\tau_t)\bigr)\\
&=-f(t)\left(J_t[X,J_t v]+[X,v]\right)\hook \phi(\tau_t),
\end{split}
\end{equation*}
i.e. $J_t'v=-f(t)\left(J_t[X,J_t v]+[X,v]\right)$, for all $v\in V$.
\item[(b)]
Using the result of (a), we get
\begin{equation*}
\begin{split}
\hat{\tau}_t'(v_1,v_2,v_3)&=-\bigl(\tau_t(J_t v_1,v_2,v_3)\bigr)'=-\tau'_t(J_t v_1,v_2,v_3)-\tau_t (J_t' v_1,v_2 v_3)\\
&=-f(t)\Bigl((X\hook d\hat{\tau_t})(J_t v_1,v_2,v_3)-\tau_t(J_t[X,J_t v_1]+[X,v_1],v_2,v_3)\Bigr) \\
&=f(t)\Bigl(\hat{\tau}_t(J_t v_1,[X,v_2],v_3)+\hat{\tau_t}(J_t v_1,v_2,[X,v_3])+\tau_t([X,v_1],v_2,v_3)\Bigr) \\
&=f(t)\Bigl(\tau_t(v_1,[X,v_2],v_3)+\tau_t(v_1,v_2,[X,v_3])+\tau_t([X,v_1],v_2,v_3)\Bigr)\\
&=-f(t) (X\hook d\tau_t)(v_1,v_2,v_3)
\end{split}
\end{equation*}
for any $v_1,v_2,v_3\in V$.
\end{enumerate}
\end{proof}
Next, we prove the announced invariance and simplification result for the 
hypo flow on a Lie algebra $\mfg$ with starting value in the class
$2V_1\oplus V_8\oplus V_{12}$. Thereto, note that for $\nu\in \Lambda^2 \mfg^*
$ and $\tau\in\Lambda^3 \mfg^*$ having common one-dimensional kernel $\cD_{
\omega}$
and $\nu$ having model tensor $e^{12}+e^{34}+e^{56}$ and $\tau$ having model 
tensor $\left(e^1_{\bC}\wedge e^3_{\bC}\wedge e^3_{\bC}\right)$, Definition
\ref{def:saHs} allows to define (non-zero) six-forms $\phi(\nu)$ and $\phi(
\tau)$ on $\mfg^*$ annihilating $\cD_{\omega}$. As the space of such six-
forms is one-dimensional, the quotient $\tfrac{\phi(\tau)}{\phi(\nu)}$
is a well-defined real number. Note further that for an $\SU(3)$-structure $(
\alpha,\omega,\psi)$, we have $\phi(\psi)=2\phi(\omega)$. With these 
clarifications, we can prove now
\begin{proposition}\label{pro:hypoflownotorsionV6}
Let $(\alpha_0,\omega_0,\psi_0)$ be a hypo $\SU(3)$-structure of class $2V_1
\oplus V_8\oplus V_{12}$ on a seven-dimensional Lie algebra $\mfg$.
Then the maximal solution $(\alpha(t),\omega(t),\linebreak \psi(t))_{t\in I}$ of the 
hypo flow with initial value $(\alpha_0,\omega_0,\psi_0)$ at $t=0$ is given by
\begin{equation}\label{eq:sol2V1V8V12}
(\alpha(t),\omega(t),\psi(t))=\left(x'(t) \alpha_0, \omega_0-x(t) d\alpha_0, 
\frac{\tau(t)+i\hat{\tau}(t)}{x'(t)}\right)
\end{equation}
where $x:I\rightarrow \bR$, $\tau:I\rightarrow \Lambda^3 \mfg^*$ is the 
maximal solution of the initial value problem
\begin{equation}\label{eq:simplify2V1V8V12}
\begin{split}
x'(t)=&\, \sqrt{\frac{\phi(\tau(t))}{2\phi(\omega_0-x(t)d\alpha_0)}}, \quad 
\tau'(t)=\frac{X_0\hook d\hat{\tau}(t)}{\sqrt{\frac{\phi(\tau(t))}{2\phi(
\omega_0-x(t)d\alpha_0)}}}, \\
x(0)= &\, 0,\qquad \qquad \quad \tau(0)=\rho_0.
\end{split}
\end{equation}
In particular, $(\alpha(t),\omega(t),\psi(t))$ is of class $2V_1\oplus V_8
\oplus V_{12}$ for all $t\in I$.
\end{proposition}
\begin{proof}
First of all note that Equation \eqref{eq:simplify2V1V8V12} has a unique 
maximal solution \linebreak $(x(t),\tau(t))_{t\in I}$ with $\tau(t)$
having model tensor $\re\left(e^1_{\bC}\wedge e^3_{\bC}\wedge e^3_{\bC}\right)
$ as this is an open condition by Remark \ref{re:stable}. Moreover, this 
solution stays in the subspace $V:=\{\left.\nu\in \Lambda^3\mfg^*\right|X_0
\hook \nu=0,\omega_0\wedge \nu=d\alpha_0\wedge \nu=0\}$:

Namely, let $\beta\in \{\omega_0,d\alpha_0\}$ and let $\nu\in V$ have model tensor $\re\left(e^1_{\bC}\wedge e^3_{\bC}\wedge e^3_{\bC}\right)$. Then $\beta\wedge \nu=0$ implies that $\beta$ is of type $(1,1)$ with respect to the induced almost complex structure $J_{\nu}$ on $\ker(\alpha_0)$. Hence, also $\beta\wedge \hat{\nu}=0$. But so, since $d\beta=0$ and $X_0\hook \beta=0$, we get
\begin{equation*}
\beta\wedge X_0\hook d\hat{\nu}=X_0\hook (\beta\wedge d\hat{\nu})=X_0\hook d(\beta\wedge \hat{\nu})=0.
\end{equation*}
Now $\omega'(t)=-x'(t)d\alpha_0=-d\alpha(t)$ and, setting $\rho(t):=\frac{\tau(t)}{x'(t)}$, we get $\omega(t)\wedge \rho(t)=0$ as $\tau(t)$ stays in $V$ and
\begin{equation*}
\phi(\rho(t))=\frac{\phi(\tau(t))}{(x'(t))^2}=\frac{\phi(\tau(t))}{\frac{\phi(\tau(t))}{2\phi(\omega_0-x(t)d\alpha_0)}}=2\phi(\omega_0-x(t)d\alpha_0)=2\phi(\omega(t)).
\end{equation*}
So Remark \ref{re:stable} shows that $(\omega(t),\rho(t))_{t\in I}$ is a one-parameter family of special almost Hermitian structures on $\cD_{\alpha_0}$ for all
$t\in I$. Hence, $(\alpha(t),\omega(t),\psi(t))$ is an $\SU(3)$-structure with $\cD_{\alpha(t)}=\cD_{\alpha_0}$, $\cD_{\omega(t)}=\cD_{\omega_0}$,
$d\omega(t)=0$ and $\omega'(t)=-d\alpha(t)$ for all $t\in I$. Thus, by Proposition \ref{pro:inttorsall} and the discussion directly after this proposition,
$d\hat{\tau}(t)\wedge \tau(t)=0$ for all $t\in I$. Moreover, $d(\omega_0^3)=0$, i.e. $d(\Lambda^6 \cD_{\alpha_0}^*)=0$, and $[X_0,V]\subseteq V$ as
$d\alpha_0\in \Lambda^2 \cD_{\alpha_0}^*$ by assumption. So Lemma \ref{le:flowofhat} gives us that
\begin{equation*}
\hat{\tau}'(t)= -\frac{X_0\hook \tau(t)}{\sqrt{\frac{\phi(\tau(t))}{2\phi(\omega_0-x(t)d\alpha_0)}}}
\end{equation*}
for all $t\in I$. But this shows that $\tau(t)$ stays even in $U:=V\cap \{\left.\nu\in \Lambda^3\mfg^*\right|\alpha_0\wedge d\nu=\alpha_0\wedge d\hat{\nu}=0\}$ as for any $\nu\in U$ we have
\begin{equation*}
\begin{split}
\alpha_0\wedge d(X_0\hook d\hat{\nu})&=-d(\alpha_0\wedge (X_0\hook d\hat{\nu}))+d\alpha_0\wedge (X_0\hook d\hat{\nu})\\
&=d\bigl(X_0\hook (\alpha_0\wedge d\hat{\nu})\bigr)-d(\alpha_0(X_0)d\hat{\nu})+X_0\hook d(\alpha_0\wedge d\hat{\nu})\\
&=0
\end{split}
\end{equation*}
and so also $\alpha_0\wedge d(X_0\hook \nu)=-\alpha_0\wedge d(X_0\hook \widehat{\widehat{\nu}})=0$. As $\tau(t)$ and $\hat{\tau}(t)$ are in $U$ for all $t\in I$,
we have $d(\alpha(t)\wedge \psi(t))=d(\alpha_0\wedge (\tau(t)+i\hat{\tau}(t)))=0$ for all $t\in I$. So $(\alpha(t),\omega(t),\psi(t))$ is hypo with
$d\alpha(t)\in \Lambda^2 \cD_{\alpha(t)}^*$, i.e. of type $2V_1\oplus V_8\oplus V_{12}$, for all $t\in I$.
Hence, $(\alpha(t),\omega(t),\psi(t))$ solves the hypo flow as
\begin{equation*}
\begin{split}
(\alpha(t)\wedge \psi(t))'&=\bigl(\alpha_0\wedge (\tau(t)+i\hat{\tau}(t))\bigr)'=-i\alpha_0\wedge X_0\hook \tfrac{d\left(\tau(t)+i\hat{\tau}(t))\right)}{x'(t)}\\
&=-i\alpha_0\wedge X_0\hook d\psi(t)=-i\alpha_0(X_0)d\psi(t)=-id\psi(t).
\end{split}
\end{equation*}
\end{proof}
Proposition \ref{pro:hypoflownotorsionV6} gives us not only the invariance of the class $2V_1\oplus V_8\oplus V_{12}$ under the hypo flow but also the invariance of certain subclasses as well as further simplifications of the hypo flow for these subclasses.
\begin{corollary}\label{co:hypoflowV1(lambda2)V12}
Let $(\alpha_0,\omega_0,\psi_0)$ be a hypo $\SU(3)$-structure on a seven-dimensional Lie algebra $\mfg$ and $(\alpha(t),\omega(t),\psi(t))_{t\in I}$ be the maximal solution of the hypo flow with initial value $(\alpha_0,\omega_0,\psi_0)$ at $t=0$. Then:
\begin{enumerate}
\item[(a)]
If $(\alpha_0,\omega_0,\psi_0)$ is of class $V_1(\lambda_2)\oplus V_{12}$, then
\begin{equation*}
(\alpha(t),\omega(t),\psi(t))=\left(\sqrt{\tfrac{\phi(\tau(t))}{2\phi(\omega_0)}}\alpha_0, \omega_0,\tfrac{\tau(t)+i\hat{\tau}(t)}{\sqrt{\frac{\phi(\tau(t))}{2\phi(\omega_0)}}}\right)
\end{equation*}
for $\tau:I\rightarrow \Lambda^3 \mfg^*$ being the maximal solution of the initial value problem
\begin{equation*}
\tau'(t)= \tfrac{X_0\hook d\hat{\tau}(t)}{\sqrt{\tfrac{\phi(\tau(t))}{2\phi(\omega_0)}}},\quad \tau(0)=\rho_0.
\end{equation*}
In particular, $(\alpha(t),\omega(t),\psi(t))$ is of class $V_1(\lambda_2)\oplus V_{12}$ for all $t\in I$.
\item[(b)]
If $(\alpha_0,\omega_0,\psi_0)$ is of class $2V_1\oplus V_8$, then
\begin{equation*}
(\alpha(t),\omega(t),\psi(t))=\left(x'(t)\alpha_0,\omega_0-x(t) d\alpha_0, \frac{y(t)}{x'(t)}\psi_0\right)
\end{equation*}
for $x,\,y:I\rightarrow \bR$ being the maximal solution of the initial value problem
\begin{equation*}
x'=\frac{y}{\sqrt{\frac{2\phi(\omega_0-xd\alpha_0)}{\phi(\psi_0)}}},\; y'=\lambda_2 \sqrt{\frac{2\phi(\omega_0-xd\alpha_0)}{\phi(\psi_0)}},\; x(0)=0,\; y(0)=1,
\end{equation*}
with $y>0$. In particular, $(\alpha(t),\omega(t),\psi(t))$ is of class $2V_1\oplus V_8$ for all $t\in I$.
\item[(c)]
If $(\alpha_0,\omega_0,\psi_0)$ has invariant intrinsic torsion with $\lambda_1\neq 0$, then
\begin{equation*}
(\alpha(t),\omega(t),\psi(t))=\left(x'(t)\alpha_0,\left(1-\lambda_1 x(t)\right)\omega_0, \frac{f(x(t))}{x'(t)}\psi_0\right),
\end{equation*}
where $x:I\rightarrow \bR$ is the maximal solution of the initial value problem
\begin{equation*}
x'=(1-\lambda_1 x)^{-\frac{3}{2}} f(x),\; x(0)=0
\end{equation*}
and $f(x):=\sqrt{-\frac{\lambda_2}{2\lambda_1}\left(1-\lambda_1 x\right)^4+\frac{\lambda_2}{2\lambda_1}+1}$ for $x\in J$, where $J$ is the maximal interval around $1$ on which the radicand is positive. In particular, $(\alpha(t),\omega(t),\psi(t))$ has invariant intrinsic torsion for all $t\in I$.
\end{enumerate}
\end{corollary}
\begin{proof}
In all cases, we only have to check that Equations \eqref{eq:sol2V1V8V12} and \eqref{eq:simplify2V1V8V12} are fulfilled. For (a), this follows directly from $d\alpha_0=0$.

For (b), note that for $\tau(t)+i\hat{\tau}(t):=y(t)\psi_0$, one has $\phi(\tau(t))=y^2(t)\phi(\tau_0)$. Thus, the first equation
in \eqref{eq:simplify2V1V8V12} is fulfilled. Moreover, $d\hat{\tau}(t)=y(t)d\hat{\rho}_0=\lambda_2 y(t) \alpha_0\wedge \rho_0=\lambda_2 \alpha_0\wedge \tau(t)$
and so $y'=\lambda_2 \sqrt{\tfrac{2\phi(\omega_0-xd\alpha_0)}{\phi(\psi_0)}}$ is equivalent to $\tau'= \tfrac{X_0\hook d\hat{\tau}}{\sqrt{\tfrac{\phi(\tau)}{2\phi(\omega_0-x d\alpha_0)}}}$. 

Now (c) follows from (b) as $d\alpha_0=\lambda_1 \omega_0$ and as $f:J\rightarrow \bR$ is the maximal solution of the initial value problem $\tfrac{dy}{dx}=\lambda_2 \tfrac{2\phi(\omega_0-xd\alpha_0)}{y \phi(\psi_0)}=\tfrac{\lambda_2(1-\lambda_1 x)^3}{y}$, $y(0)=1$. 
\end{proof}
\begin{remark}
If $(\alpha_0,\omega_0,\psi_0)$ has invariant intrinsic torsion with $\lambda_1=0$, then the solution of the hypo flow with this hypo $\SU(3)$-structure at $t=0$ as initial value is given by $(\alpha(t),\omega(t),\psi(t))=\left((1+\lambda_2t)\alpha_0,\omega_0,\psi_0\right)$.
\end{remark}
\subsection{Irreducible holonomy by the left-invariant hypo flow}\label{subsec:hol}

If the hypo flow on a seven-dimensional Lie algebra $\mfg$ is such that there exists a basis of $\mfg$ which stays orthogonal through the flow,
then the following lemma will turn out to be very helpful to show that the outcoming Riemannian manifold has irreducible holonomy.
\begin{lemma}\label{le:irreduciblehol}
Let $G$ be a seven-dimensional simply-connected Lie group with associated Lie algebra $\mfg$, $I$ an open interval, $(e_1,\ldots,e_7)$ a basis of $\mfg$ and $g$ a Riemannian metric on $G\times I$ of the form
\begin{equation*}
g=\sum_{i=1}^7 h_i(t) e^i\otimes e^i +dt^2
\end{equation*}
with smooth $h_i:I\rightarrow \bR^+$ for $i=1,\ldots, 7$. If $Hol(g)\subseteq \SU(4)$ and $h_i$ is not the square of a polynomial of degree at most 1 for all $i=1,\ldots 7$, then $Hol(g)\in\{\mathrm{Sp}(2),\SU(4)\}$.
\end{lemma}
\begin{proof}
We compute $R(\partial_t,e_i)(\partial_t)$ for all $i=1,\ldots,7$. As the Koszul formula gives us $\nabla_{\partial_t}\partial_t=0$, we have $R(\partial_t,e_i)(\partial_t)=\nabla_{\partial_t} \nabla_{e_i}\partial_t=\nabla_{\partial_t} \nabla_{\partial_t} e_i$. Now the Koszul formula gives us also $\nabla_{\partial_t} e_i=\frac{h_i'(t)}{2h_i(t)}e_i$ for all $i\in \{1,\ldots,n\}$ as $g(e_i,e_j)=g(e_i,\partial_t)=0$ for $j\neq i$. Hence,
\begin{equation*}
\begin{split}
R(\partial_t,e_i)(\partial_t)=& \nabla_{\partial_t} \nabla_{\partial_t} e_i=\nabla_{\partial_t} \left(\frac{h_i'(t)}{2 h_i(t)}e_i\right)=\left(\frac{h_i'(t)}{2h_i(t)}\right)' e_i+\left(\frac{h_i'(t)}{2h_i(t)}\right)^2 e_i\\
=&\frac{2 h_i(t) h_i''(t)- 2 (h_i'(t))^2+(h_i'(t))^2}{4 (h_i(t))^2} e_i=\frac{2 h_i(t) h_i''(t)-(h_i'(t))^2}{4 (h_i(t))^2} e_i.
\end{split}
\end{equation*}
for $i=1,\ldots, 7$. By the Theorem of Ambrose-Singer, $R(\partial_t,e_i)$ is contained in the holonomy algebra $\mathfrak{hol}(g)$ for all $i=1,\ldots,7$ and so
\begin{equation*}
\begin{split}
\dim(Hol(g))&\geq \left|\left\{i\in \{1,\ldots,7\}\left|\frac{2 h_i(t) h_i''(t)-(h_i'(t))^2}{4 (h_i(t))^2}  \neq 0 \textrm{ for some $t\in I$}\right.\right\}\right|\\
&=\left|\left\{i\in \{1,\ldots,7\}\left| h_i''(t)\neq \frac{(h_i'(t))^2}{2 h_i(t)}\textrm{ for some $t\in I$}\right.\right\}\right|\\
&\geq\left|\left\{i\in \{1,\ldots,7\}\left|h_i \neq p^2 \textrm{ for all polynomials $p$ of degree $0$ or $1$} \right.\right\}\right|\\
&=7
\end{split}
\end{equation*}
as the solutions of the ordinary differential equation $x''(t)= \tfrac{(x'(t))^2}{2\, x(t)}$ with positive $x(t)$ are of the form $x(t)=(at+b)^2$ for
appropriate $a,b\in \bR$. As we also have $Hol(g)\subseteq \SU(4)$, we must have $Hol(g)=\SU(3)$ or $Hol(g)\in \{\mathrm{Sp}(2),\SU(4)\}$.
In the first case, $\mathfrak{hol}(g)\left(\mfg\oplus \spa{\partial_t}\right)$ has to be six-dimensional but 
\begin{equation*}
\mathfrak{hol}(g)\left(\mfg\oplus \spa{\partial_t}\right)\supseteq \left\{\left.R(\partial_t,e_i)(\partial_t)\right|i=1,\ldots,7 \right\}=\mfg,
\end{equation*}
a contradiction. Thus, $Hol(g)\in \{\mathrm{Sp}(2),\SU(4)\}$. 
\end{proof}
This allows us now to prove that under some mild extra assumptions the hypo flow on Lie algebras with initial values in class $2V_1\oplus V_8$ yield Riemannian manifolds with irreducible holonomy. 
\begin{thm}\label{th:irreduciblehol}
Let $(\alpha_0,\omega_0,\psi_0)$ be a hypo $\SU(3)$-structure of class $2V_1\oplus V_8$ on a seven-dimensional Lie algebra $\mfg$ with $(d\alpha_0)^3\neq 0$. Then the hypo flow on $\mfg$ with initial value $(\alpha_0,\omega_0,\psi_0)$ yields a Riemannian manifold $(G\times I,g)$ with $Hol(g)\in \{\mathrm{Sp}(2),\SU(4)\}$, where $G$ denotes the associated simply-connected Lie group.
\end{thm}
\begin{proof}
By Corollary \ref{co:hypoflowV1(lambda2)V12} (b), the maximal solution $(\alpha(t),\omega(t),\psi(t))_{t\in I}$ of the hypo flow with initial value $(\alpha_0,\omega_0,\psi_0)$ at $t=0$ is given by
\begin{equation*}
(\alpha(t),\omega(t),\psi(t))=\left(x'(t)\alpha_0,\omega_0-x(t) d\alpha_0, \frac{y(t)}{x'(t)}\psi_0\right),
\end{equation*}
where $x,\,y:I\rightarrow \bR$ is the maximal solution of the initial value problem
\begin{equation*}
x'=\frac{y}{\sqrt{\frac{2\phi(\omega_0-xd\alpha_0)}{\phi(\psi_0)}}},\; y'=\lambda_2 \sqrt{\frac{2\phi(\omega_0-xd\alpha_0)}{\phi(\psi_0)}},\; x(0)=0,\; y(0)=1
\end{equation*}
In particular, $J_{\psi(t)}=J_{\psi_0}$ for all $t\in I$ and so
\begin{equation*}
\begin{split}
g=&\, g_{(\omega(t),\psi(t))}+\alpha(t)\otimes \alpha(t)+dt^2=\omega(t)(J_{\psi_0}\cdot,\cdot)+(x'(t))^2 \alpha_0\otimes \alpha_0+dt^2\\
=&\,g_0-x(t) g_1+(x'(t))^2 \alpha_0\otimes \alpha_0+dt^2
\end{split}
\end{equation*}
for $g_0:=\omega_0(J_{\psi_0}\cdot,\cdot)$ and $g_1:=d\alpha_0(J_{\psi_0}\cdot,\cdot)$. Note that $g_1$ is symmetric as $d\alpha_0$ is a $(1,1)$-form and that $g_1$ is non-degenerate as $(d\alpha_0)^3\neq 0$. We choose a basis $(e_1,\ldots,e_6)$ of $D_{\alpha_0}$ such that $g_0=\sum_{i=1}^6 e^i\otimes e^i$. As $g_1$ is symmetric, we may choose this basis such that $g_1=\sum_{i=1}^6 a_i e^i\otimes e^i$ for certain $a_1,\ldots,a_6\in \bR$. The non-degeneracy of $g_1$ is then encoded in $a_i\neq 0$ for all $i=1,\ldots,6$. Moreover,
\begin{equation*}
g=\sum_{i=1}^6 (1-a_i x(t)) e^i\otimes e^i+(x'(t))^2\alpha_0\otimes \alpha_0+dt^2.
\end{equation*}
By Lemma \ref{le:irreduciblehol}, it suffices to show that $x(t)$ is not a polynomial of degree at most $2$ in $t$ to conclude the result.
So assume the contrary. Note that we may assume that $x(t)$ has degree $2$ as otherwise $1-a_i x(t)$ can only be the square of a polynomial of degree at most $1$ if $x$ is constant, which is not possible as Corollary \ref{co:hypoflowV1(lambda2)V12} gives us $x'(0)=1$. Now $p(x):=\frac{2\phi(\omega_0-xd\alpha_0)}{\phi(\psi_0)}$ is a polynomial of degree $3$ in $x$ as $(d\alpha_0)^3\neq 0$ and so $q(t):=p(x(t))$ has degree $6$. As $y(t)=x'(t)\sqrt{p(x(t))}=x'(t) \sqrt{q(t)}$, we obtain the equation
\begin{equation*}
\lambda_2 \sqrt{q(t)}=y'(t)=x''(t) \sqrt{q(t)}+\tfrac{x'(t) q'(t)}{2\sqrt{q(t))}},
\end{equation*}
i.e. $2(\lambda_2-x''(t)) q(t)=x'(t) q'(t)$. Now $x''(t)$ is constant and so $q'$ divides $q$. Thus, $q$ has only one complex root $t_0$ of degree $6$, which has to be real as $q$ is a real polynomial. As a complex polynomial maps $\bC$ surjectively onto $\bC$ and $q(t)=p(x(t))$, the real number $x_0:=x(t_0)$ is the only complex root of $p(x)$ and we have $p(x)=b(x-x_0)^3$ for some $b\in \bR$. Moreover, $d\alpha_0=\lambda_1 \omega_0+\tilde{\omega}$ for some $\lambda_1\in \bR$ and some $\tilde{\omega}\in [\Lambda^{1,1}_0]$ by Proposition \ref{pro:inttorshypo} and so
\begin{equation*}
\begin{split}
6\phi(\omega_0-xd\alpha_0)=&6\phi\left((1-\lambda_1 x)\omega_0+x\tilde{\omega}\right)\\
=&(1-\lambda_1 x)^3 \omega_0^3+3x^2(1-\lambda_1 x) \omega_0\wedge \tilde{\omega}^2+x^3 \tilde{\omega}^3\\
=&\left((1-\lambda_1 x)^3+3 C_1 x^2(1-\lambda_1 x)+C_2 x^3\right) \omega_0^3,
\end{split}
\end{equation*}
where $C_1,\,C_2\in \bR$ are defined by $\omega_0\wedge \tilde{\omega}^2=C_1 \omega_0^3$ and $\tilde{\omega}^3=C_2 \omega_0^3$. Hence,
\begin{equation*}
\begin{split}
p(x)=& (1-\lambda_1 x)^3+3 C_1 x^2(1-\lambda_1 x)+C_2 x^3\\
=&(-\lambda_1^3-3C_1 \lambda_1+C_2)x^3+3(\lambda_1^2+C_1)x^2-3\lambda_1 x+1.
\end{split}
\end{equation*}
Now the equality $p(x)=b(x-x_0)^3$ is equivalent to $-bx_0^3=1$, $bx_0^2=-\lambda_1$, $-bx_0=\lambda_1^2+C_1$ and $b=-\lambda_1^3-3C_1 \lambda_1+C_2$. From the first two equations, we get $\lambda_1\neq 0$ and so $b=-\lambda_1^3,\, x_0=\tfrac{1}{\lambda_1}$. Then the third and the fourth give us $C_1=C_2=0$. So $\omega_0\wedge \tilde{\omega}^2=0$. Denote by $\star$ the Hodge star operator on $\cD_{\alpha_0}$ induced by $g_0$ and $\phi(\omega_0)$. Computing with one explicit element and using Schur's Lemma, one obtains $\star \nu=-\nu\wedge \omega_0$ for any $\nu\in [\Lambda^{1,1}_0]$. So $0=\omega_0\wedge \tilde{\omega}^2=-\tilde{\omega}\wedge \star\tilde{\omega}=-g(\tilde{\omega},\tilde{\omega})\phi(\omega_0)$, which implies $\tilde{\omega}=0$. Then $(d\alpha_0)^3\neq 0$ gives us $\lambda_1\neq 0$. Setting $Q(t):=P(x(t))$ with $P(x):=-\frac{\lambda_2}{2\lambda_1}\left(1-\lambda_1 x\right)^4+\frac{\lambda_2}{2\lambda_1}+1$ and noting that $Q'(t)=2\lambda_2(1-\lambda_1 x(t))^3 x'(t)$, Corollary \ref{co:hypoflowV1(lambda2)V12} (c) yields
\begin{equation*}
x'(t) Q'(t)=2\lambda_2(x'(t))^2 (1-\lambda_1 x(t))^3=2\lambda_2 Q(t).
\end{equation*}
We must have $\lambda_2\neq 0$ as otherwise $Q'(t)=0$ and so $Q$ would be constant, which is not possible as $P(x)$ is of degree $4$ in $x$ and so $Q(t)$ has degree $8$ in $t$. So $Q'$ divides $Q$ and analogous as above we get $P(x)=c(x-x_1)^4$ for certain $c\in \bR\backslash \{0\}$ and $x_1\in \bR$. From the explicit formula for $P$ we get that this implies $\frac{\lambda_2}{2\lambda_1}=-1$. But there is no hypo $\SU(3)$-structure on a Lie algebra with this property by Lemma \ref{le:exludeintrinsictorsion}. This finishes the proof.
\end{proof}
\subsection{Obstructions for holonomy $\mathrm{Sp}(2)$}\label{subsec:hyperKahler}
In this subsection, we develop obstructions for the holonomy of the Riemannian manifold obtained by the left-invariant hypo flow on a Lie group $G$ being \emph{equal} to $\mathrm{Sp}(2)$. More exactly, we show that if the holonomy is equal to $\mathrm{Sp}(2)$, then the Riemannian manifold can also be obtained by a flow of so-called hypo $\mathrm{Sp}(1)$-structures on $G$ which are \emph{left-invariant} for all times $t$ and induce the corresponding hypo $\SU(3)$-structure at the time $t$. In particular, we obtain that the Lie algebra $\mfg$ has to admit a hypo $\mathrm{Sp}(1)$-structure of a certain kind, which already gives severe restrictions on $\mfg$ in some cases.

The main ingredient in the proof of the mentioned obstruction is the well-known next lemma, which tells us that if the holonomy equals $\mathrm{Sp}(2)$,
then the only parallel forms are linear combinations of wedge products of the three K\"ahler forms.

Before stating this lemma, recall that an \emph{$\mathrm{Sp}(2)$-structure} (
in the literature often called \emph{almost hyper-Hermitian structure}) on an 
eight-dimensional manifold is a quadruple $(g,J_1,J_2,J_3)$ consisting of a 
Riemannian metric $g$ and three anti-commuting orthogonal almost complex 
structures $J_1,J_2,J_3$ fulfilling the quaternionic relation $J_1 J_2=J_3$. 
Then we have three associated fundamental forms $(\omega_1,\omega_2,\omega_3)$
, $\omega_i=g(\cdot, J_i \cdot)$. The pointwise stabilizer of $(g,J_1,J_2,J_3)
$ equals $\mathrm{Sp}(2)$, which is also the pointwise stabilizer of $(\omega_
1,\omega_2,\omega_3)$. Hence, one can reconstruct $(g,J_1,J_2,J_3)$ from $(
\omega_1,\omega_2,\omega_3)$ and we also call the triple $(\omega_1,\omega_2,
\omega_3)$ an \emph{$\mathrm{Sp}(2)$-structure}. In fact, considering $\omega_
i$ as a vector bundle homomorphism $TM\rightarrow T^*M$, one has $J_i=-\omega_
{i+1}^{-1}\circ \omega_{i+2}$ for all $i=1,2,3$ and then obtains $g$ by $g=
\omega_i(J_i\cdot,\cdot)$ for any $i\in \{1,2,3\}$. Note that an $\mathrm{Sp}(
2)$-structure $(\omega_1,\omega_2,\omega_3)$ has model tensors
\begin{equation*}
\left(e^{12}+e^{34}+e^{56}+e^{78},e^{13}-e^{24}+e^{57}-e^{68},-e^{14}-e^{23}-e
^{58}-e^{67}\right)\in \Bigl(\Lambda^2 (\bR^8)^*\Bigr)^3
\end{equation*}
and one can then check that $\omega_1^4=\omega_2^4=\omega_3^4=3\omega_i^2
\wedge \omega_j^2$ for all $(i,j)\in \{1,2,3\}$, $i\neq j$.
Another well-known important feature of $\mathrm{Sp}(2)$-structures $(\omega_1
,\omega_2,\omega_3)$ is that they are parallel with respect
to the Levi-Civita connection induced by $g$ if and only if $d\omega_1=d\omega
_2=d\omega_3=0$.
Moreover, on an eight-dimensional Riemannian manifold $(M,g)$ with holonomy 
contained in $\mathrm{Sp}(2)$ there exists a parallel $\mathrm{Sp}(2)$-
structure
(also called \emph{hyperK\"ahler structure}) $(\omega_1,\omega_2,\omega_3)$
inducing the Riemannian metric $g$. With this background, we may now state 
the above mentioned lemma.
\begin{lemma}\label{le:parallelforms}
Let $(M,g)$ be an eight-dimensional Riemannian manifold with holonomy equal 
to $\mathrm{Sp}(2)$.
Then there exists a parallel $\mathrm{Sp}(2)$-structure $(\Omega_1,\Omega_2,
\Omega_3)$ and the space of parallel two-forms equals
$\spa{\Omega_1,\Omega_2,\Omega_3}$ and the space of parallel four-forms equals
$\spa{\Omega_1^2,\Omega_2^2,\Omega_3^2,\Omega_1\wedge \Omega_2,\Omega_1\wedge 
\Omega_3,\Omega_2\wedge \Omega_3}$
\end{lemma}
\begin{proof}
By the holonomy principle, inserting a point $p\in M$ provides an isomorphism 
between the space of parallel two- or four-forms and the space
 of $\mathrm{Sp}(2)$-invariant two- or four-forms on $\Lambda^2 T_p^* M\cong 
\Lambda^2 \left(\bR^8\right)^*$, $p\in M$, respectively. By \cite{Fu},
the latter spaces are spanned by the three K\"ahler forms or the wedge 
products of the three K\"ahler forms on $\bR^8$, respectively.
\end{proof}
Applying Lemma \ref{le:parallelforms} to the situation we are interested in, 
namely to an eight-dimensional manifold $M$ endowed with the action of a Lie 
group
$\G$ and with a $\G$-invariant parallel $\SU(4)$-structure $(\Omega,\Psi)$ 
for which the induced Riemannian metric $g$ has holonomy contained in $\mathrm
{Sp}(2)$, leaves us with
two problems. Firstly, how the induced parallel $\mathrm{Sp}(2)$-structure
$(\Omega_1,\Omega_2,\Omega_3)$ is related to $(\Omega,\Psi)$ and secondly, if $(\Omega_1
,\Omega_2,\Omega_3)$ is again $\G$-invariant. In general, the latter is not 
true. E.g., take a Lie algebra of the form $\mfg=\bR^6\rtimes_{\varphi} \bR^2$
 with
$\varphi:\bR^2\rightarrow \mathfrak{su}(3)\subseteq \mathfrak{gl}(6,\bR)$ 
being a Lie algebra homomorphism such that for some $v\in \bR^2$,
the endomorphism $\varphi(v)$ has rank $6$. Then $\mfg$ admits a flat Calabi-
Yau structure but no hyperK\"ahler structure by \cite[Proposition 3.2]{BDF}.
However, the next lemma shows that the situation is different when the 
holonomy is \emph{equal} to $\mathrm{Sp}(2)$.
\begin{lemma}\label{le:SU4Sp2}
Let $(M,g)$ be a Riemannian manifold endowed with an action of a connected 
Lie group $G$ and a parallel $G$-invariant $\mathrm{SU}(4)$-structure $(\Omega
,\Psi)$ inducing the Riemannian metric $g$. If the holonomy of $(M,g)$ equals 
$\mathrm{Sp}(2)$, then there exists a $G$-invariant parallel $\mathrm{Sp}(2)$-
structure $(\Omega_1,\Omega_2,\Omega_3)\in\left(\Omega^2 M\right)^3$ inducing 
the metric $g$ such that $\Omega=\Omega_1$ and $\Psi=\frac{1}{2}\left(\Omega_2
+i\Omega_3\right)^2$.
\end{lemma}
\begin{proof}
By Lemma \ref{le:parallelforms}, there is a parallel $\mathrm{Sp}(2)$-
structure $(\Omega_1,\Omega_2,\Omega_3)\in\left(\Omega^2 M\right)^3$ inducing
the metric $g$ and the space of parallel two-forms on $(M,g)$ is given by
$\spa{\Omega_1,\Omega_2,\Omega_3}$. Hence, we may assume that $\Omega=\Omega_1$.
Moreover, Lemma \ref{le:parallelforms} yields that a complex basis of the 
parallel complex four-forms is given by
$\Omega_1^2$, $ \Omega_1\wedge \Omega_{\bC}$, $\Omega_1\wedge \overline{\Omega
}_{\bC}$, $\Omega_{\bC}^2$, $\Omega_{\bC}\wedge \overline{\Omega}_{\bC}$, $
\overline{\Omega}_{\bC}^2$, where $\Omega_{\bC}:=\Omega_2+i\Omega_3$. Now $
\Omega_{\bC}$ is a $(2,0)$-form, $\Omega_1$ is a $(1,1)$-form and $\Psi$ is a 
$(4,0)$-form with respect to $J_1$ for $J_1,J_2,J_3$ being the associated 
complex structures. Hence, $\Psi=\lambda \Omega_{\bC}^2$ for some constant $\lambda\in \bC$.
The normalization condition \eqref{eq:normalization} reads here
\begin{equation*}
\tfrac{2\left|\lambda\right|^2}{3}\Omega_1^4=\tfrac{\left|\lambda\right|^2}{4}
\left(\Omega_2^4+\Omega_3^4+2\Omega_2^2\wedge \Omega_3^2\right)=\tfrac{1}{4}
\Psi\wedge \overline{\Psi}=\phi(\Psi)=4\phi(\Omega)=\tfrac{1}{6}\Omega_1^4
\end{equation*}
and so gives us $\left|\lambda\right|=\frac{1}{2}$. After a rotation in $\spa{
\Omega_2,\Omega_3}$, we may assume that $\lambda=\tfrac{1}{2}$. Now $\Psi$ is 
$\G$-invariant and so one obtains
\begin{equation*}
0=\cL_X \Psi=\cL_X \Omega_{\bC}\wedge \Omega_{\bC}
\end{equation*}
for all $X\in \mfg=\mathrm{Lie}(\G)$. But this implies $\cL_X \Omega_{\bC}=0$ 
as the Lefschetz map is injective in this case, cf., e.g., \cite{Huy}. Hence, 
$\cL_X \Omega_2=\cL_X \Omega_3=0$. As $\G$ is connected, $\Omega_2$ and $
\Omega_3$ are both $\G$-invariant.
\end{proof}
Next, we introduce the mentioned hypo $\mathrm{Sp}(1)$-structures.
\begin{definition}
Let $M$ be a seven-dimensional manifold. Then an \emph{$\mathrm{Sp}(1)$-
structure} on $M$ is given by
$(\alpha_1,\alpha_2,\alpha_3,\omega_1,\omega_2,\omega_3)\in (\Omega^1 M)^3
\times (\Omega^2 M)^3$ with model tensors
\begin{equation*}
\left(e^7,-e^6,-e^5,e^{12}+e^{34},e^{13}-e^{24},-e^{14}-e^{23}\right)\in \left
(\left(\bR^7\right)^*\right)^3\times \left(\Lambda^2 \left(\bR^7\right)^*
\right)^3,
\end{equation*}
and it is called \emph{hypo} if
\begin{equation*}
d(\omega_1-\alpha_2\wedge \alpha_3)=0,\quad d(\omega_2-\alpha_3\wedge \alpha_1
)=0, \quad d(\omega_3-\alpha_1\wedge \alpha_2)=0.
\end{equation*}
A hypo $\mathrm{Sp}(1)$-structure induces a Riemannian metric $g=g_{((\alpha_1
,\alpha_2,\alpha_3,\omega_1,\omega_2,\omega_3)}$ on $M$ by $g=g_{(\omega_1,
\omega_2,\omega_3)}\oplus \sum_{i=1}^3 \alpha_i\otimes \alpha_i$, where $g_{(
\omega_1,\omega_2,\omega_3)}$ is the Riemannian metric induced on the 
distribution $\cD_4:=\bigcap_{i=1}^3 \ker(\alpha_i)$ of rank four by the 
almost hyper-Hermitian structure $(\omega_1,\omega_2,\omega_3)$ on $\cD_4$.
\end{definition}
Any hypo $\mathrm{Sp}(1)$-structure induces a hypo $\SU(3)$-structure as 
follows.
\begin{lemma}\label{le:Sp1Su3}
Let $M$ be a seven-dimensional manifold and $(\alpha_1,\alpha_2,\alpha_3,
\omega_1,\omega_2,\omega_3)$ be a hypo $\mathrm{Sp}(1)$-structure on $M$. Then
\begin{equation*}
\left(\alpha,\omega,\psi\right):=\left(\alpha_1,\omega_1-\alpha_2\wedge \alpha
_3,-\omega_3\wedge \alpha_2-\omega_2\wedge \alpha_3+i\left(\omega_2\wedge 
\alpha_2-\omega_3\wedge \alpha_3\right)\right)
\end{equation*}
is a hypo $\SU(3)$-structure on $M$.
\end{lemma}
\begin{proof}
Using adapted bases, we see that $\left(\alpha,\omega,\psi\right)$ is an $\SU(
3)$-structure. Moreover, $\omega$ is closed by assumption and, as 
$\omega_2^2=\omega_3^2$ and $\omega_2\wedge \omega_3=0$, we have
\begin{equation*}
\alpha\wedge \psi=\tfrac{(\omega_3-\alpha_1\wedge \alpha_2)^2-(\omega_2-\alpha
_3\wedge \alpha_1)^2}{2}-i(\omega_2-\alpha_3\wedge \alpha_1)\wedge (\omega_3-
\alpha_1\wedge \alpha_2).
\end{equation*}
Hence, $\alpha\wedge \psi$ is closed, too.
\end{proof}
Now we can show the result announced at the beginning of the subsection, 
namely that if the Riemannian manifold obtained by the left-invariant hypo 
flow on a Lie group has holonomy equal to $\mathrm{Sp}(2)$, then it can also 
be obtained by a certain flow of left-invariant hypo $\mathrm{Sp}(1)$-
structures.
\begin{proposition}\label{pro:holonomySp2inducedflow}
Let $G$ be a simply-connected seven-dimensional Lie group with associated Lie 
algebra $\mfg$ and $(\alpha(t),\omega(t),\psi(t))_{t\in I}$ be a solution of 
the hypo flow on $\mfg$. Assume that the Riemannian manifold $(G\times I,g)$ 
obtained by this solution as holonomy equal to $\mathrm{Sp}(2)$. Then there 
exists a one-parameter family of hypo $\mathrm{Sp}(1)$-structures $(\alpha_1(t
),\alpha_2(t),\alpha_3(t),\omega_1(t),\omega_2(t),\omega_3(t))_{t\in I}$ on $
\mfg$ fulfilling the flow equations
\begin{equation}\label{eq:hypoSp1flow}
\left(\omega_i(t)-\alpha_{i+1}(t)\wedge \alpha_{i+2}(t)\right)'=-d\alpha_i(t)
\end{equation}
for $i=1,2,3$ and inducing $(\alpha(t),\omega(t),\psi(t))_{t\in I}$ in the 
sense of Lemma \ref{le:Sp1Su3}. Moreover, $(\Omega_1,\Omega_2,\Omega_3)\in 
\left(\Omega^2 (G\times I)\right)^3$ with
\begin{equation*}
\Omega_i:=\omega_i(t)-\left(\alpha_{i+1}(t)\wedge \alpha_{i+2}(t)+dt\wedge 
\alpha_i(t)\right)
\end{equation*}
for $i=1,2,3$ is a parallel $\mathrm{Sp}(2)$-structure inducing the 
Riemannian metric $g$ on $G\times I$.
\end{proposition}
\begin{proof}
The parallel $\SU(4)$-structure $(\Omega,\Psi)=\left(\omega(t)+\alpha(t)\wedge
 dt,\psi(t)\wedge (\alpha(t)-idt)\right)$ on $G\times I$ obtained
by the hypo flow is $G$-invariant under the natural $G$-action by left 
multiplication on the first factor as $(\alpha(t),\omega(t),\psi(t))_{t\in I}$
were left-invariant. As the holonomy of the induced Riemannian metric $g$ is 
equal to $\mathrm{Sp}(2)$, Lemma \ref{le:SU4Sp2} gives us the existence
of a $G$-invariant parallel $\mathrm{Sp}(2)$-structure $(\Omega_1,\Omega_2,
\Omega_3)\in\left(\Omega^2 (G\times I)\right)^3$ inducing the metric
$g$ such that $\Omega=\Omega_1$ and $\Psi=\frac{1}{2}\left(\Omega_2+i\Omega_3
\right)^2$. Denote by $J_1,J_2,J_3$ the associated complex structures and
set $\alpha_i:=-J_i^*dt\in \Omega(G\times I)$ for $i=1,2,3$. As $TG$ 
is orthogonal to $\spa{\partial_t}$, $\partial_t$ is in the kernel of all
$\alpha_i$ and so $\alpha_i=\alpha_i(t)$ is a time-dependent one-form on $G$, 
which is also left-invariant as $J_i$ is $G$-invariant.
Using the quaternionic relations, we see that restricted to $\cD:=\spa{
\partial_t, J_1 \partial_t,J_2 \partial_t, J_3 \partial_t}\subseteq T(G\times 
I)$ the two-form $\Omega_i$ is equal to $-\left(\alpha_{i+1}(t)\wedge 
\alpha_{i+2}(t)+dt\wedge \alpha_i(t)\right)$ for all $i=1,2,3$. So we may write
\begin{equation*}
\Omega_i=\omega_i(t)-\left(\alpha_{i+1}(t)\wedge \alpha_{i+2}(t)+dt\wedge 
\alpha_i(t)\right)
\end{equation*}
for certain $\omega_i(t)\in \Omega^2 G$ with kernel $\cD$. The $G$-invariance 
of $\Omega_i$ gives us the left-invariance of $\omega_i(t)$. This shows that $
(\alpha_1(t),\alpha_2(t),\alpha_3(t),\omega_1(t),\omega_2(t),\omega_3(t))$ is 
an $\mathrm{Sp}(1)$-structure on $\mfg$ for all $t\in I$. Namely, $\mathrm{Sp}
(2)$ acts transitively on the seven-sphere and so we may assume that at a 
point $(e,t)\in G\times I$ we have an adapted basis $(e_1,\ldots,e_8)$ for $(
\Omega_1,\Omega_2,\Omega_3)$ with $dt=e^8$ at this point. But at this point, $\alpha_1=-J_1^*e^8=e^7$, $\alpha_2=-J_2^*e^8=-e^6$ and $\alpha_3=-
J_3^*e^8=-e^5$, we see that $\alpha_1(t)=e^7$, $\alpha_2(t)=-e^6$ and $\alpha_3(
t)=-e^5$. Inserting this into the expression for $\Omega_i$ in terms of an 
adapted basis, we get $\omega_1(t)=e^{12}+e^{34}$, $\omega_2(t)=e^{13}-e^{24}$
 and $\omega_3(t)=-e^{12}-e^{34}$.

Moreover, writing $d_8$ for the exterior derivative on $G\times I$ and $d$ for the one on $G$, the closure of $\Omega_i$ gives us
\begin{equation*}
\begin{split}
0=d_8\Omega_i=&\,dt\wedge \left(\omega_i(t)-\alpha_{i+1}(t)\wedge \alpha_{i+2}(t)\right)'+dt\wedge d\alpha_i(t)\\
&\, + d\left(\omega_i(t)-\alpha_{i+1}(t)\wedge \alpha_{i+2}(t)\right)
\end{split}
\end{equation*}
i.e.
\begin{equation*}
d\left(\omega_i(t)-\alpha_{i+1}(t)\wedge \alpha_{i+2}(t)\right)=0,\quad \left(\omega_i(t)-\alpha_{i+1}(t)\wedge \alpha_{i+2}(t)\right)'=-\alpha_i(t).
\end{equation*}
So all $\mathrm{Sp}(1)$-structures are hypo and fulfill the flow equations 
\eqref{eq:hypoSp1flow}. Finally, writing out the equalities $\Omega=\Omega_1$ 
and $\Psi=\frac{1}{2}\left(\Omega_2+i\Omega_3\right)^2$, one checks that $(
\alpha_1(t),\alpha_2(t),\linebreak \alpha_3(t),\omega_1(t),\omega_2(t),\omega_3(t))_{t\in I}$  induces $(\alpha(t),\omega(t),\psi(t))_{t\in I}$
in the sense of Lemma \ref{le:Sp1Su3}.
\end{proof}
Next, we show that hypo $\mathrm{Sp}(1)$-structures on seven-dimensional Lie 
algebras $\mfg$ inducing hypo $\SU(3)$-structures with invariant intrinsic 
torsion or of class $V_1(\lambda_2)\oplus V_{12}$ are of a very 
restricted form and that also the Lie algebra $\mfg$ is of special type.
\begin{lemma}\label{le:specialhypoSp1structures}
Let $(\alpha_1,\alpha_2,\alpha_3,\omega_1,\omega_2,\omega_3)$ be a hypo $
\mathrm{Sp}(1)$-structure which induces a hypo $\SU(3)$-structure with 
invariant intrinsic torsion or of class $V_1(\lambda_2)\oplus V_{12}$. Then $d
\omega_i=0$ for all $i=1,2,3$, $d\alpha_1=0$ and $(d\alpha_2,d\alpha_3)^t=
\alpha_1\wedge A\cdot (\alpha_2,\alpha_3)^t$ for some $A\in \mathfrak{sl}(2,
\bR)$. Moreover, $\mfg=V_4\rtimes V_3$ with $V_4:=\bigcap_{i=1}^3\ker(\alpha_i
)$ being a four-dimensional Abelian ideal and $V_3:=\ker(\omega_1)$ being a 
three-dimensional solvable unimodular Lie subalgebra acting on $V_4$ skew-symmetrically.
\end{lemma}
\begin{proof}
We assume first, slightly more general, that $(\alpha_1,\alpha_2,\alpha_3,
\omega_1,\omega_2,\omega_3)$ is a hypo $\mathrm{Sp}(1)$-structure which 
induces a hypo $\SU(3)$-structure of class $2V_1\oplus V_{12}$. Note that we 
have a splitting $\mfg=V_4\oplus V_3$ into two subspaces $V_4$ and $V_3$ of 
dimension four and three, respectively, where $V_4$ is the subspace 
annihilated by $\alpha_1,\alpha_2,\alpha_3$ and $V_3$ is the subspace 
annihilated by $\omega_1,\omega_2,\omega_3$. Then $\mfg^*=V_4^*\oplus V_3^*$ 
and we set $\Lambda^{i,j}:=\Lambda^i V_4^*\wedge \Lambda^j V_3^*$, denote by
by $\eta_{i,k-i}$ the $\Lambda^{i,k-i}$-part 
of a $k$-form $\eta$ and denote by $\Lambda^-$ the three-dimensional subspace of $\Lambda^{2,0}$ 
consisting of those element whose wedge product with all $\omega_i$, $i=1,2,3$
, is zero.
Observe now that $d\omega_i\in \alpha_{i+1}\wedge \Lambda^2 \mfg^*+\alpha_{i+2
}\wedge \Lambda^2 \mfg^*$ for all $i=1,2,3$ by the hypo equations. Hence,
\begin{equation*}
d\omega_i=\sum_{j=1}^3 \beta_i^j\wedge \omega_j +\tau_i +\nu_i
\end{equation*}
with $\beta_{i}^j\in \spa{\alpha_{i+1},\alpha_{i+2}}$, $\tau_i\in \Lambda^-
\wedge \Lambda^{0,1}$ and $\nu_i\in \Lambda^{1,2}$. Moreover, one has $\omega_
i\wedge \omega_j=\delta_{ij}\omega_1^2$ for all $i,j=1,2,3$. Differentiating this and looking at 
the $\Lambda^{4,1}$-part, we get $\beta:=\beta_1^1=\beta_2^2=\beta_3^3$ and $
\beta_{i}^j=-\beta_j^i$ for all $i\neq j$. But then $\beta_{i}^j\in \spa{\alpha_{i+
1},\alpha_{i+2}}$ implies $\beta=0$ and $\beta_i^{i+1}\in \spa{\alpha_{i+2}}$ 
for all $i=1,2,3$. Define $b_1,b_2,b_3\in \bR$ by $\beta_1^2=b_3 \alpha_3$, $
\beta_3^1=b_2 \alpha_2$ and $\beta_2^3=b_1 \alpha_1$. Then, as $d\alpha_1=
\lambda_1 (\omega_1-\alpha_2\wedge \alpha_3)$, we get
\begin{equation*}
\begin{split}
0= d\left(\omega_2-\alpha_3\wedge \alpha_1\right) &= -b_3\alpha_3\wedge \omega
_1+b_1 \alpha_1\wedge \omega_3+\tau_2 +\nu_2 -d\alpha_3\wedge \alpha_1+\lambda_1 \alpha_3\wedge \omega_1,\\
0= d\left(\omega_3-\alpha_1\wedge \alpha_2\right)&=  b_2 \alpha_2\wedge 
\omega_1- b_1 \alpha_1\wedge \omega_2+\tau_3+\nu_3-\lambda_1 \omega_1\wedge \alpha_2+\alpha_1\wedge d\alpha_2.
\end{split}
\end{equation*}
Looking at the $\Lambda^{2,1}$-parts of these equations, we get $b_2=b_3=
\lambda_1$, $(d\alpha_2)_{2,0}=b_1\omega_2+\omega_2^-$, $
(d\alpha_3)_{2,0}=b_1 \omega_3+\omega_3^-$ for certain $\omega_2^-,\omega_3^-
\in\Lambda^-$ and $\tau_2=\alpha_1\wedge \omega_3^-$, $\tau_3=-\alpha_1
\wedge \omega_2^-$. Furthermore $(d\alpha_2)_{0,2}=c_1 \alpha_1\wedge \alpha_2
+c_2 \alpha_1\wedge \alpha_3$, $(d\alpha_3)_{0,2}=c_3\alpha_1\wedge \alpha_2+c
_4 \alpha_1\wedge \alpha_3$ for certain $c_1,c_2,c_3,c_4\in \bR$ and $\nu_2=
\alpha_1\wedge (d\alpha_3)_{1,1}$, $\nu_3=-\alpha_1\wedge (d\alpha_2)_{1,1}$. 
Moreover,
\begin{equation*}
\begin{split}
0= d\left(\omega_1-\alpha_2\wedge \alpha_3\right)_{2,1} =&\, \lambda_1\left(
\alpha_3\wedge \omega_2-\alpha_2\wedge \omega_3\right)+\tau_1- b_1 \omega_2\wedge \alpha_3-\omega_2^-\wedge \alpha_3\\
&\,+b_1 \alpha_2\wedge 
\omega_3+\alpha_2\wedge \omega_3^-.
\end{split}
\end{equation*}
So $b_1=\lambda_1$ and $\tau_1=\omega_2^-\wedge 
\alpha_3-\alpha_2\wedge \omega_3^-$. Moreover, the $\Lambda^{0,3}$-part of this equation gives us $c_4=-
c_1$. Finally, the $\Lambda^{1,2}$-part gives us $\nu_1=\alpha_3\wedge (d\alpha_2)_{1,1}-\alpha_2\wedge (d
\alpha_3)_{1,1}$. Writing $(d\alpha_i)_{1,1}=
\sum_{j=1}^3 \alpha_j\wedge \gamma_i^j$ for $i=2,3$ with $\gamma_i^j\in 
\Lambda^{1,0}$ for $(i,j)\in \{2,3\}\times \{1,2,3\}$, we obtain $\nu_1=
\alpha_1\wedge \alpha_2\wedge \gamma_3^1-\alpha_1\wedge \alpha_3\wedge \gamma_
2^1-\alpha_2\wedge \alpha_3\wedge (\gamma_3^3+\gamma_2^2)$, $\nu_2=\alpha_1
\wedge \alpha_2 \wedge \gamma_3^2+\alpha_1\wedge \alpha_3\wedge \gamma_3^3$ 
and $\nu_3=-\alpha_1\wedge \alpha_2 \wedge \gamma_2^2-\alpha_1\wedge \alpha_3
\wedge \gamma_2^3$. Now the $\Lambda^{3,2}$-part of the differentials of the equalities $\omega
_1^2=\omega_2^2=\omega_3^2$ gives us $\omega_1\wedge \nu_1=\omega_2\wedge \nu_2
=\omega_3\wedge \nu_3$ and so
\begin{equation}\label{eq:gammas}
\gamma_3^3=-\gamma_2^2,\, \gamma_3^1\wedge \omega_1=\gamma_3^2\wedge 
\omega_2=-\gamma_2^2 \wedge \omega_3,\, \gamma_2^1\wedge \omega_1=-\gamma_3^3\wedge \omega_2=
\gamma_2^3\wedge \omega_3.
\end{equation}
So let us now first discuss the case $V_1(\lambda_2)\oplus V_{12}$. Then 
Theorem \ref{th:ClassV1lambda2V12} implies that for any $\eta\in \mfg$, the two-form
$d\eta$ fulfills $(d\eta)^2\in \alpha_1\wedge \Lambda^2 \mfg^*$. Now
\begin{equation*}
d\alpha_2=\omega_2^-+\alpha_2\wedge \gamma_2^2+\alpha_3\wedge \gamma_2^3+
\alpha_1\wedge (\gamma_2^1+c_1 \alpha_2+c_2 \alpha_3),\\
\end{equation*}
and so, as $\omega_2^-$ has rank four unless it is zero, we get $\omega_2^-=0$ and
that $\gamma_2^2$ and $\gamma_2^3$ are linearly dependent. Similarly, looking at $d\alpha_3=0$, we get
$\omega_3^-=0$ and that $\gamma_3^2$ and $\gamma_3^3$ are linearly dependent.
As $\gamma_3^3=-\gamma_2^2$ by Equation \eqref{eq:gammas},
the same equation implies that all $\gamma_i^j$ have to zero. This gives us $d\omega_1=d\omega_2=d\omega_3=0$ and that $d\alpha_2
$ and $d\alpha_3$ are of the claimed form. Now $V_4$ is a four-dimensional ideal in $\mfg$ admitting the hyperK\"
ahler structure $(\omega_1,\omega_2,\omega_3)$ and so has to be Abelian, cf., 
e.g., \cite{BDF}. Note that $d\omega_1=0$ already implies that $[V_3,V_3]
\subseteq V_3$, i.e. that $V_3$ is a Lie subalgebra. Finally, $V_3$ acts skew-symmetrically on $V_4$ as it preserves $(\omega_1,\omega_2,\omega_3)$ and so 
also $g_{(\omega_1,\omega_2,\omega_3)}$. This shows the statement for this 
case.

Now we have to consider the case that the induced hypo $\SU(3)$-structure has invariant intrinsic torsion. We show that then $\lambda_1=0$ 
and so we are in the situation of the just discussed case. First of all, 
Lemma \ref{le:Sp1Su3} gives us
\begin{equation*}
\lambda_2 (\alpha_1\wedge \alpha_2\wedge \omega_2-\alpha_1\wedge \alpha_3
\wedge \omega_3)=d\omega_3\wedge \alpha_2+\omega_3\wedge d\alpha_2+d\omega_2
\wedge \alpha_3+\omega_2\wedge d\alpha_3
\end{equation*}
Looking at the $\Lambda^{2,2}$-part, we get $\tau_2=\tau_3=0$, i.e. $\omega_2
^-=\omega_3^-=0$ as well as $c_1=0$ and $c_3=-c_2=\lambda_1+\lambda_2$. 
Thus, by Equation \eqref{eq:gammas},
\begin{equation*}
\begin{split}
0&=(d^2\alpha_2)_{3,0}=d\left(\lambda_1 \omega_2-(\lambda_1+\lambda_2)\alpha_
1\wedge \alpha_3+\sum_{j=1}^3 \alpha_j\wedge \gamma_2^j\right)_{3,0}=\lambda_1
 \sum_{j=1}^3 \omega_j\wedge \gamma_2^j\\
&=3\lambda_1 \omega_1\wedge \gamma_2^1.
\end{split}
\end{equation*}
So either $\lambda_1\neq 0$ or $\gamma_2^1=0$. But in the second case, Equation \eqref
{eq:gammas} implies that all $\gamma_i^j$ are zero. Hence,
\begin{equation*}
\begin{split}
0&=d^2\alpha_2=d(\lambda_1 \omega_2-(\lambda_1+\lambda_2)\alpha_1\wedge 
\alpha_3)\\
&=-(\lambda_1+\lambda_2)d(\omega_2-\alpha_3\wedge \alpha_1)+(2\lambda_1+\lambda
_2)d\omega_2\\
&=\lambda_1(2\lambda_1+\lambda_2)(\omega_3\wedge \alpha_1-\omega_1\wedge 
\alpha_3),
\end{split}
\end{equation*}
and we get again $\lambda_1=0$ as $\lambda_2=-2\lambda_1\neq 0$ is impossible by 
Lemma \ref{le:exludeintrinsictorsion}.
\end{proof}

All these results allow us now to exclude holonomy equal to $\mathrm{Sp}(2)$ for the hypo flow with certain initial values.
\begin{thm}\label{th:hyponotequalSp1}
Let $(\alpha,\omega,\psi)$ be a hypo $\SU(3)$-structure on a seven-dimensional Lie algebra $\mfg$ with invariant intrinsic torsion or of class $V_1(\lambda_2)\oplus V_{12}$. Then the hypo flow with initial value $(\alpha,\omega,\psi)$ yields a Riemannian manifold $(G\times I,g)$ with holonomy not equal to $\mathrm{Sp}(2)$. Moreover, if $(\alpha,\omega,\psi)$ has invariant intrinsic torsion with $\lambda_1\neq 0$, then the holonomy of $(G\times I,g)$ is even equal to $\mathrm{SU}(4)$.
\end{thm}
\begin{proof}
Assume the contrary and let $(\alpha(t),\omega(t),\psi(t))_{t\in I}$ be the 
solution of the hypo flow with initial value $(\alpha,\omega,\psi)$. By 
Proposition \ref{pro:holonomySp2inducedflow}, $(\alpha(t),\omega(t),\psi(t))_{t\in I}$ is 
induced by a one-parameter family of hypo $\mathrm{Sp}(1)$-structure
$(\alpha_1(t),\alpha_2(t),\alpha_3(t),\linebreak \omega_1(t),\omega_2(t),\omega_3(t))_{t\in I}$
fulfilling the flow equations \eqref{eq:hypoSp1flow}. By Lemma \ref{le:specialhypoSp1structures},
the initial value $(\alpha_1,\alpha_2,\alpha_3,
\omega_1,\omega_2,\omega_3)$ fulfills $d\omega_1=d\omega_2=d\omega_3=0$, $d
\alpha_1=0$ and $(d\alpha_2,d\alpha_3)^t=\alpha_1\wedge A(\alpha_2,\alpha_3)^t$
for some $A\in \mathfrak{sl}(2,\bR)$ and $\mfg=V_4\rtimes V_3$ with the four-dimensional Abelian ideal
$V_4=\bigcap_{i=1}^3 \ker(\alpha_i)$ and the three-dimensional Lie subalgebra $V_3=\ker(\omega_1)$ acting 
skew-symmetrically on $V_4$. Then a solution of the hypo flow is given by
$(\alpha_1(t),\alpha_2(t),\alpha_3(t),\omega_1(t),\omega_2(t),\omega_3(t))_{t\in I}$ with $\omega_i(t)\equiv \omega_i$ for 
$i=1,2,3$ and $(\alpha_1(t)\alpha_2(t),\alpha
_3(t))_{t\in I}\in \spa{\alpha_1,\alpha_2,\alpha_3}^3$ being the solution of
$\left(\alpha_{i+1}(t)\wedge \alpha_{i+2}(t)\right)'=d\alpha_i(t)$, $i=1,2,3$ (note that we can recover
all $\alpha_i(t)$ uniquely from the two-forms $\left(\alpha_1(t)\wedge \alpha_2(t),\alpha_2(t)\wedge \alpha_3(t),\alpha_3(t)\wedge \alpha_1(t)\right)$ if
this triple is not too far away from $(\alpha_1\wedge \alpha_2,\alpha_2\wedge \alpha_3,\alpha_3\wedge \alpha_1)$).
We did not determine whether the flow equations \eqref{eq:hypoSp1flow}
with given initial value have a unique solution but we know 
that any solution of \eqref{eq:hypoSp1flow} induces a solution of the hypo 
flow, and this solution is unique. So the induced Riemannian metric on $G
\times \bR$ has to be equal to $g$, the one obtained by the hypo flow. Hence, 
choosing an orthonormal basis $e_1,\ldots,e_4$ of $V_4$ for $g_{(\omega_1,
\omega_2,\omega_3)}$, we obtain
\begin{equation*}
g=\sum_{i=1}^4 e^i\otimes e^i+\sum_{i,j=1}^3
 c_{ij}(t) \alpha_{i}\otimes \alpha_{j} +dt\otimes dt
\end{equation*}
for some symmetric time-dependent matrix $C(t)=(c_{ij
}(t))_{ij}\in \bR^{3\times 3}$. As $\mfg=V_4\rtimes V_3$, $V_4$ is Abelian 
and $V_3$ acts skew-symmetrically, we get from the Koszul formula that $\nabla
_{e_i}=0$ for all $i=1,\ldots,4$, 
$\nabla_{Y}|_{V_4}=\ad(Y)|_{V_4}$ and $\nabla_{Y}(V_3\oplus \spa{\partial_t})
\subseteq V_3\oplus \spa{\partial_t}$ for all $Y\in V_3$ as well as $\nabla_{
\partial_t}|_{V_4}=0$ and $\nabla_{\partial_t}(V_3\oplus \spa{\partial_t})
\subseteq V_3\oplus \spa{\partial_t}$. Hence,
$\nabla_{e_i}\omega_j=0$ and $\nabla_{\partial_t} \omega_j=0$ for all $(i,j)
\in \{1,\ldots, 4\}\times \{1,2,3\}$. Moreover, $\nabla_Y \omega_j=\ad(Y).
\omega_j=Y\hook d\omega_j=0$ for all $Y\in V_3$. Hence, $\omega_1,\omega_2,
\omega_3$ are parallel in contradiction with Lemma \ref{le:parallelforms}. So 
the holonomy of $g$ cannot be equal to $\mathrm{Sp}(2)$.
\end{proof}
\begin{example}\label{ex:intrtorsholSU4}
Let $G$ be a Lie group whose Lie algebra $\mfg$ is one of the two Lie 
algebras of Example \ref{ex:invinttors}. Using the basis $e^1,\ldots,e^7$ of 
left-invariant one-forms on $G$ given in Example \ref{ex:invinttors}, we 
obtain, due to Theorem \ref{th:hyponotequalSp1} and Corollary \ref{co:hypoflowV1(lambda2)V12} (c),
that the Riemannian metric $g$ on $G\times 
\left(-\tfrac{1}{2},\infty\right)$ given by
\begin{equation*}
g=\sum_{i=1}^6 (1+2x) e^i\otimes e^i+\tfrac{2-(1+2x)^4}{(1+2x)^3} e^7\otimes e
^7+\tfrac{(1+2x)^3}{2-(1+2x)^4} dx\otimes dx
\end{equation*}
has holonomy equal to $\SU(4)$.

\end{example}
\section{The left-invariant Hitchin flow}\label{sec:leftinvHitchinflow}
\subsection{Holonomy reduction for the Hitchin flow}\label{subsec:holredHitchinflow}
In this subsection, we show that the Hitchin flow on certain Lie algebras always leads to Riemannian manifolds with holonomy contained in $\SU(4)$ as the initial cocalibrated $\G_2$-structure is induced by a hypo $\SU(3)$-structure in the sense of Lemma \ref{le:hypotococalibrated}. In general, we have
\begin{lemma}\label{le:wheninduced}
Let $\varphi\in \Omega^3 M$ be a cocalibrated $\G_2$-structure on a seven-dimensional manifold $M$. Then $\varphi$ is induced by a hypo $\SU(3)$-structure $(\alpha,\omega,\psi)$ if and only if there exists a unit vector field $X\in \mathfrak{X}(M)$ such that $d(X\hook \varphi)=0$ and $d(X^b\wedge \varphi)=0$.
\end{lemma}
\begin{proof}
If $\varphi$ is induced by a hypo $\SU(3)$-structure $(\alpha,\omega,\psi)$, then take $X$ to be the Lee vector field. Then $\alpha=X^b$ and $\varphi=\omega\wedge \alpha-\hat{\rho}$. So $X\hook \varphi=\omega$ is closed as well as $X^b\wedge \varphi=-\alpha\wedge \hat{\rho}$.

Conversely, assume that a unit vector field $X$ with $d(X\hook \varphi)=0$ and $d(X^b\wedge \varphi)=0$ exists. Write $\varphi=\omega\wedge X^b-\tau$ 
and $\star_{\varphi}\varphi=\Omega+X^b\wedge \rho$ for $\omega\in \Omega^2 M$, $\tau,\, \rho\in \Omega^3 M$ and $\Omega\in \Omega^4 M$ such that all these forms
annihilate $X$. As $\G_2$ acts transitively on $S^7$, we find, for any given $p\in M$, an adapted basis $(e_1,\ldots,e_7)$ for $\varphi_p$ with
$(X^b)_p=e^7$. Then one checks that $(X^b,\omega,\rho+i\tau)$ is an $\SU(3)$-structure on $\mfg$ with adapted basis $(e_2,-e_1,e_4,-e_3,e_6,-e_5,e_7)$ in $p\in M$
and $\Omega=\tfrac{\omega^2}{2}$. Hence, $d(X\hook \varphi)=0$ and $d(X^b\wedge \varphi)=0$ give us $d\omega=0$ and $d(X^b\wedge \tau)=0$. Moreover, $\varphi$ is
cocalibrated and so $d\star_{\varphi}\varphi=d\left(\tfrac{\omega^2}{2}+X^b\wedge \rho\right)=d(X^b\wedge \rho)$. Thus, $(X^b,\omega,\rho+i\tau)$ is hypo.
\end{proof}
Our previous results imply now the following statements.
\begin{thm}\label{th:holonomyreduction}
Let $\mfg$ be a seven-dimensional Lie algebra admitting a cocalibrated $\G_2$-structure $\varphi$.
If $\dim([\mfg,\mfg])= 1$ or $\mfg$ is almost Abelian, then $\varphi$ is induced by a hypo $\SU(3)$-structure on $\mfg$ and
the Hitchin flow with initial value $\varphi$ yields a Riemannian manifold $(G\times I,g)$ with holonomy $Hol(g)$
contained in $\SU(4)$ but $Hol(g)\neq \mathrm{Sp}(2)$. In particular, the Hitchin flow on the seven-dimensional Heisenberg algebra $\mfh_7$ always yields a Riemannian manifold with holonomy equal to $\SU(4)$.
\end{thm}
\begin{proof}
Combining Lemma \ref{le:hypotococalibrated} and Lemma \ref{le:wheninduced}, 
the first statement in both cases follows if we can find a unit length $X\in\mfg$ such that $d(X\hook \varphi)=0$ and $d(X^b\wedge \varphi)=0$. For any unit length
$X\in \mfg$, we denote by $V$ the orthogonal complement of $X$ in $\mfg$. 
Then we can write $\varphi=\omega\wedge X^b-\hat{\rho}$ and $\star_{\varphi}\varphi=\tfrac{
\omega^2}{2}+X^b\wedge \rho$ for a special almost Hermitian structure $(
\omega,\rho+i\hat{\rho})$ on $V$ and  if the unit length $X$ fulfills the above equations, then
a hypo $\SU(3)$-structure on $\mfg$ inducing $\varphi$ is given by $(X^b,\omega,\rho+i\hat{\rho})$, cf. the proof 
of Lemma \ref{le:wheninduced}.

Now consider first the case $\dim([\mfg,\mfg])=1$ and take $X\in [\mfg,\mfg]$ 
of unit length. If $\nu\in \Lambda^k V^*$, i.e. $\nu$ annihilates $X$, then 
$d\nu=0$ as $X$ is a basis of the commutator. Hence, $d\omega=d(X\hook \varphi
)=0$ as well as $d\rho=d\hat{\rho}=0$. So $0=d\star_{\varphi}\varphi=dX^b
\wedge \rho$, which gives that $dX^b$ annihilates $X$, i.e. $dX^b\in \Lambda^2
 V^*$, and that $dX^b$ is of type $(1,1)$ with respect to the induced almost 
complex structure on $V$. But then $-d(X^b\wedge \varphi)=dX^b\wedge \hat{
\rho}=0$ as well. For the second statement in this case, we argue by 
contradiction and assume that $Hol(g)=\mathrm{Sp}(2)$. As the hypo flow with 
initial value $(X^b,\omega,\rho+i\hat{\rho})$ also yields the metric $g$ on $G
\times I$, Proposition \ref{pro:holonomySp2inducedflow} shows the existence 
of a hypo $\mathrm{Sp}(1)$-structure $(\alpha_1,\alpha_2,\alpha_3,\omega_1,
\omega_2,\omega_3)$ inducing $(X^b,\omega,\rho+i\hat{\rho})$ Lemma \ref{le:Sp1Su3}. In particular, we have $X^b=\alpha_1$. As $\alpha_2,\alpha_3,\omega_2,
\omega_3$ all annihilate $X$, they are all closed and so the hypo equations 
give us $0=dX^b\wedge \alpha_3$ and $0=dX^b\wedge \alpha_2$, i.e. that $dX^b=\lambda \alpha_2\wedge \alpha_3$ for some $\lambda \in \bR$. Hence, a 
solution of the flow equations \eqref{eq:hypoSp1flow} with initial value $(X^b,\alpha_2,\alpha_3,\omega_1,\omega_2,\omega_3)$ is given by $\alpha_1(t)=
(1+\tfrac{3\lambda}{2} t)^{-\tfrac{1}{3}}\,X^b$, $\alpha_j(t)=(1+\tfrac{3\lambda}{2} t)^{\tfrac{1}{3}}\,\alpha_j$ 
for $j=2,3$ and $\omega_i(t)=\omega_i$ for $i=1,2,3$. This shows that $\mfg=V_4\oplus V_3$ with $V_4:=\ker \omega_1\cong \bR^4$ and $V_3:=\bigcap_{i=1}^3 
\ker(\alpha_i)$ as Lie algebras, that this decomposition is orthogonal for 
all times $t$ and that the metric on $V_4$ does not depend on $t$. Hence, the 
Koszul formula gives us $\nabla_{V_4}=0$ and $\nabla_{V_3\oplus \spa{\partial_t}} (\mfg\oplus \spa{\partial_t})
\subseteq V_3\oplus \spa{\partial_t}$ for the Levi-Civita connection $\nabla$ 
of the Riemannian metric $g$ on $G\times I$. Thus, $\nabla \omega_i=0$ for 
all $i=1,2,3$, which contradicts Lemma \ref{le:parallelforms}. So $Hol(g)\neq 
\mathrm{Sp}(2)$ in this case. If we have $\mfg=\mfh_7$, then $(X^b,\omega,\rho
+i\hat{\rho})$ is a hypo $\SU(3)$-structure of type $2V_1\oplus V_8$ with $(dX^b)^3\neq 0$ and so $Hol(g)=\SU(4)$ by Theorem \ref{th:irreduciblehol}.

In the second case, we take $X$ of unit length being orthogonal to an Abelian 
ideal $V$ of codimension one. Then $dX^b=0$ and $d\nu= X^b\wedge f.\nu$ for all $\nu\in 
\Lambda^k V^*$, where $f:=\ad(X)|_{V}$. Thus, $d(X^b\wedge \Lambda^k 
\mfg^*)=\{0\}$ and so, automatically, $d(X^b\wedge \varphi)=0$. Moreover, $0=d
\star_{\varphi}\varphi=d\omega\wedge \omega=X^b\wedge f.\omega\wedge \omega$, 
i.e. $f.\omega\wedge \omega=0$. But, as $V$ is six-dimensional, the Lefschetz map from $\Lambda^2 \mfg^*$ to $\Lambda^4 V^*$, 
cf., e.g., \cite{Huy}. Hence, $f.\omega=0$ and so $0=X^b\wedge f.\omega=d
\omega=d(X\hook \varphi)$. As the hypo $\SU(3)$-structure inducing $\varphi$ 
constructed by Lemma \ref{le:wheninduced} is in this case of class $V_1(\lambda_2)\oplus V_{12}$, Theorem \ref{th:hyponotequalSp1} implies the second 
statement.

\end{proof}
\begin{remark}
The seven-dimensional Lie algebras with one-dimensional commutator ideal are given by $\mathfrak{r}_2\oplus \bR^5$ and $\mathfrak{h}_{2k+1}\oplus \bR^{6-2k}$ for $k=1,2,3$, where $\mathfrak{r}_2$ is the unique non-Abelian two-dimensional Lie algebra and $\mathfrak{h}_{2k+1}$ is the $(2k+1)$-dimensional Heisenberg algebra. One easily sees that $\mathfrak{r}_2\oplus \bR^5$ does not admit a cocalibrated $\G_2$-structure while the others do and that the holonomy of the Riemannian metric obtained by the Hitchin flow on $\mathfrak{h}_3\oplus \bR^4$ and $\mathfrak{h}_5\oplus \bR^2$ is reducible.
\end{remark}
Let us give an explicit example of a Riemannian metric with holonomy equal to $\SU(4)$ obtained by the Hitchin flow on $\mathfrak{h}_7$. Note that for generic initial values on $\mathfrak{h}_7$ one cannot solve the Hitchin flow explicitly.
\begin{example}\label{ex:holSU4h7}
Take a basis $e_1,\ldots,e_7$ of $\mfh_3$ such that the only non-zero Lie brackets (up to anti-symmetry) are given by $[e_{2i-1},e_{2i}]=e_7$ for $i=1,2,3$ and let $\varphi=\omega\wedge e^7+\rho$ with $\omega:=e^{12}+e^{34}+e^{56}$ and $\rho:=e^{135}-e^{146}-e^{236}-e^{245}$. The solution $(\varphi(t))_{t\in I}$ of the Hitchin flow with initial value $\varphi$ is then given by $\varphi(t)=\left(\tfrac{5}{2}t+1\right)^{-\tfrac{1}{5}}\omega\wedge e^7+\left(\tfrac{5}{2}t+1\right)^{\tfrac{3}{5}}\rho$ for $t\in \left(-\tfrac{2}{5},\infty\right)$ and it induces the Riemannian metric $g$ with holonomy equal to $\SU(4)$ on $\mathrm{H}_7\times \left(-\tfrac{2}{5},\infty\right)$ given by
\begin{equation*}
g=\left(\tfrac{5}{2}t+1\right)^{\tfrac{2}{5}} \sum_{i=1}^6 e^i\otimes e^i+\left(\tfrac{5}{2}t+1\right)^{-\tfrac{6}{5}} e^7\otimes e^7+dt^2.
\end{equation*}
\end{example}
\subsection{The diagonal Hitchin flow on almost Abelian Lie algebras}\label{subsec:diagonalHitchin}
In this final subsection, we like to give many explicit examples of Riemannian metrics obtained by the left-invariant Hitchin flow on almost Abelian Lie groups $G$ with holonomy equal to $\SU(4)$. To solve the Hitchin, or alternatively the hypo flow, explicitly and to be able to show that the holonomy of the outcoming Riemannian manifold is actually equal to $\SU(4)$, recall that Lemma \ref{le:irreduciblehol} gives us, in combination with Theorem \ref{th:holonomyreduction}, a criterion for the holonomy to be equal to $\SU(4)$ in the case that there exists a basis $(e_1,\ldots,e_7)$ of the Lie algebra $\mfg$ which stays orthogonal during the flow. So it is of interest to determine first those almost Abelian Lie algebras which possess such a basis. To simplify this investigation and to make a connection to the given initial cocalibrated $\G_2$-structure $\varphi_0$ and the structure of $\mfg$, we assume that the basis is an adapted basis for $\varphi_0$ such that $(e_1,\ldots,e_6)$ is a basis of a codimension one Abelian ideal $\mfu$. In this situation, we call $(e_1,\ldots,e_7)$ \emph{adapted for $(\varphi_0,\mfu)$}. We obtain:
\begin{lemma}\label{le:diagonalflow}
Let $\mfg$ be a seven-dimensional almost Abelian Lie algebra with codimension one Abelian ideal $\mfu$, $\varphi_0\in \Lambda^3 \mfg^*$ be a cocalibrated $\G_2$-structure on $\mfg$ and $e_7\in\mfg\backslash \mfu$ be orthogonal to $\mfu$ and of norm one. Moreover, let $(\varphi(t))_{t\in I}$ be the solution of the Hitchin flow with initial value $\varphi(0)=\varphi_0$, $J_t$ be the complex structure on $\mfu$ induced by $\varphi(t)|_{\mfu}$ and the orientation induced by $\omega_0:=\left(e_7\hook \varphi_0\right)|_{\mfu}$ and set $f:=\ad(e_7)|_{\mfu}$. Then there exists an adapted basis for $(\varphi_0,\mfu)$ which last entry $e_7$ which stays orthogonal through the Hitchin flow with initial value $\varphi_0$ if and only if
\begin{itemize}
\item[(i)]
$f\in \mathfrak{u}(\mfu,\omega_0,J_0)$ and then $J_t=J_0$ for all $t\in I$ or
\item[(ii)]
there exists an orthogonal splitting $\mfu=V_4\oplus V_2$ into $J_0$- and $f$-invariant subspaces $V_4$ and $V_2$ with $\dim(V_i)=i$ for $i=2,4$ such that $[f,J_0]|_{V_4}=0$ and that form some $v\in V_2$  we have $f(v)=a v + J_0 v$ and $f(J_0 v)=c v+ a J_0 v$ for certain $a,b,c\in\bR$ and then $J_t|_{V_4}=J_0|_{V_4}$ or
\item[(iii)]
$f(e_1)=\lambda_1 e_2+a e_3$, $f(e_2)=\mu_1 e_1+a e_4$, $f(e_3)=-ae_1+\lambda_2 e_4$, $f(e_4)=-ae_2+\mu_2 e_3$, $f(e_5)=\lambda_3 e_6$ and $f(e_6)=\mu_3 e_5$ for certain $a,\lambda_i,\mu_j\in \bR$, $i,j=1,2,3$, with $a\bigl((\lambda_1-\lambda_2)^2+(\mu_1-\mu_2)^2\bigr)=0$ for an adapted basis $(e_1,\ldots,e_7)$ for $(\varphi_0,\mfu)$. This adapted basis stays orthogonal during the Hitchin flow.
\end{itemize}
\end{lemma}
\begin{proof}
Note first of all that $\varphi_0$ is cocalibrated if and only if $f\in \mathfrak{sp}(\mfu,\omega_0)$, cf. \cite{Fr}.
So let now $(e_1,\ldots,e_7)$ be an adapted basis for $(\varphi_0,\mfu)$. By the proof of Theorem \ref{th:holonomyreduction}, $\varphi_0$ is induced by the hypo $\SU(3)$-structure $(\alpha_0,\omega_0,\psi_0)$ with $\alpha_0:=e^7$ and $\psi_0:=(e^1-ie^2)\wedge (e^3-ie^4)\wedge (e^5-ie^6)$. So by Lemma \ref{le:hypotococalibrated}, the solution $(\alpha(t),\omega(t),\psi(t))_{t\in I}$ of the hypo flow with this initial value induces the solution of the Hitchin flow $(\varphi(t))_{t\in I}$ and yields, in particular, the same metric $g=g_t+dt^2$ on $G\times I$ as the Hitchin flow. Now, combining Corollary \ref{co:hypoflowV1(lambda2)V12} and Lemma \ref{le:flowofhat}, we know that $\omega\equiv \omega_0$ and that the induced almost complex structure $J_t$ fulfills $\dot{J}_t=g(t) (J_t f J_t+f)$ for some non-zero function $g:I\rightarrow \bR$. As $\dot{g}_t=\omega_0(\dot{J}_t\cdot,\cdot)$ and as the adapted basis $e_1,\ldots,e_6$ fulfills $J_0 e_{2i-1}=-e_{2i}$ and $J_0 e_{2i}=e_{2i-1}$ for $i=1,2,3$, the condition that $(e_1,\ldots,e_6)$ stays orthogonal through the hypo flow is equivalent to $\dot{J}_t e_{2i-1}\in \spa{e_{2i}}$ and $\dot{J}_t e_{2i}\in \spa{e_{2i-1}}$ for $i=1,2,3$. Using that $J_t^2=-\id_{\mfu}$, this is, in turn, equivalent to
\begin{equation*}
J_t=\left(\begin{smallmatrix} \textnormal{\LARGE 0} & \begin{smallmatrix}  h_1(t) & & \\ & h_2(t) & \\ & & h_3(t) \end{smallmatrix} \\  \begin{smallmatrix} -\tfrac{1}{h_1(t)} & & \\ & -\tfrac{1}{h_2(t)} & \\ & & -\tfrac{1}{h_3(t)}\end{smallmatrix}   & \textnormal{\LARGE 0} \end{smallmatrix}\right)
\end{equation*}
with respect to the basis $(e_1,e_3,e_5,e_2,e_4,e_6)$ for smooth functions $h_1,h_2,h_3:I\rightarrow \bR$ with $h_i(0)=1$ for $i=1,2,3$.

If $h_1$, $h_2$, $h_3$ are all constant, then $J_t=J_0$ for all $t\in I$ and so $J_0fJ_0+f=0$, i.e. $[f,J_0]=0$. Hence, $f\in \mathfrak{u}(\mfu,\omega_0,J_0)$.

If exactly two of the $h_i$, say $h_1$ and $h_2$, are constant then $J_t|_{V_4}=J_0|_{V_4}$ for all $t\in I$ and $(J_0fJ_0+f)|_{V_4}=0$ for $V_4:=\spa{e_1,e_2,e_3,e_4}$. In general, we have $J_0f (e_{2i})=-J_0fJ_0e_{2i-1}=f(e_{2i-1})=-J_t fJ_t e_{2i-1}=h_i(t) J_tf(e_{2i})$ for all $i=1,2,3$ and all $t\in I$ up to terms in $\spa{e_{2i}}$. Inserting some $t\in I$ for which $h_3(t)\neq 1$, we get that $f(e_1),f(e_3)\in V_4$ and $f(e_5)\in V_2:=\spa{e_5,e_6}$. Similarly, we see that $f(e_2),f(e_4)\in V_4$ and $(e_6)\in V_2$. Thus, $f$ preserves the subspaces $V_2$ and $V_4$. But then $(f+J_tfJ_t)(e_5)\in \spa{e_6}$ and $(f+J_tfJ_t)(e_6)\in \spa{e_5}$ are equivalent to $f(e_5)=ae_5+be_6$ and $f(e_6)=c e_5+ae_6$ for certain $a,b,c\in \bR$ and

Finally, assume that at most one of the $h_i$ is constant. Writing $f=\left(\begin{smallmatrix} A & B \\ C & D \end{smallmatrix}\right)$ with $A,B,C,D\in \bR^{3\times 3}$ with respect to the basis $(e_1,e_3,e_5,e_2,e_4,e_6)$, the condition $f\in \mathfrak{sp}(\mfu,\omega_0)$ is equivalent to $D=-A^t$, $B^t=B$ and $C^t=C$. Moreover, 
\begin{equation*}
J_0 fJ_0+f=\begin{pmatrix} A+A^t & B+C \\ B+C & -A-A^t \end{pmatrix}.
\end{equation*}
Hence, $A^t=-A$ and $c_{ij}=-b_{ij}$ for all $i\neq j$. Similarly, for $F:=\diag(h_1,h_2,h_3)$ we obtain
\begin{equation*}
J_t fJ_t+f=\begin{pmatrix} A-FAF^{-1} & B+FCF \\ C+F^{-1}BF^{-1} &  A-F^{-1}AF\end{pmatrix}.
\end{equation*}
So the desired evolution behavior is equivalent to $a_{ij}\left(1-\tfrac{h_i}{h_j}\right)=0$ for all $i, j\in \{1,2,3\}$ and $b_{ij}(1-h_i h_j)=0$ for all $i,j$ with $i\neq j$. By rotating $(e_{2l},e_{2l-1})$ by $\tfrac{\pi}{2}$ and $(e_{2k},e_{2k-1})$ by $-\tfrac{\pi}{2}$ for appropriate $l,k\in \{1,2,3\}$, $k\neq l$, we get again an adapted basis and so we may assume that $h_i h_j\not\equiv 1$ for all $i\neq j$. But then necessarily $b_{ij}=0$ for all $i\neq j$. Write $B=\diag(\mu_1,\mu_2,\mu_3)$ and $C=\diag(\lambda_1,\lambda_2,\lambda_3)$. One computes that then $\dot{J}_t e_{2i}=\dot{h}_i(t) e_{2i-1}=g(t)(\mu_i+\lambda_i h_i^2(t)) e_{2i-1}$, i.e. $\dot{h}_i(t)=g(t)(\mu_i+\lambda_i h_i^2(t))$ for all $i=1,2,3$. 

If $h_i(t)\neq h_j(t)$ for all $i\neq j$, we get $A=0$ and are in the claimed situation. If there are $i\neq j$ with $h_i=h_j$, then $\dot{h}_i= \dot{h}_j$ and so $\lambda_i+\tfrac{\mu_i}{h_i(t)}=\lambda_j+\tfrac{\mu_j}{h_j(t)}=\lambda_j+\tfrac{\mu_j}{h_i(t)}$ for all $t\in I$. As $h_i=h_j$ cannot be constant by assumption, we must have $\lambda_i=\lambda_j$ and $\mu_i=\mu_j$. If only two $h_i$s are equal, we may permute the indices so that $h_1=h_2\neq h_3$ and then must have $a_{13}=-a_{31}=a_{23}=-a_{32}=0$ and are in the claimed situation. Finally, we consider the case $h_1=h_2=h_3$. As $A$ is skew-symmetric, there exists some $G\in \mathrm{SO}(3)$ such that
\begin{equation*}
G^{-1} A G=\begin{pmatrix} 0 & -a & 0 \\ a & 0 & 0 \\ 0 & 0 & 0 \end{pmatrix}:=B
\end{equation*}
for some $a\in \bR$. Now $\diag(G,G)\in \mathrm{SU}(3)$ and so we may assume that with respect to our adapted basis $f=\left(\begin{smallmatrix} B & \mu I_3 \\
\lambda I_3 & B \end{smallmatrix}\right)$ for some $\lambda,\mu\in \bR$. This finishes the proof.
\end{proof}
\begin{thm}\label{th:diagonalflow}
Let $\mfg$ be a seven-dimensional almost Abelian Lie algebra with codimension one Abelian ideal $\mfu$, $\varphi_0$ be a cocalibrated $\G_2$-structure on $\mfg$, $e_7\in \mfg\backslash \mfu$ be orthogonal to $\mfu$ and of norm one and set $f:=\ad(e_7)|_{\mfu}$.

Then there exists a basis of $\mfg$ adapted for $(\varphi_0,\mfu)$ with last vector $e_7$ which stays orthogonal during the Hitchin flow such that the
Riemannian manifold obtained by the Hitchin flow has holonomy equal to $\SU(4)$ if and only if there exists a basis $(e_1,\ldots,e_7)$ adapted for $(\varphi_0,\mfu)$ such that $f$ is given by
\begin{equation*}
\left(\begin{array}{cc|cc|cc}
& \mu_1 & -a & & & \\
\lambda_1 & & & -a & & \\
\hline
a & & & \mu_2 & & \\
& a & \lambda_2 & & & \\
\hline
& & & & & \mu_3 \\
& & & & \lambda_3 &
\end{array}
\right)
\end{equation*}
with respect to the basis $(e_1,\ldots,e_6)$ for certain $(a,\lambda_1,\lambda_2,\lambda_3,\mu_1,\mu_2,\mu_3)\in \bR^7$ with 
\begin{equation*}
a\bigl((\lambda_1-\lambda_2)^2+(\mu_1-\mu_2)^2\bigr)=0,\quad \lambda_i+\mu_i\neq 0
\end{equation*}
for all $i=1,2,3$.

If this is the case, $(M\times I,g)$ is isometric to $(M\times J,\tilde{g})$, where
\begin{equation*}
\begin{split}
\tilde{g}=&\sum_{i=1}^3 \left(\tfrac{1}{f_{\lambda_i,\mu_i}(x)}\,e^{2i-1}\otimes e^{2i-1}+ f_{\lambda_i,\mu_i}(x)\, e^{2i}\otimes e^{2i}\right)\\
&+\prod_{j=1}^3 \tfrac{\lambda_j+\mu_j}{\lambda_j f_{\lambda_j,\mu_j}(x)+\tfrac{\mu_j}{f_{\lambda_j,\mu_j}(x)}}\left(e^7\otimes e^7+ dx^2\right),
\end{split}
\end{equation*}
for any $(\lambda,\mu)\in \bR^2$, the function $f_{\lambda,\mu}$ is the maximal solution of the initial value problem
\begin{equation}\label{eq:IVPsimpler}
f_{\lambda,\mu}'=-(\lambda f_{\lambda,\mu}^2+\mu),\qquad f_{\lambda,\mu}(0)=1
\end{equation}
and $J=\bigcap_{i=1}^3 J_i$, where $J_i$ is the intersection of the maximal interval of existence of $f_{\lambda_i,\mu_i}$ with the interval on which $f_{\lambda_i,\mu_i}>0$, $i=1,2,3$.
\end{thm}
\begin{remark}\label{re:solIVPsimpler}
Note that for $(\lambda,\mu)\in \bR^2$, the solution $f_{\lambda,\mu}$ of the initial value problem \eqref{eq:IVPsimpler} is given by
\begin{equation*}
f_{\lambda,\mu}(y)=\begin{cases} 
                   \sqrt{\tfrac{\mu}{\lambda}}\tan\left(-\sgn(\mu)\sqrt{\lambda \mu}\, y+\arctan\left(\sqrt{\tfrac{\lambda}{\mu}}\right)\right) ,& \textrm{if } \tfrac{\lambda}{\mu}>0,\\
                   \sqrt{-\tfrac{\mu}{\lambda}}\tanh\left(-\sgn(\mu)\sqrt{-\lambda \mu}\, y+\artanh\left(\sqrt{-\tfrac{\lambda}{\mu}}\right)\right) ,& \textrm{if } -1 <\tfrac{\lambda}{\mu}<0,\\
                  \sqrt{-\tfrac{\mu}{\lambda}}\coth\left(-\sgn(\mu)\sqrt{-\lambda \mu}\, y+\arcoth\left(\sqrt{-\tfrac{\lambda}{\mu}}\right)\right) ,& \textrm{if } \tfrac{\lambda}{\mu}<-1,\\ 
                  \tfrac{1}{\lambda y+1},& \textrm{if } \mu=0,\\
                  
                  1-\mu y,& \textrm{if } \lambda=0.
               \end{cases}.
\end{equation*}
\end{remark}
\begin{proof}
Let $(\varphi(t))_{t\in I}$ be the solution of the Hitchin flow with initial value $\varphi_0$ and $g_t:=g_{\varphi(t)}$, so $g=g_t+dt^2$. We need to consider the different cases in Lemma \ref{le:diagonalflow}.

In the first two cases in Lemma \ref{le:diagonalflow}, there exists a splitting $\mfu=V\oplus U$ with $V\neq 0$ such that $f$ preserves both $V$ and $U$ and acts skew-symmetrically on $V$, such that the splitting is $g_t$-orthogonal for all $t\in I$ and such that $g_t|_{V}$ is constant. By the Koszul formula, this implies $\nabla_V=0$, $\nabla_{U\oplus \spa{\partial_t}} V=0$, $\nabla_{U\oplus \spa{e_7,\partial_t}} (U\oplus \spa{e_7, \partial_t})\subseteq U\oplus \spa{e_7, \partial_t}$ and $\nabla_{e_7}|_V=f|_{V}$. Take a non-zero form $\nu\in \Lambda^{\dim(V)} \mfg^*$ annihilating $U\oplus \spa{e_7}$. As $f$ preserves $V$ and acts skew-symmetrically on $V$, we have $f.\nu=-\tr(f|_V)\nu=0$. As $\nabla_w \nu$ for all $w\in \mfu$, $\nabla_{\partial_t} \nu=0$ and $\nabla_{e_7}\nu=f.\nu=0$, the non-zero $\dim(V)$-form $\nu$ is parallel. But so the holonomy principle shows that the holonomy reducible and cannot be equal to $\SU(4)$.

In the last case in Lemma \ref{le:diagonalflow}, assume first that $\lambda_i+\mu_i=0$ for some $i\in\{1,2,3\}$. Note that in the case $i\in \{1,2\}$, the above conditions imply that $a=0$ or $\lambda_j=-\mu_j$ for all $j=1,2$. This ensures that in all cases we have a splitting $\mfu=V\oplus U$ with $V\neq 0$ such that $f$ and $J_0$ both preserve both $V$ and $U$, $f$ acts skew-symmetrically on $V$, $[J_0,f]|_V=0$ and the splitting is $g_t$-orthogonal for all $t\in I$. But $[J_0,f]|_V=0$ shows $(J_0fJ_0+f)|_V=0$ and so that $J_t|_V=J_0|_V$ for all $t\in I$. Hence, $g_t|_V$ is constant. Arguing as above, we see that the holonomy of $(G\times I,g_t+dt^2)$ has to be reducible and so cannot be equal to $\SU(4)$.

Finally, we consider the last case in Lemma \ref{le:diagonalflow} with $\lambda_i+\mu_i\neq 0$ for all $i=1,2,3$. Let $(f_1,f_2,f_3):I\rightarrow U
:=\left\{(x_1,x_2,x_3)\in \bR^3\left|x_i>0, \tfrac{\lambda_i x_i+\tfrac{\mu_i}{x_i}}{\lambda_i+\mu_i}>0 \textrm{ for }i=1,2,3\right.\right\}$ be the maximal solution of the initial value problem
\begin{equation}\label{eq:IVP}
f'_i=-\left(\lambda_i f_i^2+\mu_i\right)\sqrt{\prod_{j=1}^3\tfrac{\lambda_j f_j+\tfrac{\mu_j}{f_j}}{\lambda_j+\mu_j}}, \qquad f_i(0)=1,\qquad i=1,2,3.
\end{equation} 
We first show that
\begin{equation*}
\left(\sqrt{f_1(t)}e_1,\tfrac{1}{\sqrt{f_1(t)}}e_2,\sqrt{f_2(t)}e_3,\tfrac{1}{\sqrt{f_2(t)}}e_4, \sqrt{f_3(t)}e_5,\tfrac{1}{\sqrt{f_3(t)}}e_6,\sqrt{\prod_{i=1}^3\tfrac{\lambda_i f_i(t)+\tfrac{\mu_i}{f_i(t)}}{\lambda_i+\mu_i}}\,e_7 \right)
\end{equation*}
is an adapted basis for $\varphi(t)$ for all $t\in I$. Note that it is not clear from the beginning that the maximal interval $\tilde{I}$ of existence of the initial value problem \eqref{eq:IVP} equals $I$.
However, if we have shown that the family of $\G_2$-structures obtained by the above adapted bases solves the Hitchin flow with initial value $\varphi_0$,
then surely $\tilde{I}\subseteq I$. Now, if $\tilde{I}$ would be smaller than $I$, then either one $f_i=g_t(e_{2i},e_{2i})$ or
$\prod_{j=1}^3\tfrac{\lambda_j f_j+\tfrac{\mu_j}{f_j}}{\lambda_j+\mu_j}=\tfrac{1}{g_t(e_7,e_7)}$ goes to zero or to infinity at the boundary points of $\tilde{I}$
which are inner points of $I$, a contradiction.
 
Let $\tilde{\varphi}(t)$ be the $\G_2$-structure for which the above basis is adapted. Then we have
\begin{equation*}
\begin{split}
\tilde{\varphi}(t)=&\sqrt{\prod_{i=1}^3\tfrac{\lambda_i+\mu_i}{\lambda_i f_i(t)+\tfrac{\mu_i}{f_i(t)}}}\left(e^{127}+e^{347}+e^{567}\right)+\left(f_1(t) f_2(t)f_3(t)\right)^{-\tfrac{1}{2}}e^{135}\\
&-\sqrt{\tfrac{f_2(t) f_3(t)}{f_1(t)}}e^{146}-\sqrt{\tfrac{f_1(t) f_3(t)}{f_2(t)}}e^{236}-\sqrt{\tfrac{f_1(t) f_2(t)}{f_3(t)}}e^{245},\\
\star_{\tilde{\varphi}(t)}\tilde{\varphi}(t)&=e^{1234}+e^{1256}+e^{3456}+\tfrac{\sqrt{\prod_{i=1}^3\tfrac{\lambda_i+\mu_i}{\lambda_i+\tfrac{\mu_i}{f_i^2(t)}}}}{f_1(t)f_2(t)}e^{1367}+\tfrac{\sqrt{\prod_{i=1}^3\tfrac{\lambda_i+\mu_i}{\lambda_i+\tfrac{\mu_i}{f_i^2(t)}}}}{f_1(t)f_3(t)}e^{1457}\\
&+\tfrac{\sqrt{\prod_{i=1}^3\tfrac{\lambda_i+\mu_i}{\lambda_i+\tfrac{\mu_i}{f_i^2(t)}}}}{f_2(t)f_3(t)}e^{2357}-\sqrt{\prod_{i=1}^3\tfrac{\lambda_i+\mu_i}{\lambda_i+\tfrac{\mu_i}{f_i^2(t)}}}e^{2467}.
\end{split}
\end{equation*}
In particular, $\tilde{\varphi}(0)=\varphi_0=\varphi(0)$ and $\tilde{\varphi}(t)$ is cocalibrated for all $t\in I$. Moreover, using that $f_1=f_2$ if $a\neq 0$, we obtain
\begin{equation*}
\begin{split}
d\tilde{\varphi}=&\tfrac{1}{\sqrt{f_1f_2f_3}}\Bigl(\left(\mu_3- \lambda_2 f_2 f_3- \lambda_1 f_1f_3\right)e^{1367}+\left(\mu_2- \lambda_3 f_2f_3- \lambda_1 f_1f_2\right)e^{1457}\\
&+\left(\mu_1- \lambda_2 f_1f_2- \lambda_3 f_1f_3\right)e^{2357}\Bigr)-\sqrt{f_1f_2f_3}\sum_{i=1}^3 \tfrac{\mu_i}{f_i}e^{2467}.
\end{split}
\end{equation*}
Now
\begin{equation*}
\begin{split}
\left. \frac{d}{dt}\sqrt{\prod_{i=1}^3\tfrac{\lambda_i+\mu_i}{\lambda_i+\tfrac{\mu_i}{f_i^2}}}\right.&=\tfrac{1}{2\sqrt{\prod_{i=1}^3\tfrac{\lambda_i+\mu_i}{\lambda_i+\tfrac{\mu_i}{f_i^2}}}}\cdot \sum_{j=1}^3\left(  \tfrac{2(\lambda_j+\mu_j)\mu_j}{\left(\lambda_i+\tfrac{\mu_i}{f_i^2}\right)^2 f_j^3} f_j'\cdot  \prod_{i=1, i\neq j }^3\tfrac{\lambda_i+\mu_i}{\lambda_i+\tfrac{\mu_i}{f_i^2}}\right)\\
&=\sqrt{\prod_{i=1}^3\tfrac{\lambda_i+\mu_i}{\lambda_i+\tfrac{\mu_i}{f_i^2}}}  \sum_{j=1}^3 \tfrac{\mu_j}{f_j(\lambda_j f_j^2+\mu_j)}f_j'=-\sqrt{\prod_{i=1}^3 f_i}\cdot \sum_{j=1}^3 \tfrac{\mu_j}{f_j},
\end{split}
\end{equation*}
which implies $d\tilde{\varphi}(t)=-\star_{\tilde{\varphi}(t)}\tilde{\varphi}(t)'$ and so $\tilde{\varphi}(t)=\varphi(t)$,
i.e. the statement that the basis given above is adapted for $\varphi(t)$.

Next, we give another description of the solutions $(f_1,f_2,f_3)$ of the initial value problem \eqref{eq:IVP}, which allows us
to conclude that $(G\times I,g)$ is isometric to $(G\times \tilde{J},\tilde{g})$ and helps us in the determination of the holonomy of $g$ below.
Thereto, let $(f_{\lambda_1,\mu_1},f_{\lambda_2,\mu_2},f_{\lambda_3,\mu_3})$ be as in the assertion and let $x$ be the maximal solution of the initial value problem
\begin{equation}\label{eq:IVPforx}
x'=\sqrt{\prod_{j=1}^3\tfrac{\lambda_j f_{\lambda_j,\mu_j}(x) +\tfrac{\mu_j}{f_{\lambda_j,\mu_j}(x)}}{\lambda_j+\mu_j}},\quad x(0)=1,
\end{equation}
where we note that the right hand side is well-defined exactly on the interval $J$ of the assertion since $0\neq f_{\lambda_j,\mu_j}'(x)=-(\lambda_j f_j^2(x)+\mu_j)$ for all $x\in J$.

Now, let $I'$ be the maximal interval of existence of the initial
value problem \eqref{eq:IVPforx}. Then $(f_{\lambda_1,\mu_1}\circ x,f_{\lambda_2,\mu_2}\circ x,f_{\lambda_3,\mu_3}\circ x)$
solves the initial value problem \eqref{eq:IVP} on $I'$, so $f_i=f_{\lambda_i,\mu_i}\circ x$ on $I'$ and $I'\subseteq I$.
Note first that $x(I')=J$: Otherwise $I'$ has to be unbounded, say in positive direction, and
$J\ni x_0=\lim_{t\rightarrow \infty} x(t)$. But then there has to be some sequence $(t_k)_k$ with $\lim_{k\rightarrow \infty} t_k=\infty$ and
\begin{equation*}
0=\lim_{k\rightarrow \infty} x'(t_k)
=\sqrt{\prod_{j=1}^3\tfrac{\lambda_j f_{\lambda_j,\mu_j}(x_0) +\tfrac{\mu_j}{f_{\lambda_j,\mu_j}(x_0)}}{\lambda_j+\mu_j}},
\end{equation*}
a contradiction. Moreover, we must have $I'=I$: Otherwise, there is some $t_0\in I\cap \partial I'$ and we must have 
$x_0:=\lim_{t\rightarrow t_0} x(t)\in \partial J$ or $x_0\in \{\infty,-\infty\}$. So, from the explicit solutions of \eqref{eq:IVPsimpler}
given in Remark \ref{re:solIVPsimpler}, we see that for some $j\in \{1,2,3\}$,
$\lambda_j f_j(t_0)+\tfrac{\mu_j}{f_j(t_0)}=\lim_{t\rightarrow t_0}\lambda_j f_j(t)+\tfrac{\mu_j}{f_j(t)}=\lim_{x\rightarrow x_0}\lambda_j f_{\lambda_j,\mu_j}(x)+\tfrac{\mu_j}{f_{\lambda_j,\mu_j}(x)}=0$ or
$f_j(t_0)=\lim_{t\rightarrow t_0} f_j(t)=\lim_{x\rightarrow x_0} f_{\lambda_j,\mu_j}(x)\in \{0,\infty\}$, a contradiction. Thus, $I=I'$ and 
$(h,t)\mapsto (h,x(t))$ is an isometry between $(G\times I,g)$ and $(G\times J,\tilde{g})$

So we are left with showing that $Hol(g)=Hol(\tilde{g})=\SU(4)$. This follows from Lemma \ref{le:irreduciblehol} and Theorem \ref{th:holonomyreduction} if
$g_t(e_k,e_k)$ is not the square of a polynomial of degree at most one for all $k=1,\ldots,7$. By symmetry,
it suffices to show this for $k=1,2,7$. Moreover, replacing $f_1$ by $\tilde{f}_1:=\tfrac{1}{f_1}$,
the new triple $(\tilde{f}_1,f_2,f_3)$ fulfills the initial value problem \eqref{eq:IVP}
for $\tilde{\lambda}_1=-\mu_1$, $\tilde{\mu}_1=-\lambda_1$ and $\tilde{\lambda}_j=\lambda_j$, $\tilde{\mu}_j=\mu_j$ for $j=2,3$.
So we only need to consider $k=2,7$.

Consider first $k=7$, i.e. we assume that $g_t(e_7,e_7)$ is the square of an affine function.
Then $h(t):=\sqrt{g_t(e_7,e_7)}=\sqrt{\prod_{i=1}^3\tfrac{\lambda_i+\mu_i}{\lambda_i f_i(t)+\tfrac{\mu_i}{f_i(t)}}}$ has to be an affine function
and so we must have $h''(t)=0$ for all $t\in I$. A direct computation yields
$-2h''(t)h(t)=\sum_{j=1}^3 \left(\lambda_j f_j(t)+\tfrac{\mu_j}{f_j(t)}\right)^2$, leading to $0=-2h''(0)=\sum_{j=1}^3 (\lambda_j+\mu_j)^2$, a contradiction.

Consider now $k=2$, i.e. assume that $f_1(t)=g_t(e_2,e_2)=(at+1)^2$ for some $a\in \bR$. Then
\begin{equation*}
2a\sqrt{f_1(t)}=2a(at+1)=f_1'(t)=-(\lambda_1 f_1^2(t)+\mu_1)\sqrt{\prod_{j=1}^3\tfrac{\lambda_j f_j+\tfrac{\mu_j}{f_j}}{\lambda_j+\mu_j}},
\end{equation*}
and so $f_1 (\lambda_1 f_1+\tfrac{\mu_1}{f_1})^3(\lambda_2 f_2+\tfrac{\mu_2}{f_2})(\lambda_3 f_3+\tfrac{\mu_3}{f_3})
=f_{\lambda_1,\mu_1} (\lambda_1 f_{\lambda_1,\mu_1}+\tfrac{\mu_1}{f_{\lambda_1,\mu_1}})^3(\lambda_2 f_{\lambda_2,\mu_2}+\tfrac{\mu_2}{f_{\lambda_2,\mu_2}})(\lambda_3 f_{\lambda_3,\mu_3}+\tfrac{\mu_3}{f_{\lambda_3,\mu_3}})\circ x$
is constant and not zero. As $x$ is injective and $f_{\lambda_i,\mu_i}$ is real-analytic, we get that 
\begin{equation*}
k:=f_{\lambda_1,\mu_1} (\lambda_1 f_{\lambda_1,\mu_1}+\tfrac{\mu_1}{f_{\lambda_1,\mu_1}})^3(\lambda_2 f_{\lambda_2,\mu_2}+\tfrac{\mu_2}{f_{\lambda_2,\mu_2}})(\lambda_3 f_{\lambda_3,\mu_3}+\tfrac{\mu_3}{f_{\lambda_3,\mu_3}})
\end{equation*}
is constant and not zero on $J$.

From the explicit solution of $\eqref{eq:IVPsimpler}$ given in Remark \ref{re:solIVPsimpler}, we see that for any $(\lambda,\mu)\in \bR^2$ with $\lambda+\mu\neq 0$ the maximal interval around $1$ on which $f_{\lambda,\mu}$ is defined and greater than zero has a finite boundary point and that at each finite boundary
point, $f_{\lambda,\mu}$ goes to zero or infinity and $\lambda f_{\lambda,\mu}+\tfrac{\mu}{f_{\lambda,\mu}}$ goes to infinity.
So $J$ has a finite boundary point $x_0$. Moreover, if $f_{\lambda_1,\mu_1}$ does not go to zero at $x_0$, it is clear that the function $k$ goes to infinity at $x_0$
and we have a contradiction. However, the same happens at $x_0$ also if $f_{\lambda_1,\mu_1}$ goes to zero as
\begin{equation*}
f_{\lambda_1,\mu_1} (\lambda_1 f_{\lambda_1,\mu_1}+\tfrac{\mu_1}{f_{\lambda_1,\mu_1}})^3
=\lambda_1^3f^4_{\lambda_1,\mu_1}+3\lambda_1^2\mu_1f^2_{\lambda_1,\mu_1}+3\lambda_1\mu_1^2+\mu_1^3 f^{-2}_{\lambda_1,\mu_1}
\end{equation*}
and as $f_{\lambda_1,0}$ does not go to zero at a finite point.
\end{proof}

\begin{example}
Take $a=\lambda_1=\lambda_2=\lambda_3=0$ and $\mu_1\geq\mu_2\geq\mu_3>0$ in Theorem \ref{th:diagonalflow}. Then $\mfg$ is two-step nilpotent and the induced Riemannian metric $g$ reads
\begin{equation*}
g=\sum_{i=1}^3 \left(\tfrac{1}{1-\mu_i x}\,e^{2i-1}\otimes e^{2i-1}+ (1-\mu_i x)\, e^{2i}\otimes e^{2i}\right)+\prod_{j=1}^3(1-\mu_j x)\left(e^7\otimes e^7+ dx^2\right).
\end{equation*}
and is defined on $G\times \left(-\infty,\tfrac{1}{\mu_3}\right)$.
Now $\mfg$ obviously admits a basis with rational structure constants and so there exists a cocompact lattice $\Gamma$ in $G$.
Hence, we can build the compact nilmanifold $\Gamma\backslash G$ and get an induced Riemannian metric with holonomy equal to $\SU(4)$ on
$\Gamma\backslash G \times \left(-\infty,\tfrac{1}{\mu_3}\right)$.
\end{example}


\begin{thebibliography}{DFISUV}
\bibitem[ALRY]{ALRY} A.\ Klemm, B.\ Lian, S.-S.\ Roan, S.-T.\ Yau, {\it Calabi-Yau four-folds for M- and F-theory compactifications}, Nucl.\ Phys.\ B\ {\bf 518} (1998), no.\ 3, 515 -- 574.
\bibitem[AK]{AK} D.\ V.\ Alekseevskii, B.\ N.\ Kimel'fel'd, {\it The structure of homogeneous Riemannian spaces
with zero Ricci curvature}, Funct.\ Anal.\ Appl.\ {\bf 9} (1975), no.\ 2, 5--11.
\bibitem[AMM]{AMM} B.\ Ammann, A.\ Moroianu, S.\ Moroianu, {\it The Cauchy Problems for Einstein Metrics and Parallel Spinors}, Commun. Math. Phys. {\bf 320} (2013), no.\ 1, 173 -- 198.
\bibitem[BDF]{BDF} M.\ L.\ Barberis, I.\ Dotti, I.; A.\ Fino {\it Hyper-K\"ahler quotients of solvable Lie groups}, J.\ Geom.\ Phys.\ {\bf 56} (2006), no.\ 4, 691 -- 711.
\bibitem[Be]{Be} M.\ Berger, {\it Sur les groupes d'holonomie des vari\'{e}t\'{e}s a connexion affine et des vari\'{e}t\'{e}s riemanniennes}, Bull.\ Soc.\ Math.\ France\ {\bf 83} (1955), 279 -- 330.
\bibitem[Br]{Br} R.\ Bryant, {\it Non-embedding and non-extension results in special holonomy}, The many facets of geometry, Oxford Univ. Press, Oxford, 2010, pp. 346 -- 367.
\bibitem[C]{C} D.\ Conti, {\it Embedding into manifolds with torsion}, Math.\ Z.\ {\bf 268} (2011), nos. 3 -- 4, 725 -- 751.
\bibitem[CF]{CF} D.\ Conti, A.\ Fino, {\it Calabi-Yau cones from contact reduction}, Ann.\ Global\ Anal.\ Geom.\ {\bf 38} (2010), no.\ 1, 93 -- 118.
\bibitem[CS]{CS} D.\ Conti, S.\ Salamon, {\it Generalized Killing spinors in dimension $5$}, Trans.\ Amer.\ Math.\ Soc.\ {\bf 359} (2007), no.\ 11, 5319 -- 5343.
\bibitem[CLSS]{CLSS} V.\ Cort\'es, T.\ Leistner, L.\ Sch\"afer, F.\
  Schulte-Hengesbach, {\it Half-flat structures and special
    holonomy}, Proc.\ Lond.\ Math.\ Soc.\ (3) {\bf 102} (2011), no.\ 1, 1 -- 24.
	\bibitem[DFISUV]{DFISUV} L.\ C.\	de Andr\'{e}s, M.\ Fern\'{a}ndez, S.\ Ivanov, J.\ A.\ Santisteban, L.\ Ugarte, D.\ Vassilev, {\it Quaternionic K\"ahler and $\Spin(7)$ metrics arising from quaternionic contact Einstein structures},  Ann.\ Mat.\ Pura\ Appl.\ (4) {\bf 193} (2014), no.\ 1, 261 --290. 
	\bibitem[FG]{FG} M.\ Fern\'{a}ndez, A.\ Gray, {\it Riemannian manifolds with structure group $\G_2$}, Ann.\ Mat.\ Pura\ Appl.\ (4) {\bf 132} (1982), no.\ 1, 19 -- 45.
	\bibitem[Fr]{Fr} M.\ Freibert, {\it Cocalibrated structures on Lie algebras with a codimension one Abelian ideal},
Ann.\ Global\ Anal.\ Geom.\ {\bf 42} (2012), no. 4, 537 -- 563.
	\bibitem[FMKS]{FMKS} Th.\ Friedrich, I.\ Kath, A.\ Moroianu, U.\ Semmelmann, {\it On nearly parallel $\G_2$-structures}, J.\ Geom.\ Phys.\ {\bf 23} (1997), nos. 3 -- 4, 259 -- 286. 
 \bibitem[Fu]{Fu} T.\ Fukami, {\it Invariant tensors under the real representation of symplectic group and their applications}, T\^{o}hoku\ Math.\ J.\ (2)\
{\bf 10} (1958), no.\ 1, 81 -- 90.
\bibitem[Hi]{Hi} N.\ Hitchin, {\it Stable forms and special metrics}, Global differential geometry: the mathematical legacy of Alfred Gray, Contemporary Mathematics {\bf 288}, American Mathematical Society, Providence, 2001, pp. 70 -- 89.
\bibitem[Huy]{Huy} D.\ Huybrechts, {\it Complex Geometry. An Introduction}, Springer-Verlag Berlin Heidelberg, 2005.
\bibitem[MC]{MC} F.\ Mart\'{i}n\ Cabrera, {\it Special almost Hermitian geometry}, J.\ Geom.\ Phys.\ {\bf 55} (2005), no.\ 4, 450 -- 470.
\bibitem[Mi]{Mi} J.\ Milnor, {\it Curvatures of left-invariant metrics on
    Lie groups}, Advances in Math.\ {\bf 21} (1976), no.\ 3, 293 -- 329.
\bibitem[PSWZ]{PSWZ} J.\ Patera, R.\ T.\ Sharp, P.\ Winternitz, H.\ Zassenhaus, {\it Invariants of real low dimension Lie algebras}, J.\ Mathematical\ Phys.\ {\bf 17} (1976), no.\ 6, 986 -- 994. 
\bibitem[R1]{R1} F.\ Reidegeld, {\it Special cohomogeneity-one metrics with $Q^{1,1,1}$ or $M^{1,1,0}$ as the principal orbit}, J.\ Geom.\ Phys.\ {\bf 60} (2010), no.\ 9, 1069 -- 1088.
\bibitem[R2]{R2} F.\ Reidegeld, {\it Exceptional holonomy and Einstein metrics constructed from Aloff-Wallach spaces}, Proc.\ Lond.\ Math.\ Soc.\ (3) {\bf 102} (2011), no.\ 6, 1127 -- 1160. 
\bibitem[SH]{SH} F.\ Schulte-Hengesbach, {\it Half-flat structures on Lie groups}, PhD-thesis, University of Hamburg, 2010.
\bibitem[V]{V} C.\ Vafa, {\it Evidence for F-Theory}, Nucl.\ Phys.\ B\ {\bf 469} (1996), no.\ 3, 403 -- 415.
\bibitem[W]{W} McKenzie\ Y.\ Wang, {\it Parallel Spinors and Parallel Forms}, Ann.\ Global\ Anal.\ Geom.\ {\bf 7} (1989), no.\ 1, 59 -- 68. 

\end{thebibliography}
\end{document}